\documentclass[10pt]{amsbook}

\usepackage{amsmath, amsthm}
\usepackage{amssymb, bm}
\usepackage{amscd}
\usepackage{latexsym}
\usepackage[latin1]{inputenc}
\usepackage[english]{babel}
\usepackage[shortlabels]{enumitem}
\usepackage{extpfeil}

\usepackage{xypic}
\usepackage{graphicx}
\usepackage{changebar,color}

\usepackage{tikz-cd}
\usepackage{url}
\usepackage{appendix}

\usepackage{mathrsfs}

\numberwithin{section}{chapter}

\newtheorem{theorem}{Theorem}[section]

\newtheorem{corollary}[theorem]{Corollary}
\newtheorem{lemma}[theorem]{Lemma}

\newtheorem{proposition}[theorem]{Proposition}

\newtheorem{definition}[theorem]{Definition}

\newtheorem{defpropositionEng}[theorem]{Definition and Proposition}

\theoremstyle{remark}
\newtheorem{remark}[theorem]{Remark}

\theoremstyle{definition}

\theoremstyle{definition}

\numberwithin{equation}{section}

\usetikzlibrary{decorations.markings}
\makeatletter
\tikzcdset{
open/.code={\tikzcdset{hook, circled};},
closed/.code={\tikzcdset{hook, slashed};},
open'/.code={\tikzcdset{hook', circled};},
closed'/.code={\tikzcdset{hook', slashed};},
circled/.code={\tikzcdset{markwith={\draw (0,0) circle (.375ex);}};},
slashed/.code={\tikzcdset{markwith={\draw[-] (-.4ex,-.4ex) -- (.4ex,.4ex);}};},
markwith/.code={
\pgfutil@ifundefined{tikz@library@decorations.markings@loaded}
{\pgfutil@packageerror{tikz-cd}{You need to say 
\string\usetikzlibrary{decorations.markings} to use arrow with markings}{}}{}
\pgfkeysalso{/tikz/postaction={/tikz/decorate,
/tikz/decoration={
markings,
mark = at position 0.5 with
{#1}}}}},
}
\makeatother

\renewcommand{\phi}{\varphi}

\def\N{\mathbb{N}}
\def\Z{\mathbb{Z}}
\def\PP{\mathbb{P}}
\def\Q{\mathbb{Q}}
\def\R{\mathbb{R}}

\def\C{\mathbb{C}}

\def\H{\mathbb{H}}
\def\V{\mathbb{V}}

\newcommand{\cF}{{\mathcal  F}}

\newcommand{\cGK}{{\mathcal  GK}}

\newcommand{\ccH}{{\mathcal  H}}

\newcommand{\cN}{{\mathcal N}}

\newcommand{\cO}{{\mathcal O}}

\newcommand{\cV}{{\mathcal  V}}

\newcommand{{\OL}}{{{\mathcal O}_L}}

\newcommand{\Proj}{{\rm Proj\,}}
\newcommand{\Spec}{{\rm Spec\, }}

\newcommand{\Res}{{\rm Res}}

\newcommand{\CH}{{\mathrm{CH}}}

\newcommand{\dega}{\widehat{\rm deg }\,}

\newcommand{\rk}{{\rm rk }}

\newcommand{\ra}{\rightarrow}
\newcommand{\lrasim}{\stackrel{\sim}{\longrightarrow}}
\newcommand{\lra}{\longrightarrow}

\newcommand{\hlra}{{\lhook\joinrel\longrightarrow}}

\newcommand{\hgt}{{\mathrm{ht}}}

\newcommand{\Cc}{{\mathaccent23 C}}
\newcommand{\Dcirc}{{\mathaccent23 D}}

\newcommand{\Id}{{\mathrm{Id}}}

\renewcommand{\epsilon}{\varepsilon}

\newcommand{\Exterior}{\mathchoice{{\textstyle\bigwedge}}
    {{\bigwedge}}
    {{\textstyle\wedge}}
    {{\scriptstyle\wedge}}}

\newcommand{\disc}{{\rm{disc}}}
\newcommand{\td}{{\rm{td}}}
\newcommand{\ch}{{\rm{ch}}}
\newcommand{\Td}{{\rm{Td}}}

\newcommand{\GK}{\mathcal{GK}}

\frontmatter

\title{Griffiths heights and pencils of hypersurfaces}
\author{Thomas Mordant}
\address{Universit\'e Paris-Saclay, Laboratoire de Math\'ematiques d'Orsay, 91405 Orsay Cedex, France}
\email{thomas.mordant@universite-paris-saclay.fr}
\date{}
\begin{document}

\begin{abstract}
The Griffiths height of a variation of Hodge structures over a projective curve is defined as the degree of its canonical line bundle, as defined by Griffiths and generalized by Peters to allow bad reduction points. It may be seen as a geometric analog of the Kato height attached to pure motives over number fields. 
In this paper, we establish various formulas expressing the Griffiths height of the middle-dimensional cohomology of a pencil of projective complex hypersurfaces in terms of characteristic classes. 

Firstly, using Steenbrink's theory and the Grothendieck-Riemann-Roch theorem, we  give an expression in terms of characteristic classes of the alternating sum of the Griffiths heights of the cohomology groups of the fibers of a pencil of projective varieties with non-singular total space, whose singular fibers are divisors with normal crossings. Using this expression, we may compute the same alternating sum of Griffiths heights associated to a pencil of projective varieties with non-singular total space and only 
non-degenerate critical points.
 
By using the weak Lefschetz theorem and the above formulas, we may express the Griffiths height of the middle-dimensional cohomology of an ample pencil of hypersurfaces in a smooth pencil in terms of characteristic classes and of the Griffiths heights associated to the ambient smooth pencil. This leads to closed formulas for the Griffiths height of the middle-dimensional cohomology of pencils of projective varieties in the following special cases: for pencils of hypersurfaces in a projective bundle, and for linear pencils of hypersurfaces in a smooth projective variety, of which Lefschetz pencils are a special instance.
 
\end{abstract}

\maketitle

\tableofcontents

\chapter*{Preface} 

The first chapter of this memoir gathers its most significant results, and discusses their logical dependence. This requires the introduction of diverse technical notions, and involves rather cumbersome formulas. Our aim in this preface is to give a light overview of these results, referring to Chapter 1 for more details and further motivations. 
 \medskip

$\textbf{0.1}$  The central object of study of this memoir is the \emph{Griffiths height} attached to variations of Hodge structures, attached to the cohomology of pencils of smooth projective varieties parametrized by  some  quasi-projective curve. 

Consider $\Cc$ a smooth quasi-projective complex curve, $C$ its compactification, and let $$f: X \lra C$$ be  a projective morphism of smooth complex varieties, that is smooth  over $\Cc$. In other words, its restriction $$f_{\mid X_{\Cc}} : X_{\Cc} := f^{-1}(\Cc) \lra \Cc$$ is a complex analytic submersion.   

The fibers $(X_x)_{x \in \Cc} := (f^{-1}(x))_{x \in \Cc}$  of this submersion are diffeomorphic smooth complex manifolds. For every non-negative integer $n$, we may attach to these manifolds their integral Betti cohomology groups $H^n(X_x, \Z)$ in degree $n$, and the Hodge filtrations $F^\bullet H^n(X_x, \C)$ on their  complexifications $$ H^n(X_x, \Z) \otimes_\Z \C \simeq H^n(X_x, \C).$$

As the point $x$ varies in $\Cc$, the finitely generated $\Z$-modules $H^n(X_x, \Z)$ constitute the fibers of a local system on $\Cc$:
$$V_{\Z} = H^n(X_\Cc/\Cc, \Z),$$
and the Hodge filtrations define  a filtration $\cF^\bullet$ by subbundles of the complex analytic vector bundle on~$\Cc$:
$$\cV = V_\Z \otimes_{\Z} \cO_{\Cc} = \mathcal{H}^n(X_\Cc/\Cc).$$
The data of the  local system $V_\Z$ and the filtration $\cF^\bullet$ on the vector bundle $\cV$ is an instance of a \emph{variation of Hodge structures} of weight $n$ over $\Cc$, denoted by:
$$\V = \H^n(X_\Cc/\Cc).$$

In \cite{Griffiths70}, Griffiths defines the \emph{canonical line bundle} of such a variation of Hodge structures --- which we will call its  \emph{Griffiths line bundle} --- as the following line bundle over $\Cc$, defined in terms of the determinant line bundles $\det \cF^i :=  \Lambda^{\max} \cF^i$:
\begin{equation}\label{def GK det}
\mathcal{\cGK}_\Cc(\V) := \bigotimes_{i=1}^n \det \cF^i \simeq \bigotimes_{r=0}^n (\det \cF^r/\cF^{r+1})^{\otimes r}.
\end{equation} 

When the variation of Hodge structures $\V$ is polarized (which is the case for the primitive middle-dimensional cohomology of the fibration $X_\Cc/\Cc$), then this line bundle inherits a canonical Hermitian metric, and Griffiths shows that the curvature of this metric is non-negative. Actually the introduction by Griffiths of the line bundle defined by \eqref{def GK det} is dictated by this positivity result.

In \cite{Peters84}, using Deligne's construction \cite{Deligne70} of extensions of a vector bundle with connection over $\Cc$ to a vector bundle with logarithmic connection over $C,$ 
Peters  extends the line bundle $\mathcal{\cGK}_\Cc(\V)$ into a line bundle on $C$. There are actually two natural choices for this extension, which are called the \emph{upper extension} $\GK_{C,+}(\V)$ and the \emph{lower extension} $\GK_{C,-}(\V)$.

 The \emph{upper and lower Griffiths heights} of the variation of Hodge structure $\V$ are by definition their degrees:
$$\hgt_{GK, +}(\V) := \deg_C \mathcal{\cGK}_{C, +}  (\V) \quad \mbox{ and } \quad \hgt_{GK, -}(\V) := \deg_C \mathcal{\cGK}_{C, -}  (\V).$$

Peters also extends Griffiths' positivity result in this setting, showing that if the variation of Hodge structures $\V$ is polarized, then its \emph{upper} Griffiths height $\hgt_{GK. +}(\V)$ is \emph{non-negative}.

Another incentive to study these heights comes from Kato's introduction in \cite{Kato14, Kato18} of the \emph{Kato heights}, attached to pure motives over number fields. The Kato heights  may be seen as the  counterparts of the Griffiths heights, in the analogy between function fields and number fields, where the field of rational functions $\C(C)$ is replaced by a number field $K$. 

The Kato heights generalize the Faltings height of abelian varieties over number fields, which was introduced by Faltings in his proof of the Tate conjecture for these abelian varieties.  Kato's definition involves a tensor product of ``determinant lines," similar to the right-hand side of \eqref{def GK det}. 
Kato establishes that these heights are invariant under isogenies, by an argument involving $p$-adic cohomology and the Weil conjectures on the action of the Frobenius on \'etale cohomology groups, as established by Deligne. 

This argument --- inspired by Faltings' proof of the invariance under isogenies of Faltings heights of abelian varieties, which already played a key role in his proof of the Tate conjecture --- dictated Kato's choice of a combination of determinant lines similar to the one in \eqref{def GK det}. Kato's rationale for this specific numerology was  independent of the analogy with Griffiths heights,  as witnessed the fact that Kato himself was not aware of this analogy when he first introduced his heights for motives in \cite{Kato14}.

As emphasized by Kato and Koshikawa \cite{Kato18, Koshikawa15}, a better understanding of the properties of the Kato heights might  have far-reaching consequences for conjectures about motives over number fields --- and accordingly can only be a long-term goal. As geometry over function fields is famously more approachable than over number fields, it is sensible to study  Griffiths heights as a preliminary to the investigation of Kato heights.

\textbf{0.2}. The main result of this memoir is the computation of the Griffiths heights of the middle-dimensional cohomology of a pencil of smooth projective hypersurfaces over a quasi-projective complex curve $\Cc,$ in the ``generic" situation when it is defined by a smooth projective variety $H$ fibered over the  compactification $C$ of $\Cc$ and when the  fibers of bad reduction of $H$  (over points in  $C - \Cc$) admit at worst  ordinary double points (or equivalently when the map $H \ra C$ has only non-degenerate critical points).

 In order to state our result, let us introduce some notation.

Let $E$ be a vector bundle of rank $N+1$ over a connected smooth projective complex curve $C$, and let
$$\pi: \PP(E):= \mathrm{Proj}\,  S^\bullet E^\vee \lra C$$
be the associated projective bundle. We shall denote by $\cO_E(-1)$ the tautological rank one subbundle of $\pi^\ast E$,  and by $\cO_E(1)$ its dual.  An horizontal hypersurface in the projective bundle $\PP(E)$ is an effective Cartier divisor $H$ in $\PP(E)$ such that the morphism
$$\pi_{\mid H} : H \lra C$$
is flat.  
Then its fibers
$$H_x := \pi_{\mid H}^{-1} (x), \quad x \in C$$
are hypersurfaces in the projective spaces $\PP(E_x) \simeq \PP^N(\C)$. Their degree $d$ is independent of $x \in C$, and defines the \emph{relative degree} of the horizontal hypersurface.

We introduce the \emph{intersection-theoretic height} of an horizontal hypersurface $H$. It is defined as the rational number:
\begin{equation}\label{hintdef pref}
\mathrm{ht}_{int}(H/C) := \int_{\PP(E)} c_1(\cO_E(1))^N \cap [H] + d N \mu(E),
\end{equation}
where $$\mu(E) := \deg_C E /\rk E = \deg (c_1(E) \cap [C]) /(N+1)$$
denotes the slope of the vector bundle $E$ over $C$.

\begin{theorem}
\label{intro GK hyp P(E) pref}
Let $C$ be a connected smooth  projective complex curve with generic point $\eta$, $E$ a vector bundle of rank~$N+1$ over $C$, and $H \subset \PP(E) $ an horizontal hypersurface of relative degree $d$, smooth over $\C.$  

If $\pi_{\mid H}$ has only a finite number of critical points, all of which are non-degenerate, then the cardinality of the set $\Sigma$ of critical points satisfies:
\begin{equation}\label{cardSigmahypproj pref}
\vert \Sigma \vert = (N+1) (d-1)^N \, \mathrm{ht}_{int}(H/C).
\end{equation}

Moreover, under the same hypothesis, the following equalities hold:
$$\mathrm{ht}_{GK,+}(\H^{N-1}(H_\eta/C_\eta))
= F_+(d,N) \, \mathrm{ht}_{int}(H/C),$$
and:
$$\mathrm{ht}_{GK,-}(\H^{N-1}(H_\eta/C_\eta))
= F_-(d,N) \,  \mathrm{ht}_{int}(H/C),$$
where $F_+(d,N)$ and $F_-(d,N)$  
are the elements of $(1/12) \Z$ given when  $N$ is odd by:
$$
F_+(d,N) := \frac{N+1}{24 d^2} \left[ (d-1)^N  (7 d^2 N - 7 d^2 - 2 d N - 2 ) + 2 (d^2 - 1) \right] ,
$$
and:
$$F_-(d,N) := \frac{N+1}{24 d^2} \left[ (d-1)^N (- 5 d^2 N + 5 d^2 - 2 d N - 2 ) + 2 (d^2 - 1) \right] ,$$
and when $N$ is even by:
$$F_+(d,N) = F_-(d,N)
:= \frac{N+1}{24 d^2} \left [ (d-1)^N  (d^2 N + 2 d^2 - 2 d N - 2) - 2 (d^2-1) \right ]. $$
\end{theorem}

Observe that, after its dimension $N$ and its relative degree $d$, the intersection-theoretic height $\mathrm{ht}_{int}(H/C)$ is arguably the most basic invariant of the hypersurface $H$ embedded in the projective bundle $\PP:= \PP(E)$ over $C$. 

Actually the second term $d N \mu(E)$ in the right-hand side of its definition \eqref{hintdef pref} makes it depend only on this projective bundle $\PP$, and not on the actual choice of the vector bundle $E$ over $C$ (which may be twisted by a line bundle without changing $\PP$). Indeed $\mathrm{ht}_{int}(H/C)$ is related to the degree of $H$ with respect to the relative dualizing sheaf $\omega_{\PP/C}$ by the following formula:
$$\mathrm{ht}_{int}(H/C) = \frac{(-1)^N}{(N+1)^N} \int_{\PP} c_1(\omega_{\PP/C})^N \cap [H].$$

It is striking that ultimately the Griffiths heights of $\H^{N-1}(H_\eta/C_\eta)$ may be expressed as 
the products of the height $\hgt_{int}(H/C)$ and of some combinatorial coefficients only depending on the degree $d$ and the absolute dimension $N$ of $H$. However we are not  aware of any \emph{a priori} reason for the validity of such a proportionality. 

For instance we could try to derive this from 
a degeneration argument, which would relate the Griffiths heights of pencils of hypersurfaces of fixed dimension, but varying relative degrees and intersection-theoretic heights. However we have not been able to complete such a line of reasoning.

The  complexity of the expressions of the coefficients $F_+(d,N)$ and $F_-(d,N)$ indeed seems to exclude the existence of a significantly simpler path to these expressions than the one we present below.

\textbf{0.3} Let us explain the main points of the proof of Theorem \ref{intro GK hyp P(E) pref}. It relies on 
 several intermediate results, valid in more general situations and of independent interest. 

A first situation we shall consider is that of a pencil of projective varieties
$$g : Y \lra C,$$
with non-singular total space $Y$,  smooth above the complement $\Cc$ of a finite subset $\Delta$ in $C$, and whose singular fibers $$(Y_x)_{x \in \Delta} := (g^{-1}(x))_{x \in \Delta}$$ are \emph{divisors with strict normal crossings}.

Then Steenbrink's theory (see \cite{Steenbrink76, Steenbrink77}) allows us to express the \emph{lower} Deligne extensions of the relative de Rham cohomology in each degree and of its Hodge filtration --- and therefore the \emph{lower} Griffiths heights --- in terms of the logarithmic cohomology, namely of the (higher) direct images by $g$ of the exterior powers of the vector bundle of relative logarithmic differentials on $Y$:
$$\omega^1_{Y/C} := \Omega^1_{Y/C} (\log Y_\Delta).$$

At this stage, the Grothendieck-Riemann-Roch theorem may be used to relate the first Chern classes of  these  direct images to the  pushforwards to $C$ of characteristic classes on $Y$ associated to $\omega^1_{Y/C}$ and to the tangent bundles  $T_Y$ and $T_C$ of $Y$ and $C$. 

It turns out that considering the following alternating sum of Griffiths heights:
$$\sum_{n=0}^{2(N-1)} (-1)^{n-1} \hgt_{GK, -}(\H^n(Y_\Cc / \Cc))$$
leads to significant simplifications in these computations of characteristic classes.
Namely we shall establish the following result:

\begin{theorem} 
\label{intro GK DNC pref} Let $C$ be a  connected smooth projective complex curve with generic point $\eta$,  let $Y$ be a connected smooth projective  complex variety of  dimension $N$, and let
$$
g : Y \lra C
$$
be a surjective morphism of complex varieties. Let us assume that there exists a finite subset $\Delta$ in $C$ such that $g$ is smooth over $C - \Delta$, and such that the divisor $Y_\Delta$ is a divisor with strict normal crossings in $Y$.

Then the following equality holds:
\begin{equation}\label{eq:intro GK DNC pref}
\sum_{n=0}^{2(N-1)} (-1)^{n-1} \hgt_{GK, -}(\H^n(Y_\eta / C_\eta))
=   \int_Y \rho_{N-1}(\omega_{Y/C}^{1\vee})  \frac{\Td([T_g])}{\Td(\omega_{Y/C}^{1\vee})},
\end{equation}
where $\omega^1_{Y/C}$ denotes the vector bundle over $Y$ of relative logarithmic $1$-forms:
$$\omega^1_{Y/C} := \Omega^1_{Y/C} (\log Y_\Delta),$$ 
where $[T_g]$ is the relative tangent class in $K$-theory:
$$[T_g] := [T_{Y/\C}] - g^\ast [T_{C/\C}] \in K^0(Y),$$
and where $\Td$ denotes the Todd class and $\rho_{N-1}$  the characteristic class defined in terms of the Chern classes by:
\begin{equation}\label{rhoChern pref}
\rho_{N-1} := c_{N-2} - \frac{N-1}{2} c_{N-1} + \frac{1}{12} c_1 c_{N-1}.
\end{equation}
\end{theorem}

At first sight cumbersome, the right-hand side of \eqref{eq:intro GK DNC pref} is rather simple, because of the simplicity of the expression \eqref{rhoChern pref} for $\rho_{N-1}$, which involves only terms of codimension at least $N-2$.  

Moreover, since the restrictions to $Y - Y_\Delta$ of the vector bundle $\omega^{1 \vee}_{Y/C}$ and the relative tangent class $[T_g]$ coincide, their characteristic classes differ by cycles localized on the singular fibers $Y_\Delta$. It is therefore possible to rewrite the right-hand term of equality \eqref{eq:intro GK DNC pref} as the sum of a global term involving only $\omega^{1 \vee}_{Y/C}$ and a local term involving the geometry of the divisor $Y_\Delta$
(see Theorem \ref{intro GK DNC with localized terms} below). 

This relies on computations of characteristic classes of vector bundles of logarithmic differential forms, which we believe to be of independent interest. They constitute relative versions of the following proposition (see Proposition \ref{comparison Omega log and Omega easy cases general}):

\begin{proposition}
Let $X$ be a smooth $n$-dimensional complex scheme, and let $E$ be a reduced divisor with strict normal crossings in $X$. Let us denote by $(E_i)_{i \in I}$ the irreducible components of $E$, and for any subset $J \subseteq I$, let us consider the smooth subscheme of codimension $\vert J \vert$ in $X$:
$$E_J := \bigcap_{i \in J} E_i,$$
and the inclusion morphism 
$$i_{E_J}: E_J \hlra X.$$

The total Chern class of the vector bundle $\Omega_X^1(\log E)$ over $X$ is given by the following equality in~$\CH^\ast(X)$:
\begin{equation}
c(\Omega_X^1(\log E)) = \sum_{J \subseteq I} i_{E_J \ast} c(\Omega^1_{E_J}).
\end{equation}

\end{proposition}

\textbf{0.4} Let us emphasize that the cancellations in the computations of characteristic classes leading to Theorem \ref{intro GK DNC pref} rely crucially on the specific numerology of the combination of determinant line bundles attached to the Hodge filtration in the definition \eqref{def GK det} of the Griffiths line bundle.

Recall that, for any $N$-dimensional smooth complex projective variety $X$,     the alternating sum of the Euler characteristics  $\chi(X, \Omega^n_{X})$ of the vector bundles of differential forms of degree $n$ coincides, by Hodge theory, with the 
 topological Euler characteristic $\chi_{\mathrm{top}}(X)$ of $X$, and therefore with the degree $\int_X c_N(T_X)$ of the top Chern class of its tangent bundle:
\begin{equation}\label{Chi Omega Top X}
\sum_{n =0}^N (-1)^n \chi(X, \Omega^n_{X}) = \chi_{\mathrm{top}}(X) = \int_X c_N(T_X).
\end{equation}
The equality of the left- and right-hand side of  \eqref{Chi Omega Top X} --- which remains valid for a smooth algebraic variety over any field --- also follows from the Hirzebruch-Riemann-Roch theorem, applied to each of the vector bundles $\Omega^n_{X}$, and from the following equality of characteristic classes (see for instance \cite[Example 3.2.5]{Fulton98}):
\begin{equation}\label{Alter ch}
\sum_{n = 0}^N (-1)^n \mathrm{ch} (\Omega^n_{X}) \, \Td(T_{X}) = c_N (T_X).
\end{equation}

The cancellations alluded to above, leading to the proof of Theorem \ref{intro GK DNC pref}, may be seen as an avatar, valid in the relative situation of a pencil of varieties over a curve, of the equality~\eqref{Alter ch}. 

As mentioned above, the specific numerology in the definitions of the Griffiths and Kato heights is dictated by considerations of very different nature. We find it striking --- at least it was for us a great surprise --- that it also fits nicely with the formalism of characteristic classes and of the Grothendieck-Riemann-Roch theorem.

\textbf{0.5} Another situation we shall consider is that of a pencil of projective varieties
$$f : H \lra C$$
with non-singular total space $H$, and admitting only non-degenerate critical points.

Blowing up once these critical points in $H$ is enough to reduce to the situation of Theorem \ref{intro GK DNC pref}, and it is straightforward to express the characteristic classes on the blow-up in terms of characteristic classes on $H$. This therefore gives an expression for the alternating sum of the lower Griffiths heights attached to this pencil.

Moreover the computation in \cite[Prop. 3.10]{EFM21} of the so-called ``elementary exponents''  of a degeneration with only non-degenerate critical points  allows us to  compare the upper and lower Griffiths heights.

By this line of reasoning, we shall establish the following result.

\begin{theorem}
\label{intro GK critical points pref} 
Let $C$ be a connected smooth projective complex curve with generic point $\eta$, $H$ be a smooth projective $N$-dimensional complex scheme, and let:
$$
f : H \lra C
$$
be a  morphism of complex schemes. Let us assume that there exists a finite subset $\Sigma$ in $H$ such that $f$ is smooth on $H - \Sigma$ and admits a non-degenerate critical point at every point of $\Sigma$.

Then the following equalities hold:
\begin{equation}
\label{GK- points doubles pref}
  \sum_{n=0}^{2(N-1)} (-1)^{n-1} \hgt_{GK, -}(\H^n(H_\eta / C_\eta))
= \frac{1}{12} \int_H  c_1([\Omega^1_{H/C}]^\vee) c_{N-1}([\Omega^1_{H/C}]^\vee)  + u^-_N \, \vert\Sigma\vert,
\end{equation}
and
\begin{equation}
\sum_{n=0}^{2(N-1)} (-1)^{n-1} \hgt_{GK, +}(\H^n(H_\eta / C_\eta))
= \frac{1}{12} \int_H  c_1([\Omega^1_{H/C}]^\vee) c_{N-1}([\Omega^1_{H/C}]^\vee)  + u^+_N \, \vert\Sigma\vert, \label{GK+ points doubles pref}
\end{equation}
where $u^-_N$ and $u^+_N$ are the rational numbers defined by: 
\begin{equation*}
u^-_N := 
\begin{cases}
 (5N-3)/24 & \text{if $N$ is odd} \\
 N/24 & \text{if $N$ is even},
\end{cases}
\end{equation*}
and:
\begin{equation*}
u^+_N := 
\begin{cases}
 -(7N-9)/24 & \text{if $N$ is odd} \\
 N/24 & \text{if $N$ is even}.
\end{cases}
\end{equation*}
\end{theorem}

In the setting of pencils of projective hypersurfaces considered in $\text{0.2}$, it follows from the weak Lefschetz theorem and Poincar\'e duality that the upper (resp. lower) Griffiths height $\hgt_{GK, \pm}(\H^n(H_\eta/C_\eta))$ vanishes unless the degree $n$ is $N-1$. 

Finally, using some computations of characteristic classes on the hypersurface $H$ in $\PP(E)$, we shall obtain Theorem \ref{intro GK hyp P(E) pref} as a consequence of Theorem \ref{intro GK critical points pref}.

\textbf{0.6} The results described in $\text{0.3}$ and~$\text{0.5}$ are crucial in our strategy of proof of Theorem \ref{intro GK hyp P(E) pref}. It turns out that they also allow us to compute Griffiths heights attached to other natural instances of pencils of projective varieties.

The first of these computations concerns Lefschetz pencils. To state it, let us introduce some notation.

Let $V$ be a connected smooth projective complex scheme of pure dimension $N\geq 1$, embedded into some projective space $\PP^r$, of dimension $r\geq \max(N,2).$
Let $ \Lambda$ be a projective subspace of dimension $r-2$ in $\PP^r$ that intersects $V$ transversally, and let $P \subset \PP^{r\vee}$ be the projective line in the dual projective space $\PP^{r\vee}$ corresponding to $\Lambda$ by projective duality.

Let us denote by: 
$$\nu: \widetilde{\PP}^r_\Lambda \lra \PP^r$$
the blow-up of $\Lambda$ in $\PP^r$. If $\widetilde{V}$ denotes the proper transform of $V$ by $\nu$, the restriction:
$$\nu_{\mid \widetilde{V}} : \widetilde{V} \lra V$$
may be identified with the blow-up in $V$ of $\Lambda \cap V$, which is smooth of dimension $r-2$.
The projection of center $\Lambda$:
$$\PP^r - \Lambda \lra P$$
extends to a smooth morphism of complex schemes:
$$p: \widetilde{\PP}^r_\Lambda \lra P.$$

Recall that the pencil of hyperplanes in $\PP^r$ containing $\Lambda$ --- in other words the pencil of hyperplanes defined by $P$ --- is said to be a Lefschetz pencil with respect to the subvariety $V$ of $\PP^r$ when the morphism:
$$p_{\mid \widetilde{V}} : \widetilde{V} \lra P$$
 has a finite set $\Sigma$ of critical points, all of which are non-degenerate, and when the restriction
$$
p_{| \Sigma} : \Sigma \lra P
$$
is an injective map.

When this holds, as shown by Katz in \cite[Expos\'e XVII, cor. 5.6]{SGA7II}, the cardinality of $\Sigma$ is given in terms of characteristic classes over $V$ by the following formula:
\begin{equation}\label{KatzSGA pref}
\vert\Sigma\vert = \int_V (1 - c_1(\cO_V(1)))^{-2} c(\Omega^1_V).
\end{equation}

Using Theorem \ref{intro GK critical points pref} and the weak Lefschetz theorem, we are able to establish a similar expression for the Griffiths heights attached to Lefschetz pencils:  
 
\begin{theorem}
\label{intro Lefschetz pencils pref}
When the pencil of hyperplanes in $\PP^r$ containing $\Lambda$ is a Lefschetz pencil, then
 the following equalities hold:
\begin{equation}
\hgt_{GK,+}(\H^{N-1}(\widetilde{V}_\eta / P_\eta) ) 
= \frac{1}{12} \int_V (1 - c_1(\cO_V(1)))^{-2} c_1(\Omega^1_V)c(\Omega^1_V)  
+  \frac{(-1)^{N+1}}{12} \chi_{\mathrm{top}}(V) 
+ v_N^+\vert\Sigma\vert, \label{intro GK Lef pen + pref}
\end{equation}
and:
\begin{equation}
\hgt_{GK,-}(\H^{N-1}(\widetilde{V}_\eta / P_\eta) ) 
= \frac{1}{12} \int_V (1 - c_1(\cO_V(1)))^{-2} c_1(\Omega^1_V)c(\Omega^1_V)  
+  \frac{(-1)^{N+1}}{12} \chi_{\mathrm{top}}(V) 
+ v_N^- \vert\Sigma\vert, \label{intro GK Lef pen - pref}
\end{equation}
where:
\begin{equation*}
v^+_N := 
\begin{cases}
 7(N-1)/24 & \text{if $N$ is odd} \\
 (N+2)/24 & \text{if $N$ is even},
\end{cases}
\end{equation*}
and:
\begin{equation*}
v^-_N := 
\begin{cases}
 - 5(N-1)/24 & \text{if $N$ is odd} \\
 (N+2)/24 & \text{if $N$ is even}.
\end{cases}
\end{equation*}
\end{theorem}

The second of these results concerns pencils of surfaces with non-singular total space and whose singular fibers are semistable with smooth components (namely reduced divisors with strict normal crossings). 

This is a special case of Theorem \ref{intro GK DNC pref}, and in this case, the local terms involving the geometry of the singular fibers are easy to compute. Moreover we get a particularly  simple formula when the generic fibers of our pencil have irregularity zero, namely when their Hodge number $h^{1,0} = h^{0,1}$, or equivalently their first Betti number, vanishes.   

Indeed, in this case, we shall establish that the Griffiths height of the middle-dimensional cohomology of the pencil is given by the following topological formula:

\begin{corollary} In the situation of Theorem \ref{intro GK DNC pref}, when $N=3$,   when the degenerations of the fibers of $g$ are semistable with smooth components,  and when the generic fibers of $g$ have irregularity zero,  the following equality holds:
\begin{equation}\label{surfacessemistable pref}
\hgt_{GK} (\H^2(Y_\eta / C_\eta))  = \frac{1}{12} \int_Y c_1(\omega^1_{Y/C})  c_2 (\omega^1_{Y/C}) - \frac{1}{12} \chi_{\mathrm{top}} (D^2 - D^3),
\end{equation}
where $D^2$ (resp. $D^3$) denotes the subscheme of codimension $2$ (resp. $3$) of $Y$ defined as the union of all the intersections of two (resp. three) distinct components of the divisor with strict normal crossings~$Y_\Delta$.
\end{corollary}

$\mathbf{0.7}$ As detailed in \ref{GK CY} below, we will show how the general results of $\text{0.3}$ may also be applied to pencils of projective varieties whose smooth fibers are \emph{Calabi-Yau} manifolds, namely smooth projective varieties whose dualizing sheaf  is a trivial line bundle.

Actually, in this case, the Griffiths height is the degree of the so-called \emph{BCOV line bundle}, whose metric properties are studied 
 by Eriksson, Freixas and Mourougane
in \cite{EFM18} and \cite{EFM21}.
This line bundle arises in the study of the \emph{BCOV invariant} of Calabi-Yau manifolds, which was introduced by Bershadsky, Cecotti, Ooguri, and Vafa in \cite{BCOV94} as a product of powers of the determinant of the Laplacian acting on differential forms. The specific numerology in its definition, which is similar to the one in \eqref{def GK det}, is dictated by its role  in string theory and mirror symmetry.

It is already mentioned in \cite[p. 374]{BCOV94} that, in this situation, this specific numerology leads to some cancellations of characteristic classes. Moreover the importance of using the Grothendieck-Riemann-Roch theorem in computations related to the BCOV line bundle is already made clear in \cite{EFM21}.\footnote{See also \cite{FLY08, Yoshikawa15, Liu-Xia19} for earlier mathematical works on BCOV invariants, and \cite{EFM22} for their relation with genus one mirror symmetry.}

The papers \cite{EFM18} and \cite{EFM21} have been an important source of inspiration for this work.  Various ideas and techniques developed in these papers for the investigation of families of Calabi-Yau manifolds and their degenerations turned out to be relevant in a much wider framework, in particular for the investigation of pencils of projective hypersurfaces of arbitrary dimension and degree. 

This memoir leaves untouched many basic questions concerning Griffiths heights, even in the special case of pencils of smooth hypersurfaces. For instance, their properties outside the generic situation, when the  fibers of bad reduction are allowed to have more general singularities than ordinary double points; or their finiteness properties, which should constitute a geometric analogue of the expected finiteness properties of Kato heights of motives, and should be expressed in terms of bounded families. 

Our upcoming papers \cite{MordantGIT1, MordantGIT2} explore these questions, by relating the Griffiths heights of pencils of projective hypersurfaces with their heights defined by means of geometric invariant theory.

{\bf 0.8}   \emph{Acknowledgements} The present memoir was written as part of the author's PhD dissertation realized at the Laboratoire de Math\'ematiques d'Orsay, under the direction of Jean-Beno\^{\i}t Bost. We are very grateful to him for numerous discussions concerning the Kato and Griffiths heights and their computations, and for his careful reading of earlier versions of this manuscript.

As mentioned above, the works \cite{EFM21} and \cite{EFM22} by Eriksson, Freixas, and Mou\-rou\-gane have been an important source of inspiration for the derivation of  Theorem \ref{intro GK DNC pref} and for the results in Chapter~\ref{SteenbrinkElExp}. We are very grateful to Dennis Eriksson and Gerard Freixas for  enlightening correspondence and discussions concerning various aspects of their work, notably concerning the proof of  Proposition \ref{elementary exponents non-degenerate singularities}.

Our deepest thanks go to Daniel Huybrechts and Christophe Mourougane, the \emph{rapporteurs} of the author's PhD dissertation, for their careful reading of a first version of this memoir, and for their helpful comments and suggestions.  We are also grateful to two anonymous referees for their valuable comments and for their suggestions of additional references and of improvements of the exposition.

Due to the author's physical disability, the process of typesetting this document into \LaTeX was particularly time-consuming. We are very grateful to Damien Simon for his precious  help on this process.

\chapter*{Conventions and notation} 

If $X$ is a scheme of finite type over some field $k$, $Z$ is a closed subscheme of $X$, locally complete intersection in $X$, and if $i: Z \ra X$ is the inclusion map and $\alpha$ is a class in $\CH^\ast(X)$, we shall denote by: $$\alpha_{| Z} := i^\ast \alpha$$ its image by the Gysin morphism: $$i^\ast: \CH^\ast(X) \lra \CH^\ast(Z).$$ 
We shall  use this notation only when both $X$ and $Z$ are smooth over $k$.

If $X$ is a smooth $k$-scheme of pure dimension $N$, where $N$ is a non-negative integer, we denote by:
$$\int_X: \CH^\ast(X)_\Q \lra \Q$$
the map that sends a class $\alpha$ in $\CH^\ast(X)_\Q$ to the degree of its homogeneous component $\alpha^{[N]}$ in $\CH^N(X)_\Q \simeq \CH_0(X)_\Q$.

If $X$ is a complex scheme (separated and of finite type), we denote by $\chi_{\mathrm{top}}(X)$ the topological Euler characteristic of the complex analytic space $X(\C).$

Following \cite{Kato89} and \cite{Illusie94}, if $Y$ is a smooth complex scheme, $C$ is a smooth complex curve, and $Y \ra C$ is a morphism (necessarily smooth over a Zariski dense open subset of $C$) whose singular fibers are a divisor with normal crossings, we shall denote \textit{}by $\omega^1_{Y/C}$ the sheaf of relative logarithmic differentials (see \ref{Steenbrink theory and logarithmic Hodge bundles} below) and, for any non-negative integer $p$, by $\omega^p_{Y/C}$ its $p$-th exterior power. 

We are aware that the similarity with the traditional  notation for the tensor powers of the relative dualizing sheaf $\omega_{Y/C}$ may be confusing. However this notation for relative  logarithmic differentials takes up less space than the alternative usual notation for those, which would make a number of already  unwieldy formulas in this memoir even more cumbersome. 

If $Y$ is a smooth complex scheme and $C$ is a smooth complex curve, a morphism of complex schemes $g : Y \ra C$ is said to admit a \emph{non-degenerate critical point} at a point $P$ in $Y$ if its differential at $P$ vanishes and if its Hessian at $P$ is a non-degenerate bilinear form on the tangent space $T_{Y, P}$ of $Y$ at $P$ (with values in the line $T_{C, g(P)}$). This is equivalent to the fiber $Y_{g(P)}$ admitting an \emph{ordinary double point} at $P$. 

As a rule, we will refer to this situation by saying that the morphism $g$ admits  a non-degenerate critical point. However, according to common usage, when dealing with blow-ups of the divisor $Y_{g(P)}$, we will often say   that $P$ is  an ordinary double point of this divisor.

\mainmatter
\selectlanguage{english}
\chapter{Introduction}

\section[The Kato height and the Griffiths height]{The Kato height of motives over number fields and the Griffiths height of 
  variations of Hodge structures}

\subsection{The Kato height of motives}\label{subsec:Katoline} In \cite{Kato14}, Kato has introduced  a height function associated to pure motives over a number field $K$, defined in terms of realizations ($p$-adic, Betti, de Rham).

If $M$ is such a motive, say over $K = \Q$, and if $M_{\mathrm{dR}}$  denotes its de Rham realization,  $F^\bullet M_{\mathrm{dR}}$ its Hodge filtration, and $$\text{gr}^r _F M_{\mathrm{dR}}:= F^r M_{\mathrm{dR}}/ F^{r+1} M_{\mathrm{dR}}$$ the associated subquotients, Kato considers the  one-dimensional $K$-vector space:
\begin{equation}\label{Katoline}
L(M) = {\bigotimes}_{r \in \Z} \, ( \textrm{det} \, \text{gr}^r _F M_{\mathrm{dR}}) ^{\otimes r}
\simeq {\bigotimes}_{r > r_0} \, (\textrm{det}_K F^r M_{\mathrm{dR}} ),
\end{equation}
where $r_0$ denotes an integer such that 
$F^{r_0} M_{\mathrm{dR}} = M_{\mathrm{dR}}.$

The $\R$-vector space $L(M)_\R$ is equipped with a canonical norm, defined by means of Hodge theory. Similarly, for every prime number $p$, the $\Q_p$-vector space $L(M)_{\Q_p}$ is endowed with a natural $\Z_p$-lattice, defined by means of the $p$-adic realization of $M$ and of $p$-adic Hodge theory. Equipped with these additional structures, $L(M)$ becomes a Hermitian line bundle $\overline{L(M)}$ over $\Spec \Z$ and, as such, admits an Arakelov degree. This degree defines the \emph{Kato height} of the motive $M$:
\begin{equation}\label{Katohgtdef}
{\hgt}_K(M) := 
\dega \overline{L(M)}  \in \R.
\end{equation}

This definition is inspired by the definition of the Faltings height of an abelian variety over a number field --- which is a special instance of Kato's definition when $M$ is a  motive of weight one --- and by Faltings' proof of the invariance by isogenies of the height of abelian varieties, which played a key role in his proof of the Tate conjecture for abelian varieties over number fields \cite{Faltings83}. 

The specific numerology in the definition \eqref{Katoline} of the $K$-line $L(M)$ --- where the determinant $\textrm{det}_K \text{gr}^r _F M_{\mathrm{dR}}$ of the $r$-th subquotient $\text{gr}^r _F M_{\mathrm{dR}}$ of the Hodge filtration occurs at the tensor power $r$ --- has been actually dictated to Kato by its ``compatibility" with the Weil conjectures: as shown by Kato, the Weil conjectures imply that the  height defined as the Arakelov degree \eqref{Katohgtdef} satisfies an invariance property that generalizes  
the invariance by isogenies of the height of abelian varieties to motives of arbitrary weight. 

After Kato's seminal paper \cite{Kato14}, Kato heights  have been generalized and investigated by Kato himself \cite{Kato18}, by Koshikawa \cite{Koshikawa15} who studied their possible applications to classical boundedness and semisimplicity conjectures concerning motives over number fields, and by Venkatesh \cite{Venkatesh18}, who started to explore their relationship to automorphic forms.

\subsection{The Griffiths line bundle and the Griffiths heights of a VHS over a projective curve}\label{subsec:griffithsline} As pointed out in \cite{Kato18}, a forerunner of Kato's definitions \eqref{Katoline} and \eqref{Katohgtdef} already appears in Griffiths' famous work concerning the Hodge structures on the cohomology of complex projective varieties and the period  mapping associated to a family of such varieties. 

Recall that an integral variation of Hodge structures (VHS) of weight $n \in \N$ over a complex analytic manifold $S$ is a pair:
$$\V := (V_\Z, \cF^\bullet) ,$$
where $V_\Z$ is a local system of finitely generated free $\Z$-modules over $S$, and 
$$\cF^\bullet: \cF^0:= \cV \supseteq \cF^1 \supseteq \dots \supseteq \cF^n \supseteq \cF^{n+1} = 0$$
is a decreasing filtration by subbundles of the holomorphic vector bundle 
$$\cV := V_\Z \otimes_{\Z_S} \cO^{\mathrm{an}}_S,$$
attached to the local system $V_\Z$. 
These data must satisfy the Hodge decomposition condition --- so that for every $s\in S$, the fiber $\V_s:=  (V_{\Z,s}, \cF_s^\bullet)$ is a Hodge structure of weight $n$ --- and the Griffiths horizontality condition (see for instance \cite[III]{Voisin02} or \cite[Chapter 10]{Peters-Steenbrink08}).

To a holomorphic vector bundle $\mathcal{V}$ over a  complex analytic manifold $S$, equipped  with a finite decreasing filtration 
$$F^\bullet := F^0:= \mathcal{V} \supseteq F^1 \supseteq \dots \supseteq F^n \supseteq F^{n+1} = 0$$
by subbundles, we may attach the holomorphic line bundle on $S$:
$$\mathcal{\cGK}_S(\cV, F^\bullet) := \bigotimes_{i=1}^n \det F^i \simeq \bigotimes_{r=0}^n (\det F^r/F^{r+1})^{\otimes r} ,$$ 
that we will call the \emph{Griffiths line bundle} attached to $(\cV, F^\bullet).$ 

In particular, to a variation of Hodge structures $\V= (V_\Z, \cF^\bullet)$ over $S$ as above is attached its Griffiths line bundle:
\begin{equation}\label{defGriffithshgt}
\mathcal{\cGK}_S  (\V) := \mathcal{\cGK}_S(\cV, \cF^\bullet) = \bigotimes_{r=0}^n (\det \cF^r/\cF^{r+1})^{\otimes r}.
\end{equation}
This definition is introduced by Griffiths in \cite[Section 7 b)]{Griffiths70}, under the terminology of \emph{canonical line bundle} of the   variation of Hodge structures. The specific tensor product of powers  of line bundles that appears in the right hand side of \eqref{defGriffithshgt} is the same as in the definition of the ``Kato line"  in \eqref{Katoline}. However it 
arises in Griffiths' work for completely different reasons. As shown in \cite{Griffiths70}, a polarization on the VHS $\V$ induces canonical Hermitian metrics on the vector bundles $\cF^r/\cF^{r+1}$, and therefore on the line bundle $\mathcal{\cGK}_S  (\V)$. The rationale behind the introduction of the line bundle \eqref{defGriffithshgt} is that the curvature\footnote{or more precisely the first Chern form.} of $\mathcal{\cGK}_S  (\V)$ equipped with this canonical Hermitian structure is always a \emph{non-negative} real $(1,1)$-form on $S$. 

In particular, when $S$ is a connected smooth projective complex curve $C$, the degree of this line bundle:
$$\hgt_{GK}(\V) := \deg_C \mathcal{\cGK}_S  (\V)$$
--- which we will call the \emph{Griffiths-Kato height}, or simply the  \emph{Griffiths height}, of the VHS $\V$ over $C$ --- is a non-negative integer when $\V$ admits a polarization.

\subsection{The construction of Peters and the Griffiths height of the cohomology of a pencil of projective varieties}\label{PetersGriffitsCohomDef}
The Griffiths height $\hgt_{GK}(\V)$ may be seen as an analog of the Kato height ${\hgt}_K(M)$, where the number field $K$ is replaced by the function field $\C(C)$, and where the motive $M$ is replaced by the VHS $\V$ over $C$. As shown by Peters \cite{Peters84}, it is actually possible to extend this definition to the situation where $\V:= (V_\Z, \cF^\bullet)$ is a variation of Hodge structures over a  Zariski dense open subset $\Cc$ of $C$.

Peters' construction relies on Deligne's work (\cite{Deligne70}) on extensions of analytic local systems defined on a Zariski dense open  subset of a smooth projective complex variety 
and on Schmid's results (\cite{Schmid73}) on variations of Hodge structures.

It naturally arises in two variants, involving a choice of sign and the so-called upper and lower extensions, and leads to the introduction of some canonical extensions over $C$, denoted by $\mathcal{\cGK}_{C,+}  (\V)$ and $\mathcal{\cGK}_{C,-}  (\V)$, of the analytic line bundle $\mathcal{\cGK}_{\Cc}  (\V)$ over $\Cc$.\footnote{See section  \ref{Peters construction} below for more details. Peters actually considers only the extension  $\mathcal{\cGK}_{C,+}  (\V)$,  for which the positivity of the Griffiths height in the polarized case still holds. Steenbrink's theory (\cite{Steenbrink76, Steenbrink77, Peters-Steenbrink08})  is naturally related to the extension  $\mathcal{\cGK}_{C,-}  (\V)$. This discrepancy leads to some mistakes in the published literature, which are clarified notably in \cite{Kollar86}, \cite{Moriwaki87}, and \cite{EFM21}.}

The associated Griffiths heights will be denoted by:
$$ \hgt_{GK, +}(\V) := \deg_C \mathcal{\cGK}_{C, +}  (\V) \quad \mbox{ and } \quad \hgt_{GK, -}(\V) := \deg_C \mathcal{\cGK}_{C, -}  (\V).$$
The line bundles $\mathcal{\cGK}_{C,+}  (\V)$ and $\mathcal{\cGK}_{C,-}  (\V)$, and therefore the heights $ \hgt_{GK, +}(\V)$ and $\hgt_{GK, -}(\V)$ actually coincide when the local monodromy of $V_\Z$ at every point of the complement $C - \Cc$  (which is automatically quasi-unipotent) is unipotent.

The  setting considered by Peters covers the variations of Hodge structures that ``arise from geometry", namely those associated to a pencil of projective varieties parametrized by $C$ and to the cohomology of its fibers.  

Indeed consider   a smooth projective complex variety $X$ and 
$$f: X \lra C$$
a surjective morphism.  There exists $\Cc$ as above such that the morphism:
$$ f_{\mid X_{\Cc}}: X_{\Cc} := f^{-1}(\Cc) \lra \Cc$$
is smooth. For every $n \in \N,$ we may consider the VHS $\H^n(X_{\Cc}/\Cc)$ over $\Cc$, defined by the relative Betti cohomology in degree $n$, and the filtration on the  relative de Rham cohomology induced by the relative Hodge-de Rham spectral sequence, then
the line bundles $\mathcal{\cGK}_{C,+}  (\H^n(X_{\Cc}/\Cc))$ and $\mathcal{\cGK}_{C,-}  (\H^n(X_{\Cc}/\Cc))$ over $C$, and the corresponding Griffiths heights:
 $$ \hgt_{GK, +}(\H^n(X_{\Cc}/\Cc)) := \deg_C \mathcal{\cGK}_{C, +}  (\H^n(X_{\Cc}/\Cc))$$
 and:
 $$\hgt_{GK, -}(\H^n(X_{\Cc}/\Cc)) := \deg_C \mathcal{\cGK}_{C, -}  (\H^n(X_{\Cc}/\Cc)).$$
 
 These numbers depend only on the generic fiber $X_\eta$ of $f$ over the generic point $\eta$ of $C$, and will also be denoted by $\hgt_{GK, +}(\H^n(X_\eta/C_\eta))$ and $\hgt_{GK, -}(\H^n(X_\eta/C_\eta)).$

 \subsection{Concerning the content of this paper}\label{subsub content}
 
 The starting point of this work has been the observation that, in spite of the significance of Kato's heights for the understanding of pure motives over number fields, basically nothing was known concerning them beyond motives of abelian varieties. In particular, no explicit computation of these heights was known concerning motives of weight greater than one\footnote{apart from motives constructed by tensor operations from the motive of weight 1 associated with an abelian variety.}.  Somewhat surprisingly, the situation is not better concerning the Griffiths heights  $\hgt_{GK, \pm}(\H^n(X_\eta/C_\eta))$ when $n\geq 2$,  which constitute their  geometric counterparts.

 Our aim has been to compute these Griffiths heights $\hgt_{GK, \pm}(\H^n(X_\eta/C_\eta))$, when $n \geq 2$,  in a number of significant situations, with a view towards the derivation of similar formulas concerning Kato's heights. These formulas for Kato's heights would constitute some material for exploring various finiteness conjectures of Kato and Koshikawa (\cite{Kato18, Koshikawa15}).
 
 In this paper, we establish formulas for these Griffiths heights, when  $n$ is the relative dimension of $X/C$, and when  $X/C$ is a semistable pencil of surfaces,  a pencil of hypersurfaces in a projective space,  or a Lefschetz pencil.
 
 Our approach has been inspired by the classical computation by Hirzebruch of the Hodge numbers of complete intersections in projective space (\cite{Hirzebruch56} and \cite[Appendix I]{Hirzebruch95}), and by the recent results of  Eriksson, Freixas, and  Mourougane concerning  BCOV invariants of  Calabi-Yau manifolds of arbitrary dimension (\cite{EFM18, EFM21, EFM22}).\footnote{\label{note BCOV} The BCOV invariants of Calabi-Yau manifolds equipped with a K\"ahler metric are introduced in \cite{BCOV94}. The paper \cite{FLY08} introduces a normalized version of the BCOV invariant attached to K\"ahlerian Calabi-Yau threefolds which does not depend on the choice of a specific K\"ahler metric, and investigates its asymptotic behaviour in degenerations; see also \cite{Yoshikawa15} and \cite{Liu-Xia19}. The BCOV invariants of Calabi-Yau manifolds of arbitrary dimension are introduced and investigated in \cite{EFM18} and \cite{EFM21}, and their relation with genus one mirror symmetry is studied in \cite{EFM22}.}
  
  In the special case of Calabi-Yau manifolds, the alternating tensor product of Griffiths line bundles, which plays a key role in our computation, coincides with the BCOV line bundle that is investigated in  the work of Eriksson, Freixas, and  Mourougane.  Our general formulas shed some light on the remarkable properties of this alternating tensor product when dealing with families of Calabi-Yau manifolds.  
  
  The next sections of this introduction are devoted to the statement of our results. Section  \ref{Peters construction} provides some details on Peters' construction alluded to above, and could be skipped at first reading. 
  In the remaining sections, we state our results as equalities of rational numbers, providing closed formulas for suitable Griffiths heights $\hgt_{GK, \pm}(\H^n(X_\eta/C_\eta))$. We shall actually establish these formulas in a more refined form, concerning the class of the line bundles $\mathcal{\cGK}_{C, \pm}  (\H^n(X_{\Cc}/\Cc))$ in the rational Picard group $\CH^1(C)_\Q$ of the base curve $C$.
  
  The next chapters of the paper are organized as follows. 
  
  Chapter 2 is devoted to  preliminary results, most of them classical, concerning Steenbrink's theory. They will allow us to express the line bundle $\mathcal{\cGK}_{C, \pm}  (\H^n(X_{\Cc}/\Cc))$ in terms of the higher direct images of relative logarithmic differentials when the singular fibers of $X/C$ are divisors with strict normal crossings. 
  
  In Chapter 3, we establish various identities concerning the characteristic classes  of relative differentials and relative logarithmic  differentials in this situation. We believe that several of these formulas are of independent interest.
  
  Chapter 4 contains the proof of the first of our main results, the computation of the alternating sum:
  \begin{equation}\label{altsimhgtgr}
  \sum_{n \in \N} (-1)^{n-1}\,  \hgt_{GK, -}(\H^n(X_\eta/C_\eta)).
  \end{equation}
  Besides the results of Steenbrink recalled in Chapter 2, this computation relies on the Grothendieck-Riemann-Roch theorem and on the identities on characteristic classes established in Chapter~3. It turns out that the specific numerology in the definitions of the Kato line \eqref{Katoline} and of the Griffiths line bundle \eqref{defGriffithshgt}, already discussed in \ref{subsec:Katoline} and \ref{subsec:griffithsline}, leads to remarkable cancellations when applying the  Grothendieck-Riemann-Roch  theorem, and allows one to derive some (relatively) simple formulas for this alternating sum.

In Chapter 5, we give an expression for the alternating sum \eqref{altsimhgtgr} in the situation, better suited to explicit computations, where the 
morphism $X \ra C$ has only non-degenerate critical points.

Finally, in Chapter 6, we combine our previous expressions for the sum \eqref{altsimhgtgr} with the weak Lefschetz theorem to derive, in various significant situations, a formula for its middle-dimensional term $\hgt_{GK, -}(\H^n(X_\eta/C_\eta)),$ where $n$ denotes the relative dimension of $X/C$.

\section[Peters' construction]{Peters' construction and the Griffiths height of a variation of Hodge structures with possible degenerations over a curve}\label{Peters construction}

Let $C$ be a connected smooth projective complex curve, let $\Delta$ be a finite subset of  $C$,  and let us denote by:
$$\Cc := C - \Delta$$
the connected smooth algebraic curve defined by its complement.

 Let $D \subset \C$ be a subset such that the projection $\C \twoheadrightarrow \C / \Z$ induces a bijection $D \overset{\sim}{\ra} \C/\Z$, and let: 
$$\log_{2i\pi D}: \C^\ast \lra 2 i \pi D \subset \C$$ 
be the  inverse of the map $\exp_{\mid 2 i \pi D}.$  

For example, we can take as $D$ the subset of~$\C$:
$$D_+ := \{z \in \C | 0 \leq \Re(z) < 1\},$$
or:
$$D_- := \{z \in \C | -1 < \Re(z) \leq 0\}.$$
We shall note $\log_+$ for $\log_{2i\pi D_+}$, and $\log_-$ for $\log_{2i\pi D_-}$.

 Let $(\cV, \nabla)$ be an analytic vector bundle on $\Cc$ with a flat connection.

\begin{defpropositionEng}[{\cite[II, Prop. 5.4]{Deligne70}, see also \cite[``Key Lemma" p. 547]{Katz76}, and \cite[section 2.1]{EFM21}}]
 \label{Deligne extension} 
 
The bundle with connection $(\cV,\nabla)$ admits a unique analytic\footnote{Since $C$ is projective, the analytic vector bundle with connection $(\cV,\nabla)$ is actually algebraizable.} extension with logarithmic connection $(\overline{\cV}, \overline{\nabla}_D)$ on $C$, such that for all $x$ in $\Delta$, the eigenvalues of the residue endomorphism $\Res_x \overline{\nabla}_D \in \mathrm{End}(\overline{\cV}_D{}_x)$ belong to $- D$.
\end{defpropositionEng}
Following the convention of \cite{EFM21}, actually introduced in \cite[Section 2]{Kollar86} and expanded upon in \cite[Section 2]{Moriwaki87}, we shall call this extension the \emph{upper} (resp. \emph{lower}) \emph{Deligne extension}, if $D = D_+$ (resp. $D_-$). Note that the eigenvalues of the residue of the connection defining the \emph{lower} extension are \emph{non-negative}.

Consider now a VHS on $\Cc$:$$\V := (V_\Z, \nabla, \cF^\bullet),$$ 
where as before $\cF^\bullet$ denotes a filtration of the vector bundle $\cV := V_\Z \otimes_{\Z_\Cc} \cO^{an}_\Cc$. According to a result of Griffiths (\cite[Theorem (4.13), a)]{Schmid73}), 
the subbundles $(\cF^p)_{p \geq 0}$ are algebraic vector subbundles of $\cV_{\mid \Cc}$ equipped with the algebraic structure defined by its Deligne extension.\footnote{The algebraic structure of the restriction $\cV_{\Cc}$, defined by means of Definition and Proposition  \ref{Deligne extension} is easily seen not to depend on the choice of the subset $D$.}

Let us denote by $(\overline{\cV}_+, \overline{\nabla}_+)$ and $(\overline{\cV}_-, \overline{\nabla}_-)$  the upper and lower Deligne extensions of the bundle with connection $(\cV, \nabla)$, and for every
integer $p \geq 0$, by $\overline{\cF}^p_+$ and $\overline{\cF}^p_-$ the vector subbundles (that is, the locally direct summands) of $\overline{\cV}_+$ and $\overline{\cV}_-$ over $C$ that extend the vector subbundle $\cF^p$ of: $$\cV \simeq \overline{\cV}_{+\mid \Cc} \simeq \overline{\cV}_{-\mid \Cc}.$$

The \emph{upper} and \emph{lower Griffiths line bundles} of the VHS $\V$ on $\Cc$ are defined as the line bundles over $C$:
$$\GK_{C,+}(\V_\eta) := \GK_C(\overline{\cV}_+, \overline{\cF}^\bullet_+) = \bigotimes_{p \geq 0} (\det \overline{\cF}^p_+ / \overline{\cF}^{p+1}_+)^{\otimes p},$$
and:
$$\GK_{C,-}(\V_\eta) := \GK_C(\overline{\cV}_-, \overline{\cF}^\bullet_-) = \bigotimes_{p \geq 0} (\det \overline{\cF}^p_- / \overline{\cF}^{p+1}_-)^{\otimes p}.$$
Finally the \emph{upper} and \emph{lower Griffiths-Kato heights} (or simply \emph{upper} and \emph{lower Griffiths heights}) of~$\V$ are defined as the degree of these line bundles:
$$\hgt_{GK,+}(\V_\eta) := \deg_C(\GK_{C,+}(\V_\eta)),$$
and:
$$\hgt_{GK,-}(\V_\eta) := \deg_C(\GK_{C,-}(\V_\eta)).$$

Observe that, by construction, the line bundles $\GK_{C,\pm}(\V_\eta)$ and the heights  $\hgt_{GK,\pm}(\V_\eta)$ are unchanged when $\Delta$ is replaced by a larger finite subset of $C$. They depend only on the restriction of the VHS $\V$ to some arbitrary small Zariski open neighborhood $\Cc$ of the generic point $\eta$ of $C$, as indicated by our notation. 

Observe also that, when the local monodromy of $V_\Z$ at every point of $\Delta$ (which is automatically quasi-unipotent according to a result of Borel \cite[Lemma (4.5)]{Schmid73}) is unipotent, 
the upper and lower Deligne extensions $(\overline{\cV}_+, \overline{\nabla}_+)$ and $(\overline{\cV}_-, \overline{\nabla}_-)$ coincide, and therefore the Griffiths line bundles $\GK_{C,+}(\V_\eta)$ and $\GK_{C,-}(\V_\eta)$ also, and their degree $\hgt_{GK,+}(\V_\eta)$ and $\hgt_{GK,-}(\V_\eta)$ as well. In this case, we shall denote them by $\GK_{C}(\V_\eta)$ and $\hgt_{GK}(\V_\eta)$.

These notions are motivated by the following theorem of Peters, which extends an earlier positivity result of Griffiths \cite{Griffiths70}:

\begin{theorem} [{\cite[Th. 4.1]{Peters84}}]\label{Peterspositive}
If the VHS $\V:=(V_\Z, \cF^\bullet)$ over $\Cc$ is polarized, 
then the upper Griffiths height is always non-negative:
\begin{equation}\label{Peters ineq} \hgt_{GK,+}(\V_\eta) \geq 0.
\end{equation}
When furthermore the weight of the VHS $\V$ is positive, equality holds in \eqref{Peters ineq} if and only if all the subbundles $\cF^p$ are flat for $\nabla$ and the local monodromy of $V_\Z$ at every point of $\Delta$ is unipotent.
\end{theorem}

It is possible to introduce a variant of the heights $\hgt_{GK,+}(\V_\eta)$ and $\hgt_{GK,-}(\V_\eta)$, namely the \emph{stable Griffiths-Kato height} (or simply \emph{stable Griffiths height}) $\mathrm{ht}_{GK,stab}(\V_\eta),$ which is defined as follows (see \cite{Kato18}; this construction is inspired by the definition of the stable Faltings  height of abelian varieties over number fields).

For every $x \in \Delta$, using the result of Borel cited above (\cite[Lemma (4.5)]{Schmid73}), the eigenvalues of the local monodromy of $V_\Z$ at $x$ are roots of unity. We shall denote by $r_x$ the lcm of the order of these roots of unity.

Let $C'$ be a connected smooth projective curve with generic point $\eta'$, and let:
$$\sigma : C' \longrightarrow C$$
be a finite morphism, such that for all $x' \in \sigma ^{-1}(\Delta)$, the index of ramification of $\sigma$ in $x'$ is a multiple of $r_{\sigma(x')}$. Such pairs $(C', \sigma)$ are easily seen to exist. They may be constructed for instance as cyclic coverings of $C$.

Pulling back $\V$ by $\sigma$, we get a VHS:
$$\V' := (\sigma^\ast V_\Z, \sigma^\ast \cF^\bullet)$$
over the complement $\Cc' := C' - \Delta'$
of $\Delta' := \sigma^{-1}(\Delta).$
By construction, the local monodromy of $V'_\Z := \sigma^\ast V_\Z$ at every point of $\Delta'$ is unipotent. Therefore we may define:
$$\mathrm{ht}_{GK,stab}(\V_\eta) := \frac{1}{\deg (\sigma)} \deg _{C'} (\mathcal{\cGK}_{C'} (\sigma ^\ast \V_\eta) ). $$
This rational number is easily seen not to depend on the choice of the ramified covering $(C', \sigma)$ and to satisfy the following inequalities:
$$\hgt_{GK,-}(\V_\eta) \leq  \mathrm{ht}_{GK,stab}(\V_\eta) \leq \hgt_{GK,+}(\V_\eta).$$
These inequalities are simple consequences of the definition of the upper and lower Deligne extensions and do not require the VHS $\V$ to be polarized.

\section{Alternating sums of Griffiths heights and characteristic classes}

\subsection{The main formulas for alternating sums of Griffiths heights}\label{subsec main alter}

Using the Gro\-then\-dieck-Riemann-Roch formula and Steenbrink's theory, 
we shall establish  expressions  in terms of characteristic classes  for the alternating sum of Griffiths heights associated to a pencil of projective varieties whose singular fibers are divisors with strict normal crossings. 

\begin{theorem}
\label{intro GK DNC}
Let $C$ be a  connected smooth projective complex curve with generic point $\eta$,  let $Y$ be a connected smooth projective  complex variety of  dimension $N$, and let
$$
g : Y \lra C
$$
be a surjective morphism of complex varieties. Let us assume that there exists a finite subset $\Delta$ in $C$ such that $g$ is smooth over $C - \Delta$, and such that the divisor $Y_\Delta$ is a divisor with strict normal crossings in $Y$.

Then the following equality holds:
\begin{equation}\label{eq:intro GK DNC}
\sum_{n=0}^{2(N-1)} (-1)^{n-1} \hgt_{GK, -}(\H^n(Y_\eta / C_\eta))
=   \int_Y \rho_{N-1}(\omega_{Y/C}^{1\vee})  \frac{\Td([T_g])}{\Td(\omega_{Y/C}^{1\vee})},
\end{equation}
where $\omega^1_{Y/C}$ denotes the vector bundle over $Y$ of relative logarithmic $1$-forms\footnote{
Observe that this vector bundle does not depend on the choice of the divisor $\Delta$ in $C$ satisfying the above conditions.}:
$$\omega^1_{Y/C} := \Omega^1_{Y/C} (\log Y_\Delta),$$ 
where $[T_g]$ is the relative tangent class in $K$-theory:
$$[T_g] := [T_{Y/\C}] - g^\ast [T_{C/\C}] \in K^0(Y),$$
and where $\Td$ denotes the Todd class and $\rho_{N-1}$  the characteristic class defined in terms of the Chern classes by:
\begin{equation}\label{rhoChern}
\rho_{N-1} := c_{N-2} - \frac{N-1}{2} c_{N-1} + \frac{1}{12} c_1 c_{N-1}.
\end{equation}
\end{theorem}

The occurence of the height $\hgt_{GK, -}$, and not of $\hgt_{GK, +}$  or $\hgt_{GK, stab}$, in the left-hand side of \eqref{eq:intro GK DNC} is due to the specific role played by the \emph{lower} extensions in Steenbrink's  theory.

The class $\rho_{N-1}$, defined by the right-hand side of \eqref{rhoChern}, already appears in a computation of anomaly in \cite[5.8]{BCOV94}. \footnote{We were not aware of this reference when completing the first version of this memoir. We are grateful to Christophe Mourougane for pointing out this reference.} 

In the right-hand side of \eqref{eq:intro GK DNC}, we denote by:
$$\int_Y: \CH^\ast(Y)_\Q \lra \Q$$
the map that sends a class $\alpha$ in $\CH^\ast(Y)_\Q$ to the degree of its homogeneous component $\alpha^{[N]}$ in $\CH^N(Y)_\Q \simeq \CH_0(Y)_\Q$.

In particular, when the morphism $g$ is smooth, the class $[T_g]$ is the class of the vector bundle $$T_{Y/C} \simeq  \Omega^{1\vee}_{Y/C} = \omega^{1\vee}_{Y/C},$$
 and \eqref{eq:intro GK DNC} becomes  the simpler formula:
\begin{equation}\label{main formula smooth}
  \sum_{n=0}^{2(N-1)} (-1)^{n-1} \hgt_{GK, -}(\H^n(Y_\eta / C_\eta))
= \frac{1}{12} \int_Y c_1(\Omega^{1\vee}_{Y/C}) c_{N-1}(\Omega^{1\vee}_{Y/C}). 
\end{equation}

When the morphism $g$ is not necessarily smooth, the restrictions in $K^0(Y-Y_\Delta)$ of $[T_g]$ and  $[\omega_{Y/C}^{1\vee}]$ still coincide, and there exists a cycle  with rational coefficients $V$ supported by $Y_\Delta$ such that:
$$\frac{\Td([T_g])}{\Td(\omega_{Y/C}^{1\vee})} = 1 + [V],$$
and formula \eqref{eq:intro GK DNC} may be written:
\begin{equation}\label{main formula V}
 \sum_{n=0}^{2(N-1)} (-1)^{n-1} \hgt_{GK, -}(\H^n(Y_\eta / C_\eta))
= \frac{1}{12} \int_Y c_1(\omega_{Y/C}^{1\vee}) c_{N-1}(\omega_{Y/C}^{1\vee}) + \int_Y \rho_{N-1}(\omega_{Y/C}^{1\vee})  [V]. 
\end{equation}

Formula \eqref{main formula V} is similar to \eqref{main formula smooth}, but  its right-hand side contains a second term that depends on  the geometry of the  morphism  $g$ in the infinitesimal neighborhood of the ``bad fibers" $Y_\Delta$.

We are going to give a closed formula for this second term, that will be derived by means of an actual computation of the cycle $V$. To achieve this, we need to 
 introduce some more notation.

With the notation of Theorem \ref{intro GK DNC}, let us write the divisor $Y_\Delta$ (defined as the inverse image by $g$ of $\Delta$ seen as a reduced subscheme of $C$) as follows:
$$Y_\Delta = \sum_{i \in I} m_i D_i,$$
where $I$ is a finite set, and where the  $m_i$ are positive integers and  the $D_i$ are pairwise distinct connected non-singular divisors in $Y$.

The set $I$ may be written as the disjoint union:
$$I = \bigcup_{x \in \Delta} I_x,$$
where, for every $x \in \Delta,$ $I_x$ denotes the non-empty subset of $I$ defined by:
$$I_x := \big\{ i \in I \mid g(D_i) = \{x \} \big\}.$$

Let $\prec$ be a total order on the set $I$.

For every integer $r \geq 1$, let $D^r$ be the subscheme of codimension $r$ of $Y$ defined as the union of all the intersections of $r$ distinct components $(D_i)$:
$$D^r := \bigcup_{J \subset I, |J| = r} \bigcap_{i \in J} D_i.$$
For every $i$ in $I$, let:
$$i_{D_i} : D_i \longrightarrow Y$$
be the inclusion map, and 
$$\cN_i := \cN_{D_i} Y$$
be the normal line bundle to  $D_i$ in $Y$.

Observe also that, for every $i \in I,$ the scheme:
$$D_i \cap D^2 = \bigcup_{j \in I, j \neq i} D_{ij}$$
is a divisor with strict normal crossings in $D_i$. Similarly, for every $(i,j) \in I^2$ such that $i \prec j,$ the scheme:
$$D_{ij} \cap D^3 = \bigcup_{k \in I -\{i,j\}} D_{ij} \cap D_k$$
is a divisor with strict normal crossings in $D_{ij}$.  We shall also use the notation:
$$\Dcirc_i:= D_i - D_i \cap D^2 \quad \mbox{ and } \quad \Dcirc_{ij}:= D_{ij} - D_{ij} \cap D^3.$$

\begin{theorem}
\label{intro GK DNC with localized terms}
Under the assumptions of Theorem \ref{intro GK DNC}, using the above notation, the following equality holds:
\begin{equation}\label{eq: intro GK DNC with localized terms}
  \sum_{n=0}^{2(N-1)} (-1)^{n-1} \hgt_{GK, -}(\H^n(Y_\eta / C_\eta))
= \frac{1}{12} \int_Y c_1(\omega_{Y/C}^{1\vee}) c_{N-1}(\omega_{Y/C}^{1\vee}) + \sum_{x \in \Delta} \alpha_x,
\end{equation}
where for every $x$ in $\Delta$, $\alpha_x$ is the rational number given by the following two expressions\footnote{By $\chi_{\mathrm{top}}$, we denote the topological Euler characteristic.}:
\begin{align}
&\alpha_x = \frac{N-1}{4} \sum_{i \in I_x} (m_i - 1) \chi_{\mathrm{top}}(\Dcirc_i) + \frac{1}{12} \sum_{\substack{(i,j) \in I_x^2, \\ i \prec j}} (3 - m_i/m_j - m_j/m_i) \chi_{\mathrm{top}}(\Dcirc_{ij})  \label{intro:geommult}\\
&= \frac{1}{12} \sum_{i \in I_x} \Big[3 (N-1) (m_i - 1) \chi_{\mathrm{top}}(\Dcirc_i) + \int_{D_i} c_1(\cN_i) c_{N-2}(\Omega^{1 \vee} _{D_i} (\log D_i \cap D^2)) \Big] + \frac{1}{4} \sum_{\substack{(i,j) \in I_x^2, \\ i \prec j}} \chi_{\mathrm{top}}(\Dcirc_{ij}). \label{intro:Z/12}
\end{align}
\end{theorem}

The expression \eqref{intro:geommult} for $\alpha_x$ involves only the topology of the open strata $\Dcirc_i$ and $\Dcirc_{ij}$ of the fiber $g^{-1}(x)$ and the multiplicities $m_i$ of its components $D_i$. The expression \eqref{intro:Z/12} makes clear that $\alpha_x$ belongs to $(1/12)\Z$.

Observe that, when the divisor $Y_\Delta$ is reduced, that is when the degenerations of the fibers of $g$ are semistable with smooth components, the variations of Hodge structures   $\H^n(Y_\eta / C_\eta)$ have unipotent local monodromy, and formula \eqref{eq: intro GK DNC with localized terms} takes the simpler form:
\begin{equation}\label{eq: intro GK DNC with localized terms semistable}
  \sum_{n=0}^{2(N-1)} (-1)^{n-1} \hgt_{GK}(\H^n(Y_\eta / C_\eta))
= \frac{1}{12} \int_Y c_1(\omega_{Y/C}^{1\vee}) c_{N-1}(\omega_{Y/C}^{1\vee}) +  \frac{1}{12} \chi_{\mathrm{top}}(D^2 - D^3).
\end{equation}

In this semistable situation, using the identification of the lower Deligne extension of the cohomology and the logarithmic de Rham cohomology (see Proposition \ref{lower Deligne extension and logarithmic cohomology} below), we also have an isomorphism of line bundles over $C$:
$$\GK_{C}(\H^1(Y_\eta / C_\eta)) \lrasim \det g_\ast \omega^1_{Y/C}.$$
Moreover, using the unipotence of the local monodromy, Poincar\'e duality and the existence of an integral structure on the de Rham cohomology vector bundles as in Proposition \ref{GK Poincare duality}, we have an isomorphism of line bundles over $C$, up to 2-torsion:
$$\GK_{C}(\H^1(Y_\eta / C_\eta)) \simeq \GK_{C}(\H^{2N-3}(Y_\eta / C_\eta)).$$

Under this assumption of semistable reduction, when $N=3$, we also have the relations:
$$\hgt_{GK}(\H^1(Y_\eta / C_\eta))= \hgt_{GK}(\H^3(Y_\eta / C_\eta))= \deg g_\ast \omega^1_{Y/C},$$ and therefore
the equality \eqref{eq: intro GK DNC with localized terms semistable} becomes an expression for the Griffiths height of the variation of Hodge structures $\H^2(Y_\eta / C_\eta)$ associated to a family of surfaces with semistable degenerations:

\begin{corollary} In the situation of Theorem \ref{intro GK DNC}, when $N=3$ and when the degenerations of the fibers of $g$ are semistable with smooth components, the following equality holds:
\begin{equation}\label{surfacessemistable}
\hgt_{GK} (\H^2(Y_\eta / C_\eta))  = 2 \deg g_\ast \omega^1_{Y/C} + \frac{1}{12} \int_Y c_1(\omega^1_{Y/C})  c_2 (\omega^1_{Y/C}) - \frac{1}{12} \chi_{\mathrm{top}} (D^2 - D^3). 
\end{equation}
\end{corollary}

When the irregularity of the smooth fibers of $g$ is zero, then the direct image $ g_\ast \omega^1_{Y/C}$ vanishes, and formula \eqref{surfacessemistable} becomes:
\begin{equation}\label{surfacessemistablebis}
\hgt_{GK} (\H^2(Y_\eta / C_\eta))  = \frac{1}{12} \int_Y c_1(\omega^1_{Y/C})  c_2 (\omega^1_{Y/C}) - \frac{1}{12} \chi_{\mathrm{top}} (D^2 - D^3), 
\end{equation}
and therefore involves only characteristic classes and the topology of the fibers of $g$ with bad reduction.\footnote{The right-hand side of \eqref{surfacessemistablebis} is easily seen to be also $- (1/12) \int_Y c_1([T_g]) c_2([T_g])$.}

\subsection{Application to pencils of projective varieties 
with only non-degenerate critical points}

We shall notably apply Theorem \ref{intro GK DNC with localized terms} in the case where $Y$ is the blow-up of a finite number of non-degenerate critical points of a morphism
$$f : H \lra C.$$
In this case, using a result of Eriksson, Freixas and Mourougane concerning the comparison of Deligne extensions and the so-called ``elementary exponents'' of Hodge bundles in the case 
where $f$ admits only non-degenerate critical points (\cite[Prop. 3.10]{EFM21}), and computations of characteristic classes, we shall obtain the following result.

\begin{theorem}
\label{intro GK critical points} 
Let $C$ be a connected smooth projective complex curve with generic point $\eta$, $H$ be a smooth projective $N$-dimensional complex scheme, and let:
$$
f : H \lra C
$$
be a  morphism of complex schemes. Let us assume that there exists a finite subset $\Sigma$ in $H$ such that $f$ is smooth on $H - \Sigma$ and admits a non-degenerate critical point at every point of $\Sigma$.

Then the following equalities hold:
\begin{equation}
\label{GK- points doubles}
  \sum_{n=0}^{2(N-1)} (-1)^{n-1} \hgt_{GK, -}(\H^n(H_\eta / C_\eta))
= \frac{1}{12} \int_H  c_1([\Omega^1_{H/C}]^\vee) c_{N-1}([\Omega^1_{H/C}]^\vee)  + u^-_N \, \vert\Sigma\vert,
\end{equation}
and
\begin{equation}
\sum_{n=0}^{2(N-1)} (-1)^{n-1} \hgt_{GK, +}(\H^n(H_\eta / C_\eta))
= \frac{1}{12} \int_H  c_1([\Omega^1_{H/C}]^\vee) c_{N-1}([\Omega^1_{H/C}]^\vee)  + u^+_N \, \vert\Sigma\vert, \label{GK+ points doubles}
\end{equation}
where $u^-_N$ and $u^+_N$ are the rational numbers defined by: 
\begin{equation*}
u^-_N := 
\begin{cases}
 (5N-3)/24 & \text{if $N$ is odd} \\
 N/24 & \text{if $N$ is even},
\end{cases}
\end{equation*}
and:
\begin{equation*}
u^+_N := 
\begin{cases}
 -(7N-9)/24 & \text{if $N$ is odd} \\
 N/24 & \text{if $N$ is even}.
\end{cases}
\end{equation*}
\end{theorem}

\section[The Griffiths height of a pencil of hypersurfaces]{The Griffiths height of the middle-dimensional cohomology of a pencil of hypersurfaces}

\subsection{Families of ample hypersurfaces in a smooth pencil}

Applying Theorem  \ref{intro GK critical points} to the situation where $H$ is an hypersurface in some smooth $C$-scheme $X$, and using  the weak Lefschetz theorem (see for instance \cite[Theorem 1.29]{Voisin03}), we obtain the following somewhat technical result.

\begin{proposition}
\label{intro technical result}
Let $C$ be a connected smooth projective complex curve with generic point $\eta$, $X$ be a smooth projective complex scheme of pure dimension $N+1$, and let
$$\pi : X \lra C$$
be a smooth surjective morphism of complex schemes. Let $H$ be a non-singular hypersurface in $X$ such that the morphism
$$\pi_{| H} : H \lra C$$
is flat\footnote{or equivalently, is non-constant on every connected component of $H$, or has fibers of pure dimension $N-1$.} and has a finite set $\Sigma$ of critical points, all of which are non-degenerate.

If we denote by $L$ the line bundle $\cO_X(H)$ on $X$,  then the following equality holds:
\begin{equation}\label{SigmaLX}
\vert \Sigma \vert = \int_X (1-c_1(L))^{-1} c(\Omega^1_{X/C}),
\end{equation}

If moreover the line bundle $L$ is ample relatively to~$\pi$, 
then the following equalities hold:
\begin{multline}\label{GK+XL}
\hgt_{GK, +}(\H^{N-1}(H_\eta / C_\eta))   =\hgt_{GK}(\H^{N-1}(X/C)) + \hgt_{GK}(\H^{N+1}(X/C)) - \hgt_{GK}(\H^N(X/C)) \\
+ \frac{1}{12} \int_X \big[(1-c_1(L))^{-1}c_1(\Omega^1_{X/C}) c(\Omega^1_{X/C}) 
-  c_1(L)c_N(\Omega^1_{X/C})\big] + v^+_N\,  \vert \Sigma\vert,
\end{multline}
and:
\begin{multline}\label{GK-XL}
\hgt_{GK, -}(\H^{N-1}(H_\eta / C_\eta))   =\hgt_{GK}(\H^{N-1}(X/C)) + \hgt_{GK}(\H^{N+1}(X/C)) - \hgt_{GK}(\H^N(X/C)) \\
+ \frac{1}{12} \int_X \big[(1-c_1(L))^{-1}c_1(\Omega^1_{X/C}) c(\Omega^1_{X/C}) 
-  c_1(L)c_N(\Omega^1_{X/C})\big] + v^-_N\,  \vert \Sigma\vert,
\end{multline}
where:
\begin{equation*}
v^+_N := 
\begin{cases}
 7(N-1)/24 & \text{if $N$ is odd} \\
 (N+2)/24 & \text{if $N$ is even},
\end{cases}
\end{equation*}
and:
\begin{equation*}
v^-_N := 
\begin{cases}
 - 5(N-1)/24 & \text{if $N$ is odd} \\
 (N+2)/24 & \text{if $N$ is even}.
\end{cases}
\end{equation*}
\end{proposition}

In the expressions in the right-hand side of \eqref{SigmaLX}, \eqref{GK+XL}, and \eqref{GK-XL}, we denote by $c(\Omega^1_{X/C})$ the total Chern class of the vector bundle $\Omega^1_{X/C}$ of rank $N$:
$$c(\Omega^1_{X/C}) := 1 + c_1(\Omega^1_{X/C})+ \dots + c_N(\Omega^1_{X/C}).$$ These expressions may be expanded in terms of $c_1(L),$ $c_1(\Omega^1_{X/C}),$..., $c_N(\Omega^1_{X/C})$. For instance, \eqref{GK+XL} may also be written:
\begin{multline*}
\hgt_{GK, +}(\H^{N-1}(H_\eta / C_\eta))   =\hgt_{GK}(\H^{N-1}(X/C)) + \hgt_{GK}(\H^{N+1}(X/C)) - \hgt_{GK}(\H^N(X/C))  \\
+ \frac{1}{12} \sum_{j=0}^N \int_X c_1(L)^j c_1(\Omega^1_{X/C}) c_{N-j}(\Omega^1_{X/C}) 
-\frac{1}{12} \int_X c_1(L) c_N(\Omega^1_{X/C})  + v^+_N\,  \sum_{j=1}^{N+1} \int_X c_1(L)^j c_{N+1-j}(\Omega^1_{X/C}).
\end{multline*}

\subsection{Pencils of hypersurfaces in the projective space}

As  a first application of  Proposition \ref{intro technical result}, we may study  pencils of projective hypersurfaces. 

Let $E$ be a vector bundle of rank $N+1$ over a connected smooth projective complex curve $C$, and let
$$\pi: \PP(E):= \mathrm{Proj}\,  S^\bullet E^\vee \lra C$$
be the associated projective bundle. We shall denote by $\cO_E(-1)$ the tautological rank one subbundle of $\pi^\ast E$,  and by $\cO_E(1)$ its dual.  An horizontal hypersurface in the projective bundle $\PP(E)$ is an effective Cartier divisor $H$ in $\PP(E)$ such that the morphism
$$\pi_{\mid H} : H \lra C$$
is flat.  
Then its fibers
$$H_x := \pi_{\mid H}^{-1} (x), \quad x \in C$$
are hypersurfaces in the projective spaces $\PP(E_x) \simeq \PP^N(\C)$. Their degree $d$ is independent of $x \in C$, and defines the \emph{relative degree} of the horizontal hypersurface.

We introduce the \emph{intersection-theoretic height} of an horizontal hypersurface $H$. It is defined as the rational number:
\begin{equation}\label{hintdef}
\mathrm{ht}_{int}(H/C) := \int_{\PP(E)} c_1(\cO_E(1))^N \cap [H] + d N \mu(E),
\end{equation}
where $$\mu(E) := \deg_C E /\rk E = \deg (c_1(E) \cap [C]) /(N+1)$$
denotes the slope of the vector bundle $E$ over $C$.\footnote{The additive normalization by $ d N \mu(E)$ in the right-hand side of \eqref{hintdef} ensures that $\mathrm{ht}_{int}(H/C)$ is unchanged when the vector bundle $E$ is replaced by $E \otimes L$ for some line bundle $L$ over $C$: it depends only on $H$  as a subscheme of the projective bundle $\PP:= \PP(E)$, and not on the actual choice of a vector bundle $E$ such that $\PP \simeq \PP(E)$. See Remark \ref{invariance htint} below.}

A version of this height, in the more general setting of pencils of cycles in projective space, is already implicit in \cite{Bost94}; indeed \cite[Theorem III]{Bost94} simply states that when the Chow point attached to the generic fiber of such a pencil is semistable (in the sense of geometric invariant theory), then its height is non-negative.

The arithmetic analog of this height for cycles, when the function field of the curve is replaced by a number field, also appears implicitly in \cite[Theorem I]{Bost94}. It then appears explicitly both in \cite{Zhang96} under the notation $h_{\mathcal{E}}(X)$ and (up to multiplication by a constant factor) in \cite[2.1]{Bost96} under the notation $h_{\mathrm{norm}}(Z, \overline{E})$ (the latter actually introduces it as a special case of a normalized height $h_{\mathrm{norm}}(\mathscr{X}, \overline{\mathscr{L}})$ attached to a projective 
variety $\mathscr{X}$ over the ring of integers of a number field, equipped with a hermitian line bundle $\overline{\mathscr{L}}$).

The relation between this height and geometric invariant theory is studied in \cite[Proposition 4.2]{Zhang96} and \cite[Proposition 3.1]{Bost96}, and extended to arbitrary homogeneous representations of a general linear group in \cite{Gasbarri00} and \cite{Maculan17}.

We will show the following theorem.

\begin{theorem}
\label{intro GK hyp P(E)}
Let $C$ be a connected smooth  projective complex curve with generic point $\eta$, $E$ a vector bundle of rank~$N+1$ over $C$, and $H \subset \PP(E) $ an horizontal hypersurface of relative degree $d$, smooth over $\C.$  

If $\pi_{\mid H}$ has only a finite number of critical points, all of which are non-degenerate, then the cardinality of the set $\Sigma$ of critical points satisfies:
\begin{equation}\label{cardSigmahypproj}
\vert \Sigma \vert = (N+1) (d-1)^N \, \mathrm{ht}_{int}(H/C).
\end{equation}

Moreover, under the same hypothesis, the following equalities hold:
$$\mathrm{ht}_{GK,+}(\H^{N-1}(H_\eta/C_\eta))
= F_+(d,N) \, \mathrm{ht}_{int}(H/C),$$
$$\mathrm{ht}_{GK,-}(\H^{N-1}(H_\eta/C_\eta))
= F_-(d,N) \,  \mathrm{ht}_{int}(H/C),$$
and:
$$\mathrm{ht}_{GK,stab}(\H^{N-1}(H_\eta/C_\eta))
= F_{stab}(d,N) \,  \mathrm{ht}_{int}(H/C),$$
where $F_+(d,N)$, $F_-(d,N)$ and $F_{stab}(d,N)$ are the elements of $(1/12) \Z$ given when  $N$ is odd by:
$$F_{stab}(d,N) := \frac{N+1}{24 d^2} \left[ (d-1)^N  (d^2 N - d^2 - 2 d N - 2 )+ 2 (d^2-1) \right] ,$$
$$F_+(d,N) := F_{stab}(d,N) + \frac{(N+1) (N-1) (d-1)^N}{4},$$
and:
$$F_-(d,N) := F_{stab}(d,N) - \frac{(N+1) (N-1) (d-1)^N}{4},$$
and when $N$ is even by:
$$F_+(d,N) = F_-(d,N) = F_{stab}(d,N) := \frac{N+1}{24 d^2} \left [ (d-1)^N  (d^2 N + 2 d^2 - 2 d N - 2) - 2 (d^2-1) \right ]. $$
\end{theorem}

The expression \eqref{cardSigmahypproj} does not seem to appear explicitly in the literature. However it is a simple consequence of various known  results concerning the discriminant of homogeneous polynomials in $N+1$ variables; see \ref{Disc hyp P(E)} below. 

From Theorem \ref{intro GK hyp P(E)}, we can also deduce similar formulas for the Griffiths heights of the \emph{primitive} middle-dimensional cohomology. 

Indeed, it follows from Lefschetz weak theorem that for every integer $n < N-1$, the VHS $\H^n(H_{C - \Delta}/C - \Delta)$ on $C - \Delta$ (where $\Delta$ denotes the set of critical values of the morphism $\pi_{\mid H}$) is trivial. Namely its underlying flat vector bundle is trivial, and its Hodge filtration also. Consequently its Griffiths line bundle is trivial. 

Using the Lefschetz decomposition of the VHS $\H^{N-1}(H_{C - \Delta} / C - \Delta)$ and the easily checked fact that the Griffiths line bundle of a direct sum of VHS is isomorphic to the tensor product of their Griffiths line bundles, we obtain the equalities of rational numbers:
$$\mathrm{ht}_{GK,\pm}(\H^{N-1}_{prim}(H_\eta/C_\eta))
= \mathrm{ht}_{GK,\pm}(\H^{N-1}(H_\eta/C_\eta)),$$
and:
$$\mathrm{ht}_{GK,stab}(\H^{N-1}_{prim}(H_\eta/C_\eta))
= \mathrm{ht}_{GK,stab}(\H^{N-1}(H_\eta/C_\eta)),$$
hence, under the hypothesis of Theorem \ref{intro GK hyp P(E)}:
$$\mathrm{ht}_{GK,\pm}(\H^{N-1}_{prim}(H_\eta/C_\eta)) 
= F_\pm(d,N) \, \mathrm{ht}_{int}(H/C),$$
and:
$$\mathrm{ht}_{GK,stab}(\H^{N-1}_{prim}(H_\eta/C_\eta))
= F_{stab}(d,N) \, \mathrm{ht}_{int}(H/C).$$

Let us point out that 
the intersection-theoretic height $\hgt_{int}(H/C)$ is non-negative when $d \geq 2$. This follows from equality \eqref{cardSigmahypproj}. The constants $F_{stab}(d,N)$ and $F_{+}(d,N)$ are easily checked to be non-negative, as predicted by Peters' inequality  \eqref{Peters ineq}. 

However, for any given value of $d \geq 2,$ the constant $F_{-}(d,N)$ is negative when $N$ is large and odd, and the height $\mathrm{ht}_{GK,-}(\H^{N-1}_{prim}(H_\eta/C_\eta))$ may be negative.

Let us also indicate that, in \cite[Annexe A]{MordantMem23}, we establish transversality results which imply that the assumptions on the pencil of hypersurfaces in Theorem \ref{intro GK hyp P(E)} --- to be smooth over $\C$ and to have singular fibers with at worst singular double points --- are satisfied in general.

Namely, in \cite[Th\'eor\`eme A.1.3]{MordantMem23}, we establish the following result:

\begin{theorem} Let $C$ be a connected smooth projective complex curve of genus $g$, let $E$ be a vector bundle of positive rank over $C$, $M$ a line bundle over $C$, and $d$ a positive integer.
Let moreover 
$$\pi : \PP(E) := \Proj_C (S^\bullet E^\vee) \lra C$$
denote the projective bundle associated to $E$, and  $\mu_{\max}(E)$ the maximal slope of $E$.

If the degree of $M$ satisfies the lower bound:
$$\deg_C M > 2g-1 + d \, \mu_{\max}(E),$$
then there exists a non-empty Zariski open subset $U''$ of the vector space of sections $H^0(\PP(E), \cO_E(d) \otimes \pi^\ast M)$ such that, for every $s$ in $U''$, the hypersurface $H_s$ of $\PP(E)$ defined by the vanishing of $s$ is smooth over $\C$, the critical points of the restriction: 
$$\pi_{\mid H_s} : H_s \lra C$$
are non-degenerate, and every fiber of $\pi_{\mid H_s}$ contains at most one critical point.
\end{theorem}

\subsection{Linear pencils of hypersurfaces}

Another case in which we can apply Proposition \ref{intro technical result} is the case of linear pencils of hypersurfaces.

Let $V$ be a connected smooth projective complex scheme of pure dimension $N\geq 1$, and let $H$ be a non-singular hypersurface in $V \times \PP^1$ such that the morphism 
$$
\mathrm{pr}_{1 | H} : H \lra V
$$
is dominant, or equivalently surjective, and let $\delta$ be its degree.
The line bundle $\cO_{V\times \PP^1}(H)$ on $V\times \PP^1$ is isomorphic to the line bundle $$
\mathrm{pr}_1^\ast M \otimes \mathrm{pr}_2^\ast \cO_{\PP^1}(\delta),
$$
where $M$ denotes some line bundle over $V$, which is unique up to isomorphism.

\begin{proposition}
\label{intro linear pencils of hypersurfaces}
With the above notation, 
let us assume that the morphism
$$
\mathrm{pr}_{2 | H} : H \lra \PP^1
$$
is surjective, and has a finite set $\Sigma$ of critical points, all of which are non-degenerate.

Then the cardinality of $\Sigma$ satisfies:
$$
\vert \Sigma \vert = \delta \int_V (1 - c_1(M))^{-2} c(\Omega^1_V).
$$

Furthermore, if the line bundle $M$ on $V$ is ample, the following equalities of integers hold:\footnote{We use the notation introduced at the end of \ref{PetersGriffitsCohomDef} with $C := \PP^1$.}
\begin{equation*}
\hgt_{GK, +}(\H^{N-1}(H_\eta / \PP^1_\eta) )
= \frac{\delta}{12} \int_V (1 - c_1(M))^{-2} c_1(\Omega^1_V) c(\Omega^1_V)  
 +  \frac{(-1)^{N+1}\delta}{12} \chi_{\mathrm{top}}(V)
 +  v_N^+ \vert\Sigma\vert
 \end{equation*}
 and:
 \begin{equation*}
\hgt_{GK, -}(\H^{N-1}(H_\eta / \PP^1_\eta) )
= \frac{\delta}{12} \int_V (1 - c_1(M))^{-2} c_1(\Omega^1_V) c(\Omega^1_V) 
 +  \frac{(-1)^{N+1}\delta}{12} \chi_{\mathrm{top}}(V)
 +  v_N^-  \vert\Sigma\vert,
\end{equation*}
where $v_N^+$ and $v_N^-$ are the rational numbers defined in Proposition \ref{intro technical result}.
\end{proposition}
\subsection{Lefschetz pencils}\label{Lefschetz pencils Intro}

Proposition \ref{intro linear pencils of hypersurfaces} applies notably to Lefschetz pencils. 

Let $V$ be a connected smooth projective complex scheme of pure dimension $N\geq 1$, embedded into some projective space $\PP^r$, of dimension $r\geq \max(N,2).$

Let $ \Lambda$ be a projective subspace of dimension $r-2$ in $\PP^r$ that intersects $V$ transversally, and let $P \subset \PP^{r\vee}$ be the projective line in the dual projective space $\PP^{r\vee}$ corresponding to $\Lambda$ by projective duality.

Let us denote by: 
$$\nu: \widetilde{\PP}^r_\Lambda \lra \PP^r$$
the blow-up of $\Lambda$ in $\PP^r$. If $\widetilde{V}$ denotes the proper transform of $V$ by $\nu$, the restriction:
$$\nu_{\mid \widetilde{V}} : \widetilde{V} \lra V$$
may be identified with the blow-up in $V$ of $\Lambda \cap V$, which is smooth of dimension $r-2$.

The projection of center $\Lambda$:
$$\PP^r - \Lambda \lra P$$
extends to a smooth morphism of complex schemes:
$$p: \widetilde{\PP}^r_\Lambda \lra P.$$
Recall that the pencil of hyperplanes in $\PP^r$ containing $\Lambda$ --- in other words the pencil of hyperplanes defined by $P$ --- is said to be a Lefschetz pencil with respect to the subvariety $V$ of $\PP^r$ when the morphism:
$$p_{\mid \widetilde{V}} : \widetilde{V} \lra P$$
 has a finite set $\Sigma$ of critical points, all of which are non-degenerate, and when the restriction
$$
p_{| \Sigma} : \Sigma \lra P
$$
is an injective map.

As a simple instance of Proposition \ref{intro linear pencils of hypersurfaces}, we shall establish the  following result, which notably applies to Lefschetz pencils:
 
\begin{corollary}
\label{intro Lefschetz pencils}
With the above notation, let us assume that the morphism:
$$p_{\mid \widetilde{V}} : \widetilde{V} \lra P$$
 has a finite set $\Sigma$ of critical points, all of which are non-degenerate. 

Then the cardinality of $\Sigma$ satisfies:
\begin{equation}\label{KatzSGA}
\vert\Sigma\vert = \int_V (1 - c_1(\cO_V(1)))^{-2} c(\Omega^1_V),
\end{equation}
and the following equalities hold:
\begin{equation}
\hgt_{GK,+}(\H^{N-1}(\widetilde{V}_\eta / P_\eta) ) 
= \frac{1}{12} \int_V (1 - c_1(\cO_V(1)))^{-2} c_1(\Omega^1_V)c(\Omega^1_V)  
+  \frac{(-1)^{N+1}}{12} \chi_{\mathrm{top}}(V) 
+ v_N^+\vert\Sigma\vert, \label{intro GK Lef pen +}
\end{equation}
and:
\begin{equation}
\hgt_{GK,-}(\H^{N-1}(\widetilde{V}_\eta / P_\eta) ) 
= \frac{1}{12} \int_V (1 - c_1(\cO_V(1)))^{-2} c_1(\Omega^1_V)c(\Omega^1_V)  
+  \frac{(-1)^{N+1}}{12} \chi_{\mathrm{top}}(V) 
+ v_N^- \vert\Sigma\vert. \label{intro GK Lef pen -}
\end{equation}
\end{corollary}

The expression \eqref{KatzSGA} for the number of critical points in a Lefschetz pencil is established by Katz in \cite[Expos\'{e} XVII, cor. 5.6]{SGA7II}.\footnote{In \emph{loc. cit.}, the right-hand side of  \eqref{KatzSGA} is shown to be the product of the degree of the Gauss map and of the degree of the dual hypersurface associated to $V$.} 

\section{Griffiths heights and Calabi-Yau manifolds}
\label{GK CY}

We conclude this introduction by some observations about the special forms taken by  our formulas concerning Griffiths heights when applied to Calabi-Yau manifolds.  This sheds some light on the significance of the assumption of being Calabi-Yau in the investigation of BCOV invariants, which have already been  mentioned in \ref{subsub content} above. 

\subsection{Pencils of Calabi-Yau manifolds}

Let us consider a pencil of projective varieties:
$$g: Y \lra C$$
as in Theorems \ref{intro GK DNC} and \ref{intro GK DNC with localized terms}, and assume that the smooth fibers of the morphism $g$ are Calabi-Yau manifolds. In other terms, for every point $x$ in $C - \Delta$, we assume that the canonical line bundle:
$$
\det\Omega^1_{Y_x} \simeq (\det\omega^1_{Y/C})_{\mid Y_x}
$$
of the fiber $Y_x := g^{-1}(x)$ is trivial.

The coherent sheaf $g_\ast \det \omega^1_{Y/C}$ is torsion free on $C$, hence locally free. Over $C - \Delta$, it coincides with the relative Hodge bundle $H^{N-1,0}(Y_{C-\Delta}/C-\Delta)$, and therefore it defines a line bundle over $C$ that we shall denote by:
$$L:=  g_\ast \det \omega^1_{Y/C}.$$

The tautological ``evaluation morphism" between  line bundles over $Y$:
$$
g^\ast L \lra \det \omega^1_{Y/C}
$$
is an isomorphism on $Y - Y_\Delta$. If we denote by $V$ the divisor  in $Y$ defined by the vanishing of  this morphism,  it defines an isomorphism:
\begin{equation}
g^\ast L \lrasim (\det\omega^1_{Y/C}) (-V). \label{iso line CY}
\end{equation}
From the isomorphism \eqref{iso line CY}, we deduce an equality in $\CH^1(Y)$:
\begin{equation}\label{c1 V}
c_1(\omega^1_{Y/C}) = g^\ast c_1(L) + [V].
\end{equation}
Moreover the support of the divisor $V$ is contained in the support of $Y_\Delta$, and the divisor $V$ can be written:
\begin{equation}\label{divisors CY}
V = \sum_{i \in I} v_i D_i,
\end{equation}
where for every $i$ in $I$, $v_i$ is an integer. 

Using \eqref{c1 V} and  \eqref{divisors CY}, we obtain the equality of integers:
\begin{align}
\int_Y c_1(\omega^1_{Y/C})\, c_{N-1}(\omega^1_{Y/C})  &= \int_Y (g^\ast c_1(L) + [V])\, c_{N-1}(\omega^1_{Y/C}) , \nonumber \\
&= \int_Y g^\ast c_1(L)\, c_{N-1}(\omega^1_{Y/C}) + \sum_{i \in I} v_i \int_{D_i} c_{N-1}(\omega^1_{Y/C | D_i}), \nonumber \\
&= (\deg L) \int_{Y_\eta} c_{N-1}(\Omega^1_{Y_\eta}) + \sum_{i \in I} v_i \int_{D_i} c_{N-1}(\Omega^1_{D_i} (\log D_i \cap D^2)), \label{first eq CY pencil} \\
&= (-1)^{N-1}( \deg L)\,  \chi_{\mathrm{top}}(Y_\eta) + (-1)^{N-1} \sum_{i \in I} v_i \, \chi_{\mathrm{top}}(\Dcirc_i), \label{second eq CY pencil}
\end{align}
where in \eqref{first eq CY pencil}, we have used Proposition \ref{eq omega^1 restricted to a strata} below, and in \eqref{second eq CY pencil}, $\chi_{\mathrm{top}}(Y_\eta)$ and $\chi_{\mathrm{top}}(\Dcirc_i)$ denote respectively the Euler characteristic of a general fiber of the morphism $g$ and of $\Dcirc_i$. \footnote{See for instance Theorem \ref{chi_top complement of  DNC 2} below for the equality: $\chi_{\mathrm{top}}(\Dcirc_i) =(-1)^{N-1} \int_{D_i} c_{N-1}(\Omega^1_{D_i} (\log D_i \cap D^2)).$}

Consequently, inserting  \eqref{second eq CY pencil} in the right-hand side of  the equality \eqref{eq: intro GK DNC with localized terms} in Theorem \ref{intro GK DNC with localized terms}, we obtain the following expression for the alternating sum of the Griffiths heights of a  pencil of Calabi-Yau manifolds, which may be seen as a geometric version of the BCOV invariant: 
\begin{equation}
\sum_{n=0}^{2(N-1)} (-1)^{n-1} \hgt_{GK,-}(\H^n(Y_\eta / C_\eta)) = -\frac{1}{12} \deg(g_\ast \det \omega^1_{Y/C}) \, \chi_{\mathrm{top}}(Y_\eta) + \sum_{x \in \Delta} \beta_x, \label{eq pencil of Calabi Yau}
\end{equation}
where, for every $x$ in $\Delta$, $\beta_x$ is the rational number given by:
$$
\beta_x = \frac{1}{12} \sum_{i \in I_x} [3 (N-1) (m_i - 1) - v_i]\,  \chi_{\mathrm{top}}(\Dcirc_i) + \frac{1}{12} \sum_{(i,j) \in I_x^2, i \prec j} (3 - m_i/m_j - m_j/m_i) \,\chi_{\mathrm{top}}(\Dcirc_{ij}).
$$

Observe that, in the right-hand side of  \eqref{eq pencil of Calabi Yau}, the first term only depends on the general fiber of $Y$ and on the line bundle ``of modular forms" $L := g_\ast \det \omega^1_{Y/C}$, and the second term  only depends on the geometry of an infinitesimal neighborhood of the singular fibers.

When $N=3$ and when the smooth fibers of $g$ are K3 surfaces, and when the singular fibers are semistable\footnote{namely, when all the multiplicities $m_i$ equal 1.} with smooth components,  formula \eqref{eq pencil of Calabi Yau}
 becomes:
\begin{equation}\label{htGK K3}
\hgt_{GK}(\H^2(Y_\eta /C_\eta)) = 2 \deg g_\ast \omega_{Y/C} - \frac{1}{12} \chi_{\mathrm{top}} (D^2 - D^3) +  \frac{1}{12} \sum_{i \in I}v_i \, \chi_{\mathrm{top}}(\Dcirc_i),
\end{equation}
where:
$$\omega_{Y/C} \simeq \det \omega^1_{Y/C}$$
denotes the relative dualizing line bundle of the morphism $g$.

Indeed, in this situation, the Griffiths heights in the left-hand side of \eqref{eq pencil of Calabi Yau} vanish if $n \neq 2,$ and: $$\chi_{\mathrm{top}}(Y_\eta) = 24;$$
see for instance \cite{Huybrechts16}. One may actually expect to obtain a more explicit expression of the last two terms in the right-hand side of \eqref{htGK K3} when the degenerations of $Y/C$ are Kulikov models; see \cite[Chapter 6, \S 5]{Huybrechts16}.

\subsection{Lefschetz pencils on a  Calabi-Yau manifold embedded in a projective space}

Let us adopt the notation of Corollary \ref{intro Lefschetz pencils}, and let us assume that $V$ is a Calabi-Yau manifold, namely that the line bundle $\det \Omega^1_V$ on $V$ is trivial. 

Under this assumption, we have the equality in $\CH^1(V)$:
$$
c_1(\Omega^1_V) = 0,
$$
and using it in equalities \eqref{intro GK Lef pen +} and \eqref{intro GK Lef pen -} of Corollary \ref{intro Lefschetz pencils} yields the following expressions for the Griffiths heights of
the middle-dimensional cohomology of a Lefschetz pencil on a Calabi-Yau manifold:
$$
\hgt_{GK,+}(\H^{N-1}(\widetilde{V} _\eta / P_\eta)) = \frac{(-1)^{N-1}}{12} \chi_{\mathrm{top}}(V) + v_N^+ |\Sigma|,
$$
and:
$$
\hgt_{GK,-}(\H^{N-1}(\widetilde{V} _\eta / P_\eta)) = \frac{(-1)^{N-1}}{12} \chi_{\mathrm{top}}(V) + v_N^- |\Sigma|.
$$

In particular, these Griffiths heights only depend on some basic invariants of $V$ and its Lefschetz pencil, namely the Euler characteristic of $V$ and  the number of critical points of the  pencil.

\chapter[Steenbrink's theory and elementary exponents]{Variations of Hodge structures associated to pencils of complete varieties with SNC degenerations:  Steenbrink's theory and elementary exponents}\label{SteenbrinkElExp}
\section{Steenbrink's theory and logarithmic Hodge bundles} \label{Steenbrink theory and logarithmic Hodge bundles}

Let $C$ be a connected smooth projective complex  curve, let $\Delta \subset C$ be a  finite subset, and let: $$\Cc := C - \Delta,$$ be its complement in $C$. Let also $Y$ be a smooth complex manifold, and let: 
$$g : Y \lra C$$
be a proper complex analytic map, that is relatively projective, locally (in the analytic topology) over $C$, and such that the complex analytic map $$g_{| Y_{\Cc} }: Y_{\Cc} := g^{-1}(\Cc) \lra \Cc$$ is smooth (that 
is, a complex analytic submersion). 

Furthermore, let us assume that the divisor 
$$Y_\Delta := g^\ast \Delta$$
is a (not necessarily reduced) divisor with strict normal crossings, i.e. it can be written as $\sum \limits_{i} a_i D_i$, where the $(D_i)_i$ are smooth distinct divisors that intersect transversally and where the $(a_i)_i$ are positive integers.

\begin{definition}[see for instance \cite{Katz71}] \label{relative logarithmic de Rham cohomology}
Let $\Omega^1_Y (\log Y_\Delta )$ be the (analytic) sheaf of differentials with logarithmic singularities.
The sheaf on $Y$:
$$\Omega^1_{Y/C} (\log Y_\Delta ) := \Omega^1_Y (\log Y_\Delta) / \mathrm{Im}(g^\ast \Omega^1_C (\log \Delta) \lra \Omega^1_Y (\log Y_\Delta))$$
is called the \emph{sheaf of relative differentials with logarithmic singularities}. By taking local coordinates near a point of $Y_\Delta$, one can check that this is a locally free sheaf.  Its rank on any connected component of $Y$ is precisely the relative dimension of $g$ on that component.

We also define the \emph{relative logarithmic de Rham complex}:
$$\Omega^{\bullet}_{Y/C} (\log Y_\Delta) := \Exterior^{\bullet} \Omega^1_{Y/C} (\log Y_\Delta),$$
with the differential induced by the one on the complex $j_\ast \Omega^{\bullet}_{Y_\Cc}$ by restriction and quotient, where
$$j : Y_\Cc = Y - Y_\Delta \lra Y$$
denotes the inclusion morphism.

The \emph{relative logarithmic de Rham cohomology sheaf} in degree $n$ is defined as the analytic coherent sheaf on $C$:
$$\ccH^n_{\log}(Y/C):=  R^n g_\ast \Omega^{\bullet}_{Y/C} (\log Y_\Delta).$$
\end{definition}

For simplicity's sake, we will write $\omega^{\bullet}_{Y/C}$ for $\Omega^{\bullet}_{Y/C} (\log Y_\Delta)$ as in \cite{Kato89} and \cite{Illusie94}.  Note that these sheaves do not depend on the choice of $\Delta$ provided that $g$ is smooth over $C -\Delta$; non-singular fibers in $Y_\Delta$ would not contribute to them. 

Also note that we have a relative logarithmic Hodge-de Rham spectral sequence:
\begin{equation}\label{logarithmic Hodge-de Rham sequence}
E_1 ^{p,q} := R^q g_\ast \omega^p_{Y/C} \Rightarrow R^{p+q} g_\ast \omega^{\bullet}_{Y/C} .
\end{equation}

The restriction to $\Cc$ of this spectral sequence is the classical Hodge-de Rham spectral sequence of the smooth morphism:
$$g_{| Y_\Cc} : Y_\Cc \lra \Cc.$$
This sequence degenerates at the first page (see for instance \cite[Th. (5.5)]{Deligne68}).

Consequently, for every integer $n$, there is a filtration $(\cF^p)_{0 \leq p \leq n}$ of the vector bundle: $$\ccH^n(Y_\Cc/ \Cc):= R^n g_{\mid Y_\Cc \ast} \Omega^\bullet_{Y_\Cc/\Cc} \simeq R^n g_{| Y_\Cc \ast} \C \otimes_\C \cO_\Cc$$ on $\Cc$, such that for every integer $0 \leq p \leq n$, the subquotient $\cF^p/\cF^{p+1}$ is isomorphic to 
$$R^{n-p} g_{| Y_\Cc} \Omega^p_{Y_\Cc/\Cc},$$
and is locally free.

Finally, applying the functor $g_\ast$ to the short exact sequence of complexes of sheaves on $Y$:
$$0 \lra g^\ast \Omega^1_C (\log \Delta) \otimes \omega^{\bullet}_{Y/C} [-1] \lra \Omega^{\bullet}_Y (\log Y_\Delta) \lra \omega^{\bullet}_{Y/C} \lra 0,$$
we obtain a long exact sequence whose connecting homomorphism defines a logarithmic connection:
$$\nabla_{\log} : R^n g_\ast \omega^\bullet_{Y/C} \longrightarrow R^n g_\ast \omega^\bullet_{Y/C} \otimes_{\cO_C} \Omega^1_C (\log \Delta),$$
whose restriction to $\Cc$ is the Gauss-Manin connection.

The following results are simple consequences of Steenbrink's theory (see \cite{Steenbrink76, Steenbrink77}).

\begin{proposition} \label{lower Deligne extension and logarithmic cohomology}
For all $n \geq 0$, the lower Deligne extension to $C$ of the relative cohomology bundle $\ccH^n(Y_\Cc/ \Cc)$ with the Gauss-Manin connection is precisely the relative logarithmic de Rham cohomology $\ccH^n_{\log}(Y/C) := R^n g_\ast \omega^\bullet_{Y/C}$ equipped with the logarithmic connection $\nabla_{\log}$. In particular $\ccH^n_{\log}(Y/C)$ is locally free over $C$.

Also, denoting by $(a_i)_{i \in I}$ the multiplicities of the components of $Y_\Delta$, the eigenvalues of the residue of the Gauss-Manin connection are in the set $\{ \frac{j_i}{a_i} \; | \, i \in I, \, 0 \leq j_i \leq a_i - 1\}$, and the (counterclockwise) monodromy endomorphism at every point of $\Delta$ is quasi-unipotent.
\end{proposition}
\begin{proof}
By \cite[Th. (2.18)]{Steenbrink76}, for all $n \geq 0$, the sheaf $R^n g_\ast \omega^\bullet_{Y/C} $ is locally free on $C$. 

By \cite[VII]{Katz71} (applied to the bundle $M := \cO_{Y_\Cc}$ equipped with the connection given by the differential $d$, and to the trivial extension $\overline{M} = \cO_Y$ equipped with the differential $d$),  we obtain that the only possible eigenvalues of the residue in $\Delta$ of the logarithmic connection $\nabla_{\log}$ are in the set $\{ \frac{j_i}{a_i} \; | \, i \in I, \, 0 \leq j_i \leq a_i - 1\}$. In particular, all the eigenvalues have their real part in $[0, 1[$. 

Using the definition of the lower Deligne extension and the fact that the (counterclockwise) monodromy in a point $x \in \Delta$ is given by:
$$T_x = \exp (-2\pi  \sqrt{-1}  \Res _x \nabla_{\log} ),$$
this gives the result. 
\end{proof}

\begin{proposition} \label{logarithmic Hodge bundles}
The relative logarithmic Hodge-de Rham spectral sequence \eqref{logarithmic Hodge-de Rham sequence} degenerates at the first page, and for all $p,q \geq 0$, the sheaves $R^q g_\ast \omega^p_{Y/C}$ are locally free on $C$. Consequently, for all $(p,n)$ such that $0 \leq p \leq n$, the $p$-th subquotient of the Hodge filtration $\cF^\bullet _{\log}$ on the relative cohomology bundle $\ccH^n_{\log}(Y/C)$ given by the spectral sequence \eqref{logarithmic Hodge-de Rham sequence}satisfies:
$$\mathrm{gr}^p _{\cF^\bullet _{\log}} (R^n g_\ast \omega^{\bullet}_{Y/C}) = R^{n-p} g_\ast \omega^p_{Y/C}.$$
In particular, the vector bundles $\cF^p$ on $\Cc$ are algebraic vector subbundles of $\ccH^n (Y_{\Cc}/\Cc)$ endowed with the algebraic structure defined by means of its Deligne extension.\footnote{
This also follows from \cite[Theorem (4.13), a)]{Schmid73}. Observe that the algebraic structures defined by the upper and lower Deligne extensions of a vector bundle with connection on $\Cc$ actually coincide.
} 
\end{proposition}

\begin{proof}
By \cite[Th. (2.11)]{Steenbrink77}, for all $p,q \geq 0$, the sheaves $E_1 ^{p,q} := R^q g_\ast \omega^p_{Y/C}$ are locally free. 

The restriction of the spectral sequence \eqref{logarithmic Hodge-de Rham sequence} to $\Cc$ is the classical relative Hodge-de Rham sequence, which degenerates at the first page. This means that the differentials of the sequence 
\eqref{logarithmic Hodge-de Rham sequence} $\mathrm{d}_1 ^{p,q} : E_1 ^{p,q} \rightarrow E_1 ^{p+1, q}$ vanish on $\Cc$, hence everywhere by local freeness of the sheaves $E_1 ^{p,q}$. So the relative logarithmic Hodge-de Rham spectral sequence degenerates at the first page. 

The second statement follows by definition of the spectral sequence. 

For the third one, from the GAGA theorem applied on the projective manifold $C$, we have that the analytic vector bundles $\cF^p _{\log}$ are algebraic subbundles of the Deligne extension $R^n g_* \omega^\bullet_{Y/C}$ of~$\ccH^n(Y_\Cc / \Cc)$, hence their restrictions to $\Cc$ are  algebraic with respect to the same algebraic structure. 
\end{proof}

\begin{proposition}
\label{lower Griffiths and logarithmic cohomology}
For every integer $n \geq 0$, letting $\eta$ be the generic point of $C$, the lower Griffiths line bundle of the relative cohomology VHS $\H^n(Y_\eta/ C_\eta)$ exists and is given by: 
$$\GK_{C,-}(\H^n(Y_\eta/ C_\eta)) \simeq \bigotimes_{0 \leq p \leq n} (\det R^{n-p} g_\ast \omega^p_{Y/C})^{\otimes p}.$$
\end{proposition}
\begin{proof} Let $n$ be a non-negative integer. 
Using Proposition \ref{lower Deligne extension and logarithmic cohomology}, the lower Deligne extension of the vector bundle $\ccH^n(Y_\Cc/\Cc)$ on $\Cc$ equipped with the Gauss-Manin connection coincides with the relative logarithmic cohomology $\ccH^n_{\log} (Y/C) \simeq R^n g_\ast \omega^\bullet_{Y/C}$ equipped with the logarithmic connection $\nabla_{\log}$. 

Using Proposition \ref{logarithmic Hodge bundles}, for every integer $p \geq 0$, the vector bundle $\cF^p$ on $\Cc$ is an algebraic subbundle of the vector bundle $\ccH^n(Y_\Cc/\Cc)$ endowed with the algebraic structure defined by $R^n g_\ast \omega^\bullet_{Y/C}$, and its extension in the vector bundle $R^n g_\ast \omega^\bullet_{Y/C}$ is the vector bundle $\cF^p_{\log}$ given by the logarithmic Hodge-de Rham exact sequence.

Consequently, the lower Griffiths line bundle is given by:
\begin{align*}
    \GK_{C,-}(\H^n(Y_\eta/C_\eta)) &\simeq \bigotimes_{p \geq 0} (\det \cF^p_{\log}/\cF^{p+1}_{\log})^{\otimes p},\\
    &=  \bigotimes_{0 \leq p \leq n} (\det R^{n-p} g_\ast \omega^p_{Y/C})^{\otimes p} \textrm{ using Proposition \ref{logarithmic Hodge bundles}}. \qedhere
    \end{align*}
\end{proof}

\section{Elementary exponents of a degeneration}
\label{elementary exponents}

In this section, we recall the formalism  of the elementary exponents of a normal crossing degeneration, as developed in \cite[Sections 2.2 and 2.3]{EFM21}.

Let $\mathbb{D}$ be the unit complex disk with coordinate $t$, let $Y$ be a smooth complex manifold, and let
$$
g : Y \lra \mathbb{D}
$$
be a projective morphism such that the fiber $Y_0$ is a (not necessarily reduced) divisor with strict normal crossings.

Let us consider a \emph{semistable reduction diagram}, i.e. a commutative diagram of complex manifolds
\begin{equation} \label{semistable reduction diagram}
\xymatrix{  Y' \ar[r]^{\rho} \ar[d]^{g'} & Y \ar[d]^{g}\\
            \mathbb{D}' \ar[r]^{\sigma} & \mathbb{D}}
\end{equation}
where $\mathbb{D}'$ is a copy of the unit open disk with coordinate $t'$, $\sigma$ is the map sending $t'$ to $t'^l$ where $l$ is some integer, $Y'$ is a smooth $N$-dimensional complex manifold, $g'$ is a morphism such that the open subscheme $Y' - Y'_0$ of $Y'$ is isomorphic to the fiber product $(Y - Y_0) \times_{\mathbb{D}^\ast} \mathbb{D}'^\ast$, and such that $Y'_0$ is a \emph{reduced} divisor with strict normal crossings.

Such a diagram exists by the ``semistable reduction theorem'' (\cite[Chapter II]{KKMS73}).

For every pair of integers $(p,q)$, we can consider the coherent sheaf $R^q g_\ast \omega^p_{Y/\mathbb{D}}$ (resp. $R^q g'_\ast \omega^p_{Y'/\mathbb{D}'}$) on $\mathbb{D}$ (resp. $\mathbb{D}'$), which is locally free using Proposition \ref{logarithmic Hodge bundles}. Both locally free sheaves have the same rank, given by the Hodge number of a general fiber:
$$h^{p,q} := h^{p,q}(Y_\infty).$$

\begin{definition}[{\cite[Def. 2.5]{EFM21}}]
\label{def elementary exponents}
Let $p,q \geq 0$ be integers. With the above notation, the map of locally free sheaves on $\mathbb{D}'$ given by the pullback of forms:
$$
\sigma^\ast R^q g_\ast \omega^p_{Y/\mathbb{D}} \lra R^q g'_\ast \omega^p_{Y'/\mathbb{D}'}
$$
induces an isomorphism on $\mathbb{D}'^\ast$, so it is injective and its cokernel is an $\cO_{\mathbb{D}'}$-module of the form
$$
\bigoplus_{1 \leq j \leq h^{p,q}} \cO_{\mathbb{D}'}/(t'{} ^{b^{p,q}_j} \cO_{\mathbb{D}'}),
$$
where $(b^{p,q}_j)_{1 \leq j \leq h^{p,q}}$ is a family of integers, unique up to order.

The \emph{elementary exponents of the $(p,q)$-Hodge bundle} are the rational numbers:
$$
(\alpha^{p,q}_j)_{1 \leq j \leq h^{p,q}} := \left (\frac{b^{p,q}_j}{l} \right )_{1 \leq j \leq h^{p,q}} \quad \in \left  (\frac{1}{l} \Z \right )^{h^{p,q}}.
$$
We also define rational numbers by:
$$
b^{p,q} := \sum_{1 \leq j \leq h^{p,q}} b^{p,q} _j,
$$
$$
\alpha^{p,q} := \sum_{1 \leq j \leq h^{p,q}} \alpha^{p,q} _j = \frac{b^{p,q}}{l}.
$$
Note that these numbers do not depend on the choice of the coordinates $t'$,$t$.
\end{definition}

\begin{proposition}[{\cite[Lemma 2.4]{EFM21}}]
\label{domain elementary exponents}
For all integers $p,q$, the integers $(b_j^{p,q})_j$ are all in $\{0,...,l-1\}$. Consequently, the elementary exponents $(\alpha_j^{p,q})_j$ are all in $\{0, \frac{1}{l}, ..., \frac{l-1}{l} \}$.

Moreover, the elementary exponents do not depend on the choice of the semistable reduction diagram.
\end{proposition}

The elementary exponents are related to the eigenvalues of the (counterclockwise) monodromy.

Let $\infty$ be a point in $\mathbb{D}^\ast$, and $n \geq 0$ be an integer. As in \cite[(2.8)]{Steenbrink77}, we can equip the fiber 
$$
(R^n g_\ast \C)_\infty \simeq H^n(Y_\infty, \C)
$$
with a mixed Hodge structure. Let us denote by $F_\infty^\bullet$ the associated Hodge filtration, not to be confused with the classical Hodge filtration $(\cF^\bullet)_\infty$ on the same space.

Using \cite[Th. (2.13)]{Steenbrink77}, the semisimple part of the monodromy:
$$
T_s : H^n(Y_\infty) \lrasim H^n(Y_\infty)
$$
respects the mixed Hodge structure, hence the filtration $F_\infty^\bullet$.

\begin{proposition}[{\cite[Cor. 2.8]{EFM21}}] 
\label{elementary exponents and monodromy} 
For every pair of integers $(p,q)$ with $p+q = n$, letting $(\alpha^{p,q} _j)_j$ be the elementary exponents of the $(p,q)$-Hodge bundle, the complex numbers 
$$
(\exp(- 2i\pi \alpha^{p,q} _j) )_j
$$
are precisely the eigenvalues (counted with multiplicities) of the endomorphism $T_s$ acting on the subquotient~$F_\infty^p / F_\infty^{p+1}$.
\end{proposition}
This proposition yields a new proof that the elementary exponents do not depend on the choice of the semistable reduction diagram.

We can also use them to compare the determinant line bundles of the Hodge bundles.

\begin{proposition}
\label{elementary exponents and det of Hodge} 
With the same notation, the canonical isomorphism of line bundles on $\mathbb{D}'^\ast$:
$$
\sigma^\ast \det R^q g_\ast \Omega^p_{Y - Y_0 / \mathbb{D}^\ast} \lra  \det R^q g'_\ast \Omega^p_{Y' - Y'_0 / \mathbb{D}'^\ast} 
$$
can be extended into an isomorphism of line bundles on $\mathbb{D}'$:
$$
(\sigma^\ast \det R^q g_\ast \omega^p_{Y/\mathbb{D}}) \otimes  \cO_{\mathbb{D}'}(\{0\})^{\otimes l \alpha^{p,q} } \lrasim \det R^q g'_\ast \omega^p_{Y'/\mathbb{D}'}. 
$$
\end{proposition}

\begin{proof}
This follows from the multiplicativity of the determinant line bundle applied to the exact sequence of coherent sheaves on $\mathbb{D}'$:
$$
0 \lra \sigma^\ast R^q g_\ast \omega^p_{Y/\mathbb{D}} \lra R^q g'_\ast \omega^p_{Y'/\mathbb{D}'} \lra \bigoplus_{1 \leq j \leq h^{p,q}} \cO_{\mathbb{D}'}/(t'^{l \alpha^{p,q}_j} \cO_{\mathbb{D}'}) \lra 0. \qedhere
$$
\end{proof}

\section{Comparison of Griffiths line bundles and elementary exponents}\label{Comp Griff Elem}

Let $C$ be a connected smooth projective complex curve with generic point $\eta$, $\Delta$ be a finite subset of $C$, which we also see as a reduced divisor, $Y$ be a smooth $N$-dimensional complex scheme, and 
$$
g : Y \lra C
$$
be a projective morphism which is smooth over $\Cc := C - \Delta$, and such that $Y_\Delta$ is a divisor with strict normal crossings. As before, we shall denote:
$$Y_\Cc := g^{-1}(\Cc) = Y - Y_\Delta.$$

For every point $x$ in $\Delta$, let $l_x$ be an integer such that there exists an open disk $\mathbb{D}$ centered in $x$ with a coordinate $t$ that identifies it with the unit open disk, such that the restriction
$$g_{\mathbb{D}} : Y_{\mathbb{D}} := g^{-1}(\mathbb{D}) \lra \mathbb{D}$$
and the morphism
$$\sigma : \mathbb{D'} \lra \mathbb{D}, \quad t' \lra t'^{l_x}$$
where $\mathbb{D}'$ is a copy of the unit open disk with coordinate $t'$, can be inserted into a semistable reduction diagram of the form \eqref{semistable reduction diagram}.

Let $C'$ be a connected smooth projective complex curve with generic point $\eta'$, and
$$
\sigma : C' \lra C
$$
be a finite morphism satisfying the following condition:

\begin{equation}
\mbox{\emph{For every $x'$ in $\sigma^{-1}(\Delta)$, the ramification index $r_{x'}$ of $\sigma$ in $x'$ is a multiple of the integer $l_{\sigma(x')}$. }}\label{condition sigma}
\end{equation}

In particular, for every integer $n$, the monodromy of the bundle $\sigma^\ast \ccH^n(Y_\Cc/ \Cc)$ with the connection given by the pullback of the Gauss-Manin connection, is unipotent at every point of $\sigma^{-1}(\Delta)$. 

For $n$ an integer, let $\cV^n$ be the vector bundle with connection $\ccH^n(Y_\Cc/ \Cc)$ on $\Cc$, and $F^\bullet \cV^n$ be its Hodge filtration.

Let $\overline{\cV^n}_+$ (resp. $\overline{\cV^n}_-$) be its upper (resp. lower) Deligne extension to $C$.

We know that the vector bundle with connection $\sigma^\ast \cV^n$ on: $$\Cc' := C' - \sigma^{-1}(\Delta)$$ has unipotent local monodromy at the points of $\Delta' := \sigma^{-1}(\Delta),$ so we can consider its Deligne extension $\overline{\sigma^\ast \cV^n}$ to $C'$.

Using Proposition \ref{logarithmic Hodge bundles}, for every integer $p$, the vector subbundle $F^p \cV^n$ on $\Cc$ is an algebraic subbundle of the vector bundle $\cV^n$ endowed with the algebraic structure defined by means of its lower Deligne extension $\overline{\cV^n}_-$. As the algebraic structures defined by the upper and the lower Deligne extensions actually coincide, the vector subbundle $F^p \cV^n$ on $\Cc$ is also an algebraic subbundle of the vector bundle $\cV^n$ endowed with the algebraic structure defined by its upper Deligne extension $\overline{\cV^n}_+$. Consequently, we can consider the extension $\overline{F^p \cV^n}_+$ (resp. $\overline{F^p \cV^n}_-$) in the vector bundle $\overline{\cV^n}_+$ (resp. $\overline{\cV^n}_-$) of the vector subbundle $F^p \cV^n$ in $\cV^n$.

Similarly, we can consider the extension $\overline{\sigma^\ast F^p \cV^n}$ in $\overline{\sigma^\ast \cV^n}$, over $C'$, of the subbundle $\sigma^\ast F^p \cV^n$ in $\sigma^\ast \cV^n$ over $\Cc'$. 

For all integers $p,q$, and for every $x$ in $\Delta$, let $(\alpha^{p,q}_{j,x})_{1 \leq j \leq \mathrm{rg}(H^{p,q}(Y_\eta))}$ be the elementary exponents of the $(p,q)$-Hodge bundle at point $x$, and let $\alpha^{p,q}_x$ be their sum. 

Let us define a divisor on $C$ with rational coefficients by:
$$
A^{p,q} := \sum_{x \in \Delta} \alpha^{p,q}_x \{x\}.
$$

\begin{proposition}
\label{lower Deligne extensions and elementary exponents}
With the above notation, for all integers $p$, $n$, there is an isomorphism of line bundles on $C'$:
$$
\sigma^\ast \det(\overline{F^p \cV^n}_- / \overline{F^{p+1} \cV^n}_- ) \otimes \cO_{C'}(\sigma^\ast A^{p,n-p} ) \simeq \det(\overline{\sigma^\ast F^p \cV^n} / \overline{\sigma^\ast F^{p+1} \cV^n} )
$$
whose restriction to $\Cc'$ is the identity of the line bundle $\sigma^\ast \det(F^p \cV^n / F^{p+1} \cV^n)$. 
\end{proposition}

Note that even though $A^{p,n-p}$ is a divisor with only rational coefficients, its pullback by $\sigma$ necessarily has integral coefficients.

\begin{proof}
It is enough to define these isomorphisms locally around a point $x'$ in $\sigma^{-1}(\Delta)$. Let $r_{x'}$ be the ramification index of $\sigma$ at $x'$, and $x$ be the image of $x'$ by $\sigma$.

Let $\mathbb{D}$ be an open disk centered in $x$ with a coordinate $t$ that identifies $\mathbb{D}$ with the unit open disk,  and let $\mathbb{D}'$ be an open disk around $x'$ with a coordinate $t'$  that identifies $\mathbb{D}'$ with the unit open disk, such that  $\sigma(\mathbb{D}') \subseteq \mathbb{D}$ and that we have the equality of functions on $\mathbb{D}'$:
$$
\sigma^\ast t = t'^{r_{x'}}.
$$
Using  condition \eqref{condition sigma}, the integer $r_{x'}$ can be written $r' l_x$, where $r' \geq 1$ and $l_x$ is an integer such that we have a semistable reduction diagram
$$
\xymatrix{  (Y_\mathbb{D})'' \ar[r] \ar[d]^{(g_\mathbb{D})''} & Y_\mathbb{D} \ar[d]^{g_\mathbb{D}}\\
            \mathbb{D}'' \ar[r]^{\sigma'} & \mathbb{D}}
$$
where $\sigma'$ is the elevation to the power $l_x$.

Let 
$$
\tau : \mathbb{D}' \lra \mathbb{D}''
$$
be the elevation to the power $r'$, so that we have the equality of morphisms:
$$
\sigma' \circ \tau = \sigma_{| \mathbb{D}'}.
$$
Using Proposition \ref{lower Deligne extension and logarithmic cohomology}, and the fact that $(g_\mathbb{D})''$ is semistable, i.e. that the fiber $(Y_\mathbb{D})''_0$ is a reduced divisor with strict normal crossings, we obtain that for every $n \geq 0$, the monodromy of the vector bundle $R^n (g_\mathbb{D}) ''_{\mid  (Y_\mathbb{D})'' - (Y_\mathbb{D})''_0 \ast} \Omega^\bullet_{(Y_\mathbb{D})'' - (Y_\mathbb{D})''_0 / \mathbb{D''} - \{0\}}$, with the Gauss-Manin connection, is unipotent. Consequently, the upper and lower Deligne extensions on $\mathbb{D}''$ coincide and are compatible with base change by ramified morphisms.

Consequently, the extensions of the subquotients in the lower Deligne extensions are also compatible with base change. So using Proposition \ref{logarithmic Hodge bundles}, we obtain, for all integers $p,n$, an isomorphism of vector bundles on $\mathbb{D}'$:
\begin{equation}
   \tau^\ast R^{n-p} (g_\mathbb{D})''_\ast \omega^p_{(Y_\mathbb{D}) '' / \mathbb{D}''} \simeq \overline{\sigma^\ast F^p \cV^n}/\overline{\sigma^\ast F^{p+1} \cV^n} \label{first isomorphism comparison Deligne extensions}
\end{equation}
that extends the canonical isomorphism on $\mathbb{D}' - \{ x'\}$.

Using Proposition \ref{elementary exponents and det of Hodge}, we have an isomorphism of line bundles on $\mathbb{D}''$:
$$
(\sigma'^\ast \det R^{n-p} g_{\mathbb{D} \ast} \omega^p_{Y_\mathbb{D}/\mathbb{D}}) \otimes \cO_{\mathbb{D}''}(\{0\})^{\otimes l_x \alpha^{p,n-p}_x } \simeq \det R^{n-p} (g_\mathbb{D}) ''_\ast \omega^p_{(Y_\mathbb{D})'' / \mathbb{D}''},
$$
that extends the canonical isomorphism on $\mathbb{D}'' - \{0\}$. Pulling it back by $\tau$ and using the isomorphism \eqref{first isomorphism comparison Deligne extensions}, we get  an isomorphism of line bundles on $\mathbb{D}'$:
\begin{align*}
(\sigma^\ast \det R^{n-p} g_{\mathbb{D} \ast} \omega^p_{Y_\mathbb{D}/\mathbb{D}}) \otimes \cO_{\mathbb{D}'}(\{x'\})^{\otimes r_{x'} \alpha^{p,n-p}_x } &\simeq \tau^\ast \det R^{n-p} (g_\mathbb{D})''_\ast \omega^p_{(Y_\mathbb{D})'' / \mathbb{D}''} \\
&\simeq \det(\overline{\sigma^\ast F^p \cV^n}/\overline{\sigma^\ast F^{p+1} \cV^n} ).
\end{align*}

Using again Proposition \ref{logarithmic Hodge bundles}, this can be rewritten as an isomorphism of line bundles on $\mathbb{D}'$:
$$
\sigma^\ast \det(\overline{F^p \cV^n}_- / \overline{F^{p+1} \cV^n}_-) \otimes \cO_{\mathbb{D}'}(\{x'\})^{\otimes r_{x'} \alpha^{p,n-p}_x } \simeq \det (\overline{\sigma^\ast F^p \cV^n}/\overline{\sigma^\ast F^{p+1} \cV^n})
$$
that extends the identity map on $\mathbb{D}' - \{x'\}$. Using this extension for every $x'$ in $\sigma^{-1}(\Delta)$ shows that there is an isomorphism of line bundles on $C'$:
$$
(\sigma^\ast \det \overline{F^p \cV^n}_- / \overline{F^{p+1} \cV^n}_- ) \otimes \cO_{C'}\left (\sum_{x' \in \sigma^{-1}(\Delta)} r_{x'} \alpha^{p,n-p}_{\sigma(x')} \{x'\} \right ) \simeq \det(\overline{\sigma^\ast F^p \cV^n} / \overline{\sigma^\ast F^{p+1} \cV^n} ).
$$
that extends the canonical isomorphism on $\Cc'$.

The wanted isomorphism follows, using the fact that for every $x'$ in $\sigma^{-1}(\Delta)$, the multiplicity of the divisor $\sigma^\ast A^{p,n-p}$ at $x'$ is exactly $r_{x'} \alpha^{p,n-p} _{\sigma(x')}$ (which is an integer because the rational number $\alpha^{p,n-p} _{\sigma(x')}$ is in $\frac{1}{l_{\sigma(x')}} \Z$ by definition, and because the integer $r_{x'}$ is a multiple of $l_{\sigma(x')}$ by condition~\eqref{condition sigma}).
\end{proof}

\begin{proposition}
\label{upper Deligne extensions and elementary exponents}
For every pair $(p, n)$ of integers such that $0 \leq p \leq n,$ there is an isomorphism of line bundles on $C'$:
$$
\sigma^\ast \det(\overline{F^p \cV^n}_+ / \overline{F^{p+1} \cV^n}_+ ) \otimes \cO_{C'}(- \sigma^\ast A^{(N-1)-p,(N-1)-(n-p)}) \lra \det(\overline{\sigma^\ast F^p \cV^n} / \overline{\sigma^\ast F^{p+1} \cV^n} )
$$
whose restriction to $\Cc'$ is the identity of the line bundle $\sigma^\ast \det(F^p \cV^n / F^{p+1} \cV^n)$.
\end{proposition}

\begin{proof}
Using Poincar\'e and Serre duality, for every integer $n$, the vector bundle $\cV^n$ on $\Cc$ is isomorphic to the dual of the vector bundle $\cV^{2(N-1)-n}$, and for every integer $p$, the subquotient: 
$$
F^p \cV^n/F^{p+1} \cV^n \simeq R^{n-p} g_\ast \Omega^p_{Y_\Cc}
$$
is identified to the dual of the subquotient
$$F^{(N-1)-p} \cV^{2(N-1)-n} / F^{(N-1)-p + 1} \cV^{2(N-1)-n} \simeq R^{(N-1)-(n-p)} g_\ast \Omega^{(N-1)-p}_{Y_\Cc}.$$
As the duality exchanges the upper and lower Deligne extensions, we obtain an isomorphism of vector bundles on $C$:
$$
\overline{F^p \cV^n}_+ / \overline{F^{p+1} \cV^n}_+ \simeq (\overline{F^{(N-1)-p} \cV^{2(N-1)-n} }_- / \overline{F^{(N-1)-p+1} \cV^{2(N-1)-n} }_- )^\vee
$$
that extends the canonical isomorphism on $\Cc$; hence, pulling back by $\sigma$ and taking the determinant line bundle, an isomorphism of line bundles on $C'$:
$$
\sigma^\ast \det (\overline{F^p \cV^n}_+ / \overline{F^{p+1} \cV^n}_+ ) \simeq \sigma^\ast \det(\overline{F^{(N-1)-p} \cV^{2(N-1)-n} }_- / \overline{F^{(N-1)-p+1} \cV^{2(N-1)-n} }_- )^\vee
$$
that extends the canonical isomorphism on $\Cc'$.

Applying Proposition \ref{lower Deligne extensions and elementary exponents} to the integers $p' := (N-1)-p$, $n' := 2(N-1) - n$, we obtain an isomorphism:
\begin{align*}
&\sigma^\ast \det(\overline{F^{(N-1)-p} \cV^{2(N-1)-n} }_- / \overline{F^{(N-1)-p+1} \cV^{2(N-1)-n} }_- ) \otimes \cO(\sigma^\ast A^{(N-1)-p, (N-1) - (n-p)}) \\
&\simeq \det(\overline{\sigma^\ast F^{(N-1)-p} \cV^{2(N-1)-n} } / \overline{\sigma^\ast F^{(N-1)-p+1}  \cV^{2(N-1)-n} } )
\end{align*}
that extends the canonical isomorphism on $\Cc'$; and consequently, an isomorphism:
\begin{align*}
&\sigma^{\ast} \det(\overline{F^p \cV^n}_+ / \overline{F^{p+1} \cV^n}_+) \otimes \cO(-\sigma^{\ast} A^{(N-1)-p, (N-1) - (n-p)}) \\
&\simeq \det(\overline{\sigma^\ast F^{(N-1)-p} \cV^{2(N-1)-n}} / \overline{\sigma^\ast F^{(N-1)-p+1} \cV^{2(N-1)-n}} )^\vee
\end{align*}
that extends the canonical isomorphism on $\Cc'$.

Applying Poincar\'e and Serre duality again to the vector bundle $\sigma^\ast \cV^{2(N-1)-n}$ on $\Cc'$ and its subquotient $\sigma^\ast F^{(N-1)-p} \cV^{2(N-1)-n} / \sigma^\ast F^{(N-1)-p+1} \cV^{2(N-1)-n}$, and using that duality sends the (upper or lower) Deligne extension on $C'$ of this subquotient to the (upper or lower) Deligne extension of its dual, we obtain an isomorphism:
$$
(\overline{\sigma^\ast F^{(N-1)-p} \cV^{2(N-1)-n} } / \overline{\sigma^\ast F^{(N-1)-p+1} \cV^{2(N-1)-n}} ) ^\vee \simeq \overline{\sigma^\ast F^p \cV^n} / \overline{\sigma^\ast F^{p+1} \cV^n},
$$
that extends the canonical isomorphism on $\Cc'$.

This shows the result.
\end{proof}

\begin{proposition}
\label{upper and lower Deligne extensions and elementary exponents}
With the same notation, for  all integers $p, n$, 
there is an isomorphism of line bundles on $C$:
$$
\det(\overline{F^p \cV^n}_- / \overline{F^{p+1} \cV^n}_- ) \otimes \cO_C(A^{p,n-p} + A^{(N-1)-p,(N-1)-(n-p)}) \lra \det(\overline{F^p \cV^n}_+ / \overline{F^{p+1} \cV^n}_+)
$$
whose restriction to $\Cc$ is the identity of the line bundle $\det(F^p \cV^n / F^{p+1} \cV^n)$.

In particular, the coefficients of the divisor $A^{p,n-p} + A^{(N-1)-p,(N-1)-(n-p)}$ are integers.
\end{proposition}

\begin{proof}
Combining the isomorphisms of Propositions \ref{lower Deligne extensions and elementary exponents} and \ref{upper Deligne extensions and elementary exponents} yields an isomorphism of line bundles on $C'$:
$$
\sigma^\ast \det(\overline{F^p \cV^n}_- / \overline{F^{p+1} \cV^n}_- ) \otimes \cO_{C'}(\sigma^\ast (A^{p,n-p} + A^{(N-1)-p, (N-1)-(n-p)}) ) \simeq \sigma^\ast \det(\overline{F^p \cV^n}_+ / \overline{F^{p+1} \cV^n}_+ )
$$
that extends the identity isomorphism on $\Cc'$.

This  isomorphism, restricted to $\Cc$, is exactly the pullback by $\sigma$ of the identity isomorphism, so that the latter has an extension to an isomorphism on  $C$ as wanted.
\end{proof}

Now, we can apply these results to compute Griffiths bundles.

\begin{proposition}
\label{GK and elementary exponents}
For every integer $n \geq 0$, we have isomorphisms of line bundles on $C'$:
\begin{equation} \label{GK OA 1}
\sigma^\ast \GK_{C,-}(\H^n(Y_\eta / C_\eta)) \otimes \cO_{C'}\left (\sum_{0 \leq p \leq n} p \, \sigma^\ast A^{p,n-p} \right )
 \simeq \GK_{C'}(\sigma^\ast \H^n(Y_\eta / C_\eta) ),
\end{equation}
\begin{equation} \label{GK OA 2}
\sigma^\ast \GK_{C,+}(\H^n(Y_\eta / C_\eta)) \otimes \cO_{C'} \left (- \sum_{0 \leq p \leq n} p \,\sigma^\ast A^{(N-1)-p,(N-1)-(n-p)} \right )
 \simeq \GK_{C'}(\sigma^\ast \H^n(Y_\eta / C_\eta) ),
\end{equation}
and an isomorphism of line bundles on $C$:
\begin{equation} \label{GK OA 3} 
\GK_{C,-}(\H^n(Y_\eta / C_\eta)) \otimes \cO_C \left (\sum_{0 \leq p \leq n} p (A^{p,n-p} + A
^{(N-1)-p, (N-1)-(n-p)}) \right ) \simeq \GK_{C,+}(\H^n(Y_\eta / C_\eta)).
\end{equation}
\end{proposition} 
\begin{proof}
This follows from Propositions \ref{lower Deligne extensions and elementary exponents}, \ref{upper Deligne extensions and elementary exponents} and \ref{upper and lower Deligne extensions and elementary exponents} and the definition of Griffiths bundles.
\end{proof}

\begin{corollary} 
\label{GK heights, logarithmic cohomology and elementary exponents}
For every $n \geq 0$, we have the equalities in $\CH_0(C)$:
\begin{align*}
c_1(\GK_{C,+}(\H^n(Y_\eta / C_\eta))) &= c_1(\GK_{C,-}(\H^n(Y_\eta / C_\eta))) + \sum_{0 \leq p \leq n} p ([A^{p,n-p}] + [A^{(N-1)-p, (N-1)-(n-p)}]), \\
&= \sum_{0 \leq p \leq n} p c_1(R^{n-p} g_\ast \omega^p_{Y/C}) + \sum_{0 \leq p \leq n} p ([A^{p,n-p}] + [A^{(N-1)-p, (N-1)-(n-p)}]), 
\end{align*}
hence the equalities of integers:
\begin{align*}
\hgt_{GK,+}(\H^n(Y_\eta / C_\eta)) &= \hgt_{GK,-}(\H^n(Y_\eta / C_\eta)) + \sum_{0 \leq p \leq n} p \deg_C(A^{p,n-p} + A^{(N-1)-p, (N-1)-(n-p)}), \\
&= \sum_{0 \leq p \leq n} p \deg_C(R^{n-p} g_\ast \omega^p_{Y/C}) + \sum_{0 \leq p \leq n} p \deg_C(A^{p,n-p} + A^{(N-1)-p, (N-1)-(n-p)});
\end{align*}
and we also have the equalities of rational numbers:
\begin{align*}
\hgt_{GK, stab}(\H^n(Y_\eta / C_\eta)) &= \hgt_{GK,-}(\H^n(Y_\eta / C_\eta)) + \sum_{0 \leq p \leq n} p \deg_C(A^{p,n-p}), \\
&= \sum_{0 \leq p \leq n} p \deg_C(R^{n-p} g_\ast \omega^p_{Y/C}) + \sum_{0 \leq p \leq n} p \deg_C(A^{p,n-p}).
\end{align*}
\end{corollary}
\begin{proof}

The first four equalities follow from the  isomorphism \eqref{GK OA 3} in Proposition \ref{GK and elementary exponents} and Proposition \ref{lower Griffiths and logarithmic cohomology}.

For the last two equalities, by definition of the stable height,  using the  isomorphism \eqref{GK OA 1} in  Proposition \ref{GK and elementary exponents}  and Proposition \ref{lower Griffiths and logarithmic cohomology}, we have the equalities of rational numbers:
\begin{align*}
\hgt_{GK, stab}(\H^n(Y_\eta / C_\eta)) &= \frac{1}{\deg(\sigma)}  \deg_{C'}(\GK_{C'}(\sigma^\ast \H^n(Y_\eta / C_\eta) )), \\
&= \frac{1}{\deg(\sigma)}  \deg_{C'} \left (\sigma^\ast \GK_{C,-}(\H^n(Y_\eta / C_\eta)) \otimes \cO_{C'} \left (\sum_{0 \leq p \leq n} p \, \sigma^\ast A^{p,n-p} \right ) \right ),\\
&= \deg_C(\GK_{C,-}(\H^n(Y_\eta / C_\eta))) + \sum_{0 \leq p \leq n} p \deg_C(A^{p,n-p}),\\
&= \sum_{0 \leq p \leq n} p \deg_C(R^{n-p} g_\ast \omega^p_{Y/C}) + \sum_{0 \leq p \leq n} p \deg_C(A^{p,n-p}).
\qedhere
\end{align*}
\end{proof}

\section[Griffiths line bundles and non-degenerate critical points]{Comparison of Griffiths line bundles in the case of non-degenerate critical points}

\subsection{Non-degenerate critical points and elementary exponents}

Let $C$ be a connected smooth projective complex curve, let $H$ be an  $N$-dimensional smooth projective complex manifold, and let:
$$
f : H \lra C
$$
be a complex analytic morphism whose only critical points are non-degenerate.\footnote{Namely. at every point where the differential of $f$ vanishes, its Hessian is non-degenerate.} Let $\Sigma$ be the set of these critical points and $\Delta:= f(\Sigma)$ be its set-theoretic image by $f$.

We can rephrase the hypothesis as the divisor in $H$:
$$
H_\Delta := f^\ast \Delta
$$
having only finitely many singularities, all of which are ordinary double points.

Let 
$$
\nu : Y := \tilde{H} \lra H
$$
be the blow-up of these points and
$$
g := f \circ \nu : Y \lra C
$$
be the composition. For every point $P$ in $\Sigma$, let us define a subscheme in $Y$ by:
$$E_P := \nu^{-1}(\{P\}),$$
so that the $(E_P)_{P \in \Sigma}$ are the connected components of the exceptional divisor $E$. In particular, they are disjoint divisors isomorphic to the complex projective space $\PP^{N-1}$.

Let us define a divisor in $Y$ by:
$$
Y_\Delta := g^\ast \Delta = \nu^\ast H_\Delta.
$$
Since the divisor $H_\Delta$ has its only singularities at points of $\Sigma$, and since all of these singularities are ordinary double points, its pullback divisor $Y_\Delta$ can be written as follows:
\begin{equation}
Y_\Delta = 2 \sum_{P \in \Sigma} E_P + W \label{decomposition Htilde_Delta ele exp}
\end{equation}
where $W$ is the proper transform in $Y$ of the divisor $H_\Delta$ in $H$: it is a non-singular divisor intersecting transversally the components $(E_P)_{P\in \Sigma}$ of the exceptional divisor, and for every point $P$, the intersection $E_P \cap W$ is a smooth quadric in the projective space $E_P$.

In particular, the divisor $Y_\Delta$ is a divisor with strict normal crossings.

This section is devoted to the proof of the following result.

\begin{proposition}
\label{elementary exponents non-degenerate singularities}
With the above notation, for every $x$ in $\Delta$ and $p,q$ integers, all the elementary exponents $(\alpha^{p,q}_{j,x})_j$ of the degeneration $g$ at point $x$ vanish unless $N-1$ is even and $p = q = \frac{N-1}{2}$.

If $N-1$ is even and $p = q = \frac{N-1}{2}$, precisely $\vert \Sigma_x \vert$ of the $(\alpha^{p,q}_{j,x})_j$ are equal to $\frac{1}{2}$, and the other ones  vanish.

In particular, the rational number $\alpha^{p,q}_x$ vanishes unless $N-1$ is even and $p = q = \frac{N-1}{2}$, in which case it is equal to $\frac{1}{2} \vert \Sigma_x\vert$.

In other terms, the divisor $A^{p,q}$ with rational coefficients in $C$ vanishes unless $N-1$ is even and $p = q = \frac{N-1}{2}$, in which case it is given by:
$$
A^{\frac{N-1}{2}, \frac{N-1}{2}} = \frac{1}{2} f_\ast[\Sigma].
$$
\end{proposition} 

This proposition is stated in \cite[Prop. 3.10]{EFM21}, but only proved there in the case where each fiber of the morphism $f$ contains only one critical point.  In Subsection \ref{Proof Prop eens} below, we present an argument, which was explained  to us by the authors of \cite{EFM21}, that reduces the proposition to this special case.

\begin{corollary}
\label{GK bundles with elementary exponents for non-degenerate singularites}
For every integer n, we have the equality in $\CH_0(C)_\Q$:
\begin{align}
c_1(\GK_{C,+}(\H^n(H_\eta / C_\eta))) &= c_1(\GK_{C,-}(\H^n(H_\eta / C_\eta))) + \delta^{n,N-1} \eta_N \frac{N-1}{2} f_\ast [\Sigma], \label{c1 GK + non-degenerate 1}  \\
&= \sum_{0 \leq p \leq n} p c_1(R^{n-p} g_\ast \omega^p_{Y/C}) + \delta^{n,N-1} \eta_N \frac{N-1}{2} f_\ast [\Sigma], \label{c1 GK - non-degenerate 2}
\end{align}
hence the equality of rational numbers:
\begin{align}
\hgt_{GK,+}(\H^n(H_\eta / C_\eta)) &= \hgt_{
GK,-}(\H^n(H_\eta / C_\eta)) + \delta^{n,N-1} \eta_N \frac{N-1}{2} \vert \Sigma \vert \label{ht + non-degenerate 1} \\
&= \sum_{0 \leq p \leq n} p \deg_C(R^{n-p} g_\ast \omega^p_{Y/C}) + \delta^{n,N-1} \eta_N \frac{N-1}{2} \vert \Sigma \vert; \label{ht + non-degenerate 2}
\end{align}
and we also have the equality of rational numbers:
\begin{align}
\hgt_{GK, stab}(\H^n(H_\eta / C_\eta)) &= \hgt_{GK,-}(\H^n(H_\eta / C_\eta)) + \delta^{n,N-1} \eta_N \frac{N-1}{4} \vert \Sigma \vert \label{ht stab non-degenerate 1} \\
&= \sum_{0 \leq p \leq n} p \deg_C(R^{n-p} g_\ast \omega^p_{Y/C}) + \delta^{n,N-1} \eta_N \frac{N-1}{4} \vert \Sigma\vert, \label{ht stab non-degenerate 2}
\end{align}
where $\eta_N$ is $1$ if $N$ is odd and $0$ if $N$ is even.
\end{corollary}
\begin{proof}
For $n$ an integer, using Proposition \ref{elementary exponents non-degenerate singularities}, the sum in $\CH_0(C)_\Q$:
$$
\sum_{0 \leq p \leq n} p ([A^{p,n-p}] + [A^{(N-1)-p, (N-1)-(n-p)}])
$$
vanishes unless $n = N-1$ and $N-1$ is even, and in this case, it is equal to
$$
\frac{N-1}{2} ([A^{\frac{N-1}{2}, \frac{N-1}{2}}] + [A^{\frac{N-1}{2}, \frac{N-1}{2}}])
= \frac{N-1}{2} f_\ast [\Sigma].
$$
So for every integer $n$, we obtain the equality:
$$
\sum_{0 \leq p \leq n} p ([A^{p,n-p}] + [A^{(N-1)-p, (N-1)-(n-p)}]) = \delta^{n,N-1} \eta_N \frac{N-1}{2} f_\ast [\Sigma].
$$
Similarly, for every integer $n$, we obtain the equality in $\CH_0(C)_\Q$:
$$
\sum_{0 \leq p \leq n} p [A^{p,n-p}] = \delta^{n,N-1} \eta_N \frac{N-1}{4} f_\ast [\Sigma].
$$
Replacing these equalities in Corollary \ref{GK heights, logarithmic cohomology and elementary exponents} yields the result.
\end{proof}

\subsection{Proof of Proposition \ref{elementary exponents non-degenerate singularities}}\label{Proof Prop eens} 

Let $x$ be a point in $\Delta$. As the proposition is local around $x$, we can replace the curve $C$ by the unit open disk $\mathbb{D}$ with coordinate $t$, the point $x$ by $0$, and assume that $f$ is smooth over $\mathbb{D}^\ast := \mathbb{D} - \{0\}$ and 
admits only non-degenerate critical points. In particular, these critical points are the points in $\Sigma_x$; we shall simply call this subset $\Sigma$.

Let us introduce a copy $\mathbb{D}'$ of the unit open disk with coordinate $t'$, let us define an analytic morphism by:
$$
\sigma : \mathbb{D}' \lra \mathbb{D}, \quad t' \longmapsto t'^2,
$$
and let us define the fiber product:
\begin{equation}
\xymatrix{
\tilde{H}' := Y \times_{\mathbb{D}} \mathbb{D}' \ar[d]^{g'} \ar[r]^-{\rho} & Y \ar[d]^{g} \\
\mathbb{D}' \ar[r]^{\sigma} & \mathbb{D} \\
} \label{commutative diagram double points}
\end{equation}
Let us define the following divisors in the singular analytic space $\tilde{H}'$:
$$
E'_P := \rho^\ast E_P, P \in \Sigma,
$$
$$
E' := \bigsqcup_{P \in \Sigma} E'_P,
$$
and:
$$
W' := \rho^\ast W.
$$
Using \eqref{decomposition Htilde_Delta ele exp}, we get the equality of Cartier divisors in $\tilde{H}'$:
\begin{equation}
2 \tilde{H}'_0 = \rho^\ast Y_0 = 2 E' + W'. \label{eq div fiber product}
\end{equation}
Let us consider the normalization of the singular analytic space $\tilde{H}'$:
$$
\pi : Y' \lra \tilde{H}',
$$
and the composition:
$$
h := g' \circ \pi : Y' \lra \mathbb{D}'.
$$
Since the morphism $g$ is smooth on the open subset $Y - E \subset Y$, the morphism $g'$ is smooth on the open subset $\tilde{H}' - E' \subset \tilde{H}'$, in particular the analytic space $\tilde{H}' - E'$ is non-singular, and the morphism $\pi$ induces an isomorphism between $Y' - \pi^{-1}(E')$ and $\tilde{H}' - E'$.

The cokernel of the canonical injective morphism of $\cO_{\tilde{H}'}$-modules
$$
\phi: \cO_{\tilde{H}'} \lra \pi_\ast \cO_{Y'}
$$
is supported by the closed subset $E'$, so we can write it as follows:
$$
\mathrm{Coker}(\phi) = \bigoplus_{P \in \Sigma} \mathcal{F}_P,
$$
where for every point $P$ in $\Sigma$, $\mathcal{F}_P$ is a coherent sheaf supported by the closed subset $E'_P$.

The proof of Proposition \ref{elementary exponents non-degenerate singularities} shall rely on the following three lemmas.

\begin{lemma}
\label{ele exp double points 1}
The three following statements hold.

\emph{(i)} The morphism 
$$
h : Y' \lra \mathbb{D}'
$$
defines a semistable reduction of the morphism $g$; namely the analytic space $Y'$ is non-singular and the fiber $Y'_0:=h^{-1}(0)$ is a reduced divisor with strict normal crossings in $Y'$.

\emph{(ii)} For every integer $p$, the canonical isomorphism of vector bundles on $Y'_{\mathbb{D}'^\ast}$:
$$
(\rho \circ \pi)^\ast \Omega^p_{Y_{\mathbb{D}^\ast} / \mathbb{D}^\ast} \lra \Omega^p_{Y'_{\mathbb{D}'^\ast} / \mathbb{D}'^\ast}
$$
extends into an isomorphism of vector bundles on $Y'$:
$$
(\rho \circ \pi)^\ast \omega^p_{Y/\mathbb{D}} \lra \omega^p_{Y'/\mathbb{D}'}.
$$

\emph{(iii)} For every point $P$ in $\Sigma$, the $\cO_{\tilde{H}'}$-module $(g'^\ast t')  \mathcal{F}_P$ vanishes.
\end{lemma}
\begin{proof}

The three statements can be proved in local models of $\tilde{H}'$ and of $Y'$. We shall only show them in a local model of $\tilde{H}'$ near a point of $E'_P \cap W'$ and a local model of $Y'$ near a point of $\pi^{-1}(E'_P \cap W')$ (where $P$ is a point in $\Sigma$) when $N \geq 2$, which is the most complicated case.

Let $P$ be a point in $\Sigma$. Near a point of $E_P \cap W$, we can choose a local coordinate system $(y_1,...,y_N)$ of $Y$, such that the divisor $W$ is defined by $(y_1 = 0)$, the divisor $E_P$ is defined by $(y_2 = 0)$, and satisfying the equality of functions:
$$
g^\ast t = y_1 y_2^2.
$$
Consequently, near a point of $E'_P \cap W'$, the analytic space $\tilde{H}'$ admits coordinates $(g'^\ast t', \rho^\ast y_1,...,\rho^\ast y_N)$, subjected to the relation:
$$
g'^\ast t'^2 = (\rho^\ast y_1) \, (\rho^\ast y_2)^2.
$$
The local model is that of the product of a ``Whitney umbrella'' singularity by $\C^{N-2}$.

Finally, near a point of $\pi^{-1}(E'_P \cap W')$, the analytic space $Y'$ admits a local coordinate system $(z, (\rho \circ \pi)^\ast y_2,...,(\rho \circ \pi)^\ast y_N)$ satisfying the relations:
$$
h^\ast t' = z \, (\rho \circ \pi)^\ast y_2
$$
and:
$$
(\rho \circ \pi)^\ast y_1 = z^2.
$$
In particular, the analytic space $Y'$ is smooth, and the divisor $\pi^\ast W'$ in $Y'$ is of multiplicity $2$, which allows us to define a divisor by $(1/2) \pi^\ast W'$ (defined by $(z = 0)$ in the above local model). We also have that the morphism $h$ is smooth over $\mathbb{D}'^\ast$, and that the divisor $Y'_0$ in $Y'$ can be written:
\begin{equation}
Y'_0 = \sum_{P \in \Sigma} \pi^\ast E'_P + (1/2) \pi^\ast W', \label{eq div normalization}
\end{equation}
where the divisors $(\pi^\ast E'_P)_{P \in \Sigma}$ are disjoint and non-singular, where the divisor $(1/2) \pi^\ast W'$ is non-singular, and intersects transversally the divisors $(\pi^\ast E'_P)_{P \in \Sigma}$.

Consequently, the degeneration $h$ is semistable, which shows (i).

It clearly suffices to prove (ii) when the integer $p$ is $1$. In the above local models, the locally free sheaf $\omega^1_{Y/\mathbb{D}}$ on $Y$ is generated by $([\mathrm{d} y_1/y_1], [\mathrm{d} y_2/y_2], [\mathrm{d} y_3],...,[\mathrm{d} y_N])$ modulo the relation $[\mathrm{d} y_1/y_1] + 2 [\mathrm{d} y_2/y_2] = 0$, hence it admits a local frame $([\mathrm{d} y_2/y_2], [\mathrm{d} y_3],...,[\mathrm{d} y_N])$. 

On the other hand, the locally free sheaf $\omega^1_{Y'/\mathbb{D}'}$ on $Y'$ is generated by ($[\mathrm{d} z/z]$, $(\rho \circ \pi)^\ast [\mathrm{d} y_2/y_2]$, $(\rho \circ \pi)^\ast [\mathrm{d} y_3]$,...) modulo the relation $[\mathrm{d} z/z] + (\rho \circ \pi)^\ast [\mathrm{d} y_2/y_2] = 0$, hence it admits a local frame $((\rho \circ \pi)^\ast [\mathrm{d} y_2/y_2], (\rho \circ \pi)^\ast [\mathrm{d} y_3],...)$.

The canonical morphism of sheaves on $Y'_{\mathbb{D}'^\ast}$:
$$
(\rho \circ \pi)^\ast \Omega^1_{Y_{\mathbb{D}^\ast}/\mathbb{D}^\ast} \simeq (\rho \circ \pi)^\ast \omega^1_{Y/\mathbb{D} | Y_{\mathbb{D}^\ast}} \lra \Omega^1_{Y'_{\mathbb{D}'^\ast}/\mathbb{D}'^\ast} \simeq \omega^1_{Y'/\mathbb{D}' | Y'_{\mathbb{D}'^\ast}}
$$
is compatible with these local frames, hence extends into an isomorphism on $Y'$, which shows (ii).

For (iii), in the above local models, the $\cO_{\tilde{H}'}$-algebra $\pi_\ast \cO_{Y'}$ is generated by $z$, subjected to the relations:
$$
g'^\ast t' = z \,  (\rho^\ast y_2)
$$
and:
$$
\rho^\ast y_1 = z^2.
$$
Consequently, the $\cO_{\tilde{H}'}$-module $\pi_\ast \cO_{Y'}$ is generated by the powers of $z$.
For every even integer $r$, the element $z^r$ is contained in $\cO_{\tilde{H}'}$, so this module is generated by the odd powers of $z$.

Let $r$ be an odd integer, we have the equality of elements of the algebra $\pi_\ast \cO_{Y'}$:
$$
(g'^\ast t') \, z^r = z^{r+1} \, (\rho^\ast y_2) \in \cO_{\tilde{H}'},
$$
Consequently, the $\cO_{\tilde{H}'}$-submodule $(g'^\ast t') \, \pi_\ast \cO_{Y'}$ of $\pi_\ast \cO_{Y'}$ is contained in $\cO_{\tilde{H}'}$. By definition of the morphism $\phi$, the element $(g'^\ast t')$ is in the annihilator of its cokernel $\mathcal{F}_P$, which shows (iii).
\end{proof}

\begin{lemma}
\label{ele exp double points 2}
For every pair of non-negative integers $(p,q)$, the cokernel of the natural injective map of coherent sheaves on $\mathbb{D}'$:
$$
u^{p,q} : \sigma^\ast R^q g_\ast \omega^p_{Y/\mathbb{D}} \lra R^q h_\ast \omega^p_{Y'/\mathbb{D}'}
$$
is isomorphic to the coherent sheaf:
$$
\bigoplus_{P \in \Sigma} R^q g'_\ast (\mathcal{F}_P \otimes \rho^\ast \omega^p_{Y/\mathbb{D}}).
$$
\end{lemma}

\begin{proof}

Tensoring the following exact sequence of coherent sheaves on $\tilde{H}'$:
$$
0 \lra \cO_{\tilde{H}} \overset{\phi}{\lra} \pi_\ast \cO_{Y'} \lra \bigoplus_{P \in \Sigma} \mathcal{F}_P \lra 0,
$$
by the locally free sheaf $\rho^\ast \omega^p_{Y/\mathbb{D}}$, and using the projection formula, yields an exact sequence of coherent sheaves on $\tilde{H}'$:
\begin{equation}
0 \lra \rho^\ast \omega^p_{Y/\mathbb{D}} \lra \pi_\ast (\rho \circ \pi)^\ast \omega^p_{Y/\mathbb{D}} \lra \bigoplus_{P \in \Sigma} \mathcal{F}_P \otimes (\rho^\ast \omega^p_{Y/\mathbb{D}}) \lra 0. \label{second exact sequence elementary exponents}
\end{equation}
Using Lemma \ref{ele exp double points 1}, (ii), we have a canonical isomorphism of vector bundles on $Y'$:
$$
(\rho \circ \pi)^\ast \omega^p_{Y/\mathbb{D}} \lra \omega^p_{Y'/\mathbb{D}'}.
$$
Hence replacing in \eqref{second exact sequence elementary exponents}, we obtain an exact sequence of coherent sheaves on $\tilde{H}'$:
\begin{equation}
0 \lra \rho^\ast \omega^p_{Y/\mathbb{D}} \lra \pi_\ast \omega^p_{Y'/\mathbb{D}'} \lra \bigoplus_{P \in \Sigma} \mathcal{F}_P \otimes (\rho^\ast \omega^p_{Y/\mathbb{D}}) \lra 0. \label{third exact sequence elementary exponents}
\end{equation}

The higher direct images of  \eqref{third exact sequence elementary exponents} by the morphism $g'$ fit into the following long exact sequence of coherent analytic sheaves over $\mathbb{D}'$ (with the convention that $R^{-1} g'_\ast = 0$):
\begin{multline}
... \lra \bigoplus_{P \in \Sigma} R^{q-1} g'_\ast (\mathcal{F}_P \otimes \rho^\ast \omega^p_{Y/\mathbb{D}}) 
\lra R^q g'_\ast (\rho^\ast \omega^p_{Y/\mathbb{D}}) \lra R^q g'_\ast (\pi_\ast \omega^p_{Y'/\mathbb{D}'}) \\
\lra \bigoplus_{P \in \Sigma} R^q g'_\ast (\mathcal{F}_P \otimes \rho^\ast \omega^p_{Y/\mathbb{D}}) 
\lra R^{q+1} g'_\ast (\rho^\ast \omega^p_{Y/\mathbb{D}}) \lra ... \label{long exact sequence elementary exponents}
\end{multline}

Since the morphism $\sigma$ is flat, by base change using the commutative diagram \eqref{commutative diagram double points}, we obtain the following isomorphisms of coherent sheaves on $\mathbb{D}'$:
$$
R^q g'_\ast (\rho^\ast \omega^p_{Y/\mathbb{D}}) \simeq \sigma^\ast R^q g_\ast \omega^p_{Y/\mathbb{D}},
$$
and:
$$
R^{q+1} g'_\ast (\rho^\ast \omega^p_{Y/\mathbb{D}}) \simeq \sigma^\ast R^{q+1} g_\ast \omega^p_{Y/\mathbb{D}}.
$$
These sheaves are locally free on $\mathbb{D}'$ by Proposition \ref{logarithmic Hodge bundles}. 

On the other hand, since for every point $P$ in $\Sigma$, the coherent sheaf $\mathcal{F}_P$ on $\tilde{H}'$ is supported on $E'_P$, the coherent sheaves $\bigoplus_{P \in \Sigma} R^q g'_\ast (\mathcal{F}_P \otimes \rho^\ast \omega^p_{Y/\mathbb{D}})$ and $\bigoplus_{P \in \Sigma} R^{q-1} g'_\ast (\mathcal{F}_P \otimes \rho^\ast \omega^p_{Y/\mathbb{D}})$ on $\mathbb{D}'$ are clearly supported by the point $0$. 

Consequently, the connecting morphisms in \eqref{long exact sequence elementary exponents} vanish, and \eqref{long exact sequence elementary exponents} induces a short exact sequence of coherent sheaves on $\mathbb{D}'$:
\begin{equation}
0 \lra R^q g'_\ast (\rho^\ast \omega^p_{Y/\mathbb{D}}) \simeq \sigma^\ast R^q g_\ast \omega^p_{Y/\mathbb{D}} \lra R^q g'_\ast (\pi_\ast \omega^p_{Y'/\mathbb{D}'}) \lra \bigoplus_{P \in \Sigma} R^q g'_\ast (\mathcal{F}_P \otimes \rho^\ast \omega^p_{Y/\mathbb{D}}) \lra 0. \label{short exact sequence D elementary exponents}
\end{equation}

Furthermore, the normalization morphism $\pi$ is finite, so its higher direct image functors on coherent sheaves vanish. Using the spectral sequence of composed functors, we can rewrite the second sheaf in \eqref{short exact sequence D elementary exponents}:
$$
R^q g'_\ast (\pi_\ast \omega^p_{Y'/\mathbb{D}'}) \simeq R^q h_\ast \omega^p_{Y'/\mathbb{D}'}.
$$
Under this identification, the second arrow in the diagram \eqref{short exact sequence D elementary exponents} is precisely $u^{p,q}$. So we obtain that its cokernel is the coherent sheaf on $\mathbb{D}'$ given by:
$$
\mathrm{Coker}\, u^{p,q} \simeq \bigoplus_{P \in \Sigma} R^q g'_\ast (\mathcal{F}_P \otimes \rho^\ast \omega^p_{Y/\mathbb{D}}),
$$
which concludes the proof.
\end{proof}

\begin{lemma}
\label{ele exp double points 3}
The isomorphism class of the coherent sheaf  on ${\mathbb{D}}'$, supported by $\{0\}$:
$$
R^q g'_\ast (\mathcal{F}_P \otimes \rho^\ast \omega^p_{Y/\mathbb{D}}),
$$
where $(p,q)$ is a pair of non-negative integers and $P$ is a point in $\Sigma$, only depends on the integers $N,$ $p,$ and~$q$.
\end{lemma}

This lemma is a simple consequence of the local and intrinsic character of the notion of non-degenerate critical point, and of the  operations of blow-up, base change and normalization involved in the above construction. For the sake of completeness, we include some details that the readers could skip at their discretion.

\begin{proof}
For every point $P$ in $\Sigma$, we can construct local models of neighborhoods of $E'_P$ in $\tilde{H}'$ and of $\pi^{-1}(E'_P)$ in $Y'$, that only depend on the integer $N$, as follows.

Since the morphism $f$ admits a non-degenerate critical point at $P$, we can choose an analytic coordinate system $(x_{1, P},...,x_{N, P})$ on a neighborhood $U_P$ of $P$ in $H$, defining a morphism inserted in a commutative diagram:
$$\xymatrix{U_P \ar[r]^-{\sim} \ar[d]^{f_{| U_P}} & U \subset \C^N \ar[ld]^{f_{loc}} 
\\ \mathbb{D} & }$$
where $U$ is some fixed neighborhood of $0$ in $\C^N$ (for instance the open polydisk of radius $\sqrt{1/N}$ centered in 0), and $f_{loc}$ is the morphism sending $(x_1,...,x_N)$ to $x_1^2 + ... + x_N^2$.

Taking the blow-ups of $U_P$ at $P$ and of $U$ at the origin, we obtain a commutative diagram:
$$\xymatrix{V_P \ar[r]^-{\sim} \ar[d]_{g_{| V_P}} & V \ar[ld]^{g_{loc}} 
\\ \mathbb{D} & }$$
where $V_P$ is a neighborhood of $E_P$ in $\tilde{H}$, $V$ is an analytic manifold that only depends on $N$, containing a divisor $E_{loc}$ identified with $E_P$, and $g_{loc}$ is an  analytic morphism that only depends on $N$.

Taking the fiber product of the morphisms $g_{| V_P}$ and $g_{loc}$ with the morphism $\sigma$, then taking the normalizations of the resulting analytic spaces, we obtain a commutative diagram:
$$\xymatrix{W_P \ar[r]^-{\sim} \ar[d]^{\pi_{| W_P}} & W \ar[d]^{\pi_{loc}} & \\
V'_P \ar[r]^-{\sim} \ar[dr]_{g'_{| V'_P}} & V' \ar[d]^{g'_{loc}} \ar[r]^{\rho_{loc}} & V \ar[d]^{g_{loc}}
\\ & \mathbb{D'} \ar[r]^{\sigma} & \mathbb{D} }$$
where $V'_P$ (resp. $W_P$) is a neighborhood of $E'_P$ (resp. $\pi^{-1}(E'_P)$) in $\tilde{H}'$ (resp. $Y'$), $V'$ (resp. $W$) is an analytic space (resp. manifold) that only depends on $N$, containing a divisor $E'_{loc}$ (resp. $\pi_{loc}^{-1}(E'_{loc})$) identified with $E'_P$ (resp. $\pi^{-1}(E'_P)$) via the isomorphisms, and $\pi_{loc}$, $g'_{loc}$, and $\rho_{loc}$ are analytic morphisms that only depend on $N$.

The morphism of coherent sheaves $\phi_{| V'_P}$ on $V'_P$ is identified via these isomorphisms with a morphism of coherent sheaves on $V'$:
$$\phi_{loc} : \cO_{V'} \lra \pi_{loc \ast} \cO_W.$$
Let $\mathcal{F}_{loc}$ be its cokernel: it is a coherent sheaf on $V'$ supported on $E'_{loc}$, that only depends on $N$, identified via the isomorphisms with $\mathcal{F}_P$.

Similarly, for every integer $p$, the locally free sheaf $(\rho^\ast \omega^p_{Y/\mathbb{D}})_{| V'_P}$ on $V'_P$ is identified with the locally free sheaf $\rho_{loc}^\ast \omega^p_{V/\mathbb{D}}$ on $V'$, which only depends on $N$.

Consequently, for every pair of integers $(p,q)$, it follows from the commutativity of the diagram, and from the fact that the coherent sheaf $\mathcal{F}_P \otimes (\rho^\ast \omega^p_{Y/\mathbb{D}})$ is supported on $E'_P$, hence on $V'_P$, that we have an isomorphism of coherent sheaves on $\mathbb{D}'$:
\begin{align*}
    R^q g'_\ast (\mathcal{F}_P \otimes \rho^\ast \omega^p_{Y/\mathbb{D}}) &\simeq R^q g'_\ast (\mathcal{F}_P \otimes \rho^\ast \omega^p_{Y/\mathbb{D}})_{| V'_P}, \\
    &\simeq R^q g'_{loc \ast} (\mathcal{F}_{loc} \otimes \rho_{loc}^\ast \omega^p_{V/\mathbb{D}}),
\end{align*}
so that this coherent sheaf only depends on the integers $N,p,q$.
\end{proof}

\begin{proof}[Proof of Proposition \ref{elementary exponents non-degenerate singularities}]
Using Lemma \ref{ele exp double points 1}, (i), and the fact that $\pi$ is an isomorphism over $\mathbb{D}'^\ast$, the commutative diagram:
$$
\xymatrix{
Y' \ar[r]^{\rho \circ \pi} \ar[d]^{h} & Y \ar[d]^{g} \\
\mathbb{D}' \ar[r]^{\sigma} & \mathbb{D} \\
}
$$
is a semistable reduction diagram. Therefore, according to the  definition of elementary exponents, for every pair of integers $(p,q)$, the elementary exponents $(\alpha^{p,q}_{j,x})_{j}$ can be described in terms of the cokernel of the injective map $u^{p,q}$ from Lemma \ref{ele exp double points 2}. By this Lemma, this cokernel is isomorphic to the coherent sheaf:
$$
\bigoplus_{P \in \Sigma} R^q g'_\ast (\mathcal{F}_P \otimes \rho^\ast \omega^p_{Y/\mathbb{D}}).
$$
Using Lemma \ref{ele exp double points 1}, (iii), for every point $P$ in $\Sigma$, the coherent sheaf $\mathcal{F}_P$ on $\tilde{H}'$ is annihilated by the function $g^\ast t'$, hence can be written $i_{\tilde{H}'_0 \ast} \mathcal{F}'_P$, where $\mathcal{F}'_P$ is some coherent sheaf on $\tilde{H}'_0$, and $i_{\tilde{H}'_0}$ is the inclusion map. 

This implies that the cokernel of $u^{p,q}$ is isomorphic to the coherent sheaf:
$$
\bigoplus_{P \in \Sigma} H^q(\tilde{H}'_0, \mathcal{F}'_P \otimes (\rho^\ast \omega^p_{Y/\mathbb{D}})_{| \tilde{H}'_0} )_0,
$$
where $(.)_0$ denotes the skyscraper sheaf on $0$ in $\mathbb{D}'$ associated with a vector space.

We can rewrite this coherent sheaf as:
$$
\bigoplus_{P \in \Sigma} (\cO_{\mathbb{D}'}/(t' \cO_{\mathbb{D}'}) )^{\oplus  h^q(\tilde{H}'_0, \mathcal{F}'_P \otimes (\rho^\ast \omega^p_{Y/\mathbb{D}})_{| \tilde{H}'_0})}.
$$
Hence by definition of the elementary exponents $(\alpha^{p,q}_{j,x})_{j}$, there are precisely $$\sum_{P \in \Sigma} h^q(\tilde{H}'_0, \cF'_P \otimes (\rho^\ast \omega^p_{Y/\mathbb{D}})_{| \tilde{H}'_0})$$ of them that are  nonzero, and all of these are equal to $1/2$.

Using Lemma \ref{ele exp double points 3}, the number: 
$$
h^q(\tilde{H}'_0, \cF'_P \otimes (\rho^\ast \omega^p_{Y/\mathbb{D}})_{| \tilde{H}'_0})
$$
only depends on the integers $N,p,q,$ and we shall denote it $c_N^{p,q}$. Consequently, the number of nonzero elementary exponents is $|\Sigma| . c_N^{p,q}$.

If there is only one critical point in the fiber,\footnote{An instance of this situation actually exists for every choice of $N\geq 1$.  Consider for instance the hypersurface $X_1^2 + \dots + X_N^{2} = t X_0^2$ in $\PP^N \times \mathbb{D}$.} the argument in \cite{EFM21} shows that the number of nonzero elementary exponents is $0$ unless $N-1$ is even and $p = q = (N-1)/2$, in which case it is $1$. Consequently, the integer $c_N^{p,q}$ is $0$ unless $N-1$ is even and $p = q = (N-1)/2$, in which case it is $1$, and we obtain in general that all the elementary exponents vanish unless $N-1$ is even and $p = q = (N-1)/2$, in which case precisely $|\Sigma|$ of them do not vanish, and they are equal to~$1/2$.
\end{proof}

\section[Poincar\'e duality and Griffiths line bundles]{Poincar\'{e} duality and Griffiths line bundles of cohomology in complementary degrees}

We return to the notation of Section \ref{Comp Griff Elem}. Namely, let $Y$ be a smooth projective complex scheme of pure dimension $N$, $C$ be a connected smooth projective curve, and
$$
g : Y \lra C
$$
be a surjective morphism that is smooth over $\Cc:= C - \Delta$, where $\Delta$ is a finite subset of $C$, and such that the divisor $Y_\Delta$ is a divisor with strict normal crossings. Finally we denote:
$$Y_\Cc := g^{-1}(\Cc) = Y - Y_\Delta.$$

Let $\eta$ be the generic point of $C$.

Let $n$ be an integer such that $0 \leq n \leq 2(N-1)$.

For every point $x$ in $\Delta$, let $(\alpha^n_{j,x})_j$ be the union over the pairs $(p,q)$ such that $p+q = n$, of the elementary exponents $(\alpha^{p,q}_{j,x})_{j, p,q}$ of the $(p,q)$-Hodge bundle of the degeneration $g$ at the point $x$. Let $\alpha^n_x$ be their sum, and let us define a divisor in $C$ with rational coefficients by:
$$
A^n := \sum_{x \in \Delta} \alpha^n_x \{x\}.
$$
Observe that with the notation of the previous sections, this divisor is also given by:
$$A^n = \sum_{p,q \geq 0, p+q = n} A^{p,q}.$$
Let $r$ be the least common multiple of the denominators of the rational numbers $(\alpha^n_x)_{x \in \Delta}$, so that $r A^n$ is a divisor with integral coefficients.
\begin{lemma}
\label{det H^n 2-torsion}
With the above notation, the line bundle on $C$:
$$
\big(\det \overline{\ccH^n(Y_\Cc/ \Cc)}_- \big)^{\otimes r} \otimes \cO_C(r A^n)
$$
is of 2-torsion.
\end{lemma}
\begin{proof}
Let us denote simply $\cV$ the complex analytic vector bundle $\ccH^n(Y_\Cc / \Cc)$ on $\Cc$. It has an integral structure, associated to Betti cohomology with $\Z$-coefficients, so its determinant line bundle is of $2$-torsion, and we have an isomorphism of line bundles on $\Cc$:
$$
\psi : \cO_{\Cc} \lra (\det \cV)^{\otimes 2},
$$
that sends $1$ to $(v_1 \wedge ... \wedge v_s)^2$, where $(v_1,...,v_s)$ is any local integral frame of the vector bundle $\cV$.

Let $x$ be a point of $\Delta$, let $U$ be a neighborhood of $x$ with a local coordinate $t$ centered in $x$, let $\infty$ be a point in $U - \{x\}$, and let $v_1,...,v_s$ be an integral frame of the vector space $\cV_\infty$.

Let $T$ be the (counterclockwise) monodromy automorphism on $\cV_\infty$ around the point $x$. As the divisor $Y_x$ in $Y$ is a divisor with strict normal crossings, using Proposition \ref{lower Deligne extension and logarithmic cohomology}, this automorphism is quasi-unipotent.

As in the construction of the lower Deligne extension (see \cite[II, Prop. 5.4]{Deligne70}, \cite[``Key Lemma" p. 547]{Katz76}, and \cite[section 2.1]{EFM21}), let $B_-$ be the only automorphism on $\cV_\infty$ whose eigenvalues have their real parts in $]-1, 0]$, and satisfying the equality,
$$
e^{2i\pi B_-} = T.
$$
As the automorphism $T$ is quasi-unipotent, these eigenvalues are rational numbers.

By Proposition \ref{domain elementary exponents}, the elementary exponents $(\alpha_{j,x}^n)_j$ are all in $[0, 1[$, and by Proposition \ref{elementary exponents and monodromy}, they satisfy that the $(e^{- 2i\pi \alpha^n_{j,x}})_j$ are the eigenvalues of the monodromy $T$. 

Consequently, the eigenvalues of the automorphism $B_-$ are exactly the $(- \alpha^n_{j,x})_j$.

By definition of the lower Deligne extension, we have an isomorphism of vector bundles on $U$:
$$
\chi : \cV_\infty \otimes_\C \cO_U \lra (\overline{\cV}_-)_{| U},
$$
such that on every point $x'$ in $U$ with coordinate $t(x')$, an integral frame of $\cV_{x'}$ is given by:
$$
(\chi_{x'}(e^{B_- (\log(t(x')) - \log(t(\infty)))} . v_j) )_j,
$$
depending on a choice of logarithm of $t(x')$ and of $t(\infty)$.

Observe that adding a multiple of $2i\pi$ to $\log(t(x')) - \log(t(\infty))$ multiplies the basis by $e^{2i\pi B_-} = T$, which respects the integral structure.

Let $x'$ be a point in $U$ with coordinate $t(x')$, and let us choose a logarithm $\log(t(x'))$ and a logarithm $\log(t(\infty))$. By definition of the morphisms $\psi$ and $\chi$, the image of $1$ by the morphism of vector spaces
$$
(\det \chi_{x'}^{-1}) ^{\otimes 2r} \circ \psi_{x'}^{\otimes r} : \C \lra (\det \cV_\infty)^{\otimes 2r},
$$
is given by:
\begin{align}
((\det \chi_{x'}^{-1}) ^{\otimes 2r} \circ \psi_{x'}^{\otimes r})(1) 
&= (e^{B_- (\log(t(x')) - \log(t(\infty)))} v_1 \wedge ... \wedge e^{B_- (\log(t(x')) - \log(t(\infty)))} v_s )^{2r} \nonumber\\
&= \det(e^{B_- (\log(t(x')) - \log(t(\infty)) )} )^{2r} . (v_1 \wedge ... \wedge v_s)^{2r} \nonumber\\
&= e^{2r \mathrm{tr}(B_- (\log(t(x')) - \log(t(\infty)))) } . (v_1 \wedge ... \wedge v_s)^{2r} \nonumber\\
&= e^{2r (\log(t(x')) - \log(t(\infty)) ) \mathrm{tr}(B_-)} . (v_1 \wedge ... \wedge v_s)^{2r} \nonumber\\
&= e^{2r (\log(t(x')) - \log(t(\infty))) \sum_j (- \alpha^n_{j,x}) } . (v_1 \wedge ... \wedge v_s)^{2r} \label{eq Griff int 1}\\
&= e^{-2r (\log(t(x')) - \log(t(\infty))) \alpha^n_x} . (v_1 \wedge ... \wedge v_s)^{2r} \nonumber\\
&= (t(x')/t(\infty))^{-2r \alpha^n_x} . (v_1 \wedge ... \wedge v_s)^{2r}. \label{eq Griff int 2}
\end{align}
The equality \eqref{eq Griff int 1} holds because 
 the eigenvalues of the automorphism $B_-$ are exactly the $(-\alpha^n_{j,x})_j$; \eqref{eq Griff int 2} holds because $r \alpha^n_x$ is an integer.
Observe that this expression does not depend on the choices of the logarithms $\log(t(x'))$ and $\log(t(\infty))$.

Consequently, the section $((\det \chi^{-1}) ^{\otimes 2r} \circ \psi_{| U - \{x\}}^{\otimes r})(1)$ over $U -\{x\}$ of the line bundle $$(\det \cV_\infty)^{\otimes 2r} \otimes \cO_{U}$$ admits a pole of order $2r \alpha^n_x$ in $x$.

Since $\chi$ is an isomorphism of vector bundles, the section $\psi_{| U - \{x\}}^{\otimes r}(1)$ of the line bundle 
$$
(\det \cV_{| U - \{x\}})^{\otimes 2r} \simeq (\det \overline{\cV}_-)_{| U - \{x\}}^{\otimes 2r} 
$$
also admits a pole of order $2r \alpha^n_x$ in $x$.

Consequently, the isomorphism of line bundles $\psi^{\otimes r}$ on $\Cc$ can be extended into an isomorphism of line bundles on $C$:
$$
\overline{\psi^{\otimes r}}_- : \cO_C(-2r \sum_{x \in \Delta} \alpha^n_x \{x\}) = \cO_C(-2r A^n) \lra (\det \overline{\cV}_- )^{\otimes 2r},
$$
so that the line bundle on $C$:
$$
\big(\det \overline{\ccH^n(Y_\Cc / \Cc)}_- \big)^{\otimes r} \otimes \cO_C(r A^n)
$$
is of $2$-torsion, as wanted.
\end{proof}

Now, let us show the main result of this section.

\begin{proposition}
\label{GK Poincare duality}
With the same notation as in Lemma \ref{det H^n 2-torsion}, the line bundle on $C$:
$$
\GK_{C,+}(\H^{2(N-1)-n}(Y_\eta/C_\eta))^{\otimes r} \otimes \GK_{C,-}(\H^n(Y_\eta/C_\eta))^{\otimes -r} \otimes \cO_C(-r (N-1) A^n)
$$
is of $2$-torsion.

Moreover, if the local monodromy at every point of $\Delta$ of the VHS $\H^n(Y_\Cc/\Cc)$ on $\Cc$ is unipotent, then so is the local monodromy of the VHS $\H^{2(N-1)-n}(Y_\Cc /\Cc)$, and the line bundle on $C$:
$$
 \GK_{C}(\H^{2(N-1)-n}(Y_\eta/C_\eta)) \otimes \GK_{C}(\H^n(Y_\eta/C_\eta))^\vee
$$
is of $2$-torsion.
\end{proposition}
\begin{proof}
Let us denote by $F^{\bullet} \ccH^n$ (resp. $F^{\bullet} \ccH^{2(N-1)-n}$) the Hodge filtration by locally direct summands on the vector bundle $\ccH^n(Y_\Cc / \Cc)$ (resp. $\ccH^{2(N-1)-n} (Y_\Cc / \Cc)$).

Using Poincar\'{e} and Serre duality, we have an isomorphism of vector bundles on $\Cc$:
$$
\phi_n : \ccH^{2(N-1)-n}(Y_\Cc / \Cc) \lra \ccH^n(Y_\Cc / \Cc) ^\vee,
$$
that sends, for every integer $p$, the subbundle $F^p \ccH^{2(N-1)-n}$ into the orthogonal subbundle:
$$
(F^{N-p} \ccH^n)^\perp \subset \ccH^n(Y_\Cc / \Cc) ^\vee.
$$
Consequently, for every integer $p$, this isomorphism induces by restriction and quotient an isomorphism:
$$
\phi_n^p : \mathrm{gr}_F^p \ccH^{2(N-1)-n} \lra (F^{N-p} \ccH^n)^\perp / (F^{N-1-p} \ccH^n)^\perp \simeq (\mathrm{gr}_F^{N-1-p} \ccH^n)^\vee.
$$
Recall that the duality of vector bundles exchanges the upper and lower Deligne extensions, therefore, the isomorphism $\phi_n$ can be extended into an isomorphism of vector bundles on $C$:
$$
\overline{\phi_n} : \overline{\ccH^{2(N-1)-n} (Y_\Cc / \Cc) } _+ \lra (\overline{\ccH^n (Y_\Cc / \Cc)}_- )^\vee,
$$
that sends for every integer $p$, the subbundle $\overline{F^p 
\ccH^{2(N-1)-n}}_+$ into the 
subbundle $(\overline{F^{N-p} \ccH^n}_-)^\perp$, hence induces an isomorphism by restriction and quotient:
$$
\overline{\phi_n^p} : \mathrm{gr}_{\overline{F}_+} ^p \overline{\ccH^{2(N-1)-n}}_+ \lra (\mathrm{gr}_{\overline{F}_-}^{N-1-p} \overline{\ccH^n}_-)^\vee.
$$
By definition, the upper Griffiths line bundle on $C$ of the VHS $\H^{2(N-1)-n}(Y_\Cc / \Cc)$ on $\Cc$ is given by:
\begin{align*}
\GK_{C,+}(\H^{2(N-1)-n}(Y_\eta/C_\eta))
&= \bigotimes_{p \in \Z} (\det \mathrm{gr}_{\overline{F}_+} ^p \overline{\ccH^{2(N-1)-n}}_+ )^{\otimes p}\\
&\simeq \bigotimes_{p \in \Z} (\det \mathrm{gr}_{\overline{F}_-}^{N-1-p} \overline{\ccH^n}_-)^{\otimes -p} \textrm{ using the isomorphisms $(\overline{\phi_n^p})_p$}\\
&\simeq \bigotimes_{p' \in \Z} (\det \mathrm{gr}_{\overline{F}_-}^{p'} \overline{\ccH^n}_- )^{\otimes p' - (N-1)}\\
&\simeq \bigotimes_{p' \in \Z} (\det \mathrm{gr}_{\overline{F}_-}^{p'} \overline{\ccH^n}_- )^{\otimes p'} \otimes \big(\bigotimes_{p' \in \Z} (\det \mathrm{gr}_{\overline{F}_-}^{p'} \overline{\ccH^n}_- ) \big)^{\otimes - (N-1)}\\
&\simeq \GK_{C,-}(\H^n(Y_\eta / C_\eta)) \otimes (\det \overline{\ccH^n(Y_\Cc / \Cc)}_-  )^{\otimes -(N-1)}.
\end{align*}
Consequently, there is an isomorphism of line bundles on $C$:
\begin{equation*}
\GK_{C,+}(\H^{2(N-1)-n}(Y_\eta/C_\eta)) \otimes \GK_{C,-}(\H^n(Y_\eta / C_\eta))^\vee 
\simeq \big(\det \overline{\ccH^n(Y_\Cc / \Cc)}_- \big)^{\otimes -(N-1)}.
\end{equation*}

Using Lemma \ref{det H^n 2-torsion}, the line bundle on $C$:
$$
\big(\det \overline{\ccH^n(Y_\Cc / \Cc)}_- \big)^{\otimes r} \otimes \cO_C(r A^n)
$$
is of $2$-torsion. Consequently, the line bundle on $C$:
\begin{align*}
\GK_{C,+}(&\H^{2(N-1)-n}(Y_\eta/C_\eta))^{\otimes r} \otimes \GK_{C,-}(\H^n(Y_\eta / C_\eta))^{\otimes (-r)} \otimes \cO_C(-r (N-1) A^n)\\
&\simeq (\det \overline{\ccH^n(Y_\Cc / \Cc)}_- )^{\otimes - r(N-1)} \otimes \cO_C(-r (N-1) A^n), \\
&\simeq \big [(\det \overline{\ccH^n(Y_\Cc / \Cc)}_- )^{\otimes r} \otimes \cO_C( r A^n) \big ]^{\otimes - (N-1)},
\end{align*}
is of $2$-torsion, as wanted.

For proving the second assertion in Proposition  \ref{GK Poincare duality}, we use that the isomorphism $\phi_n$ is also compatible with the connections of both VHS, hence with their local monodromy. If the local monodromy at every point of $\Delta$ of the VHS $\H^n(Y_\Cc / \Cc)$ is unipotent, so is the local monodromy of the dual VHS $\H^n(Y_\Cc / \Cc) ^\vee$, and so is the monodromy of the VHS $\H^{2(N-1)-n}(Y_\Cc / \Cc)$.

Finally,  applying the first part of the proof with $A^n = 0$ and $r = 1$, we obtain the last assertion in Proposition  \ref{GK Poincare duality}.
\end{proof}

\chapter[Characteristic classes of logarithmic relative differentials]{Characteristic classes of relative differentials and of logarithmic relative differentials}
\label{section comparing Omega^1 and omega^1}

In this chapter, we establish various results relating the Chern classes of sheaves of differentials, and vector bundles of logarithmic differentials. These results will be used in Chapter \ref{Alternating product of Griffiths line bundles and the Grothendieck-Riemann-Roch formula} to derive Theorems \ref{intro GK DNC with localized terms}  and  \ref{GK in terms of classes in D non-reduced} from Theorems \ref{intro GK DNC} and \ref{GK and rho_r}, and are also of independent interest. 

The statements of our results are gathered in Section \ref{The statements}. The following sections, which contain their rather technical proofs and will not be referred to in the remainder of this paper, could be skipped at first reading.

\section[Chern and Todd classes of differentials and logarithmic differentials]{Comparing Chern classes and Todd classes of differentials and logarithmic differentials}
\label{The statements}

\subsection{Chern classes of differentials and of logarithmic differentials: the absolute case}
Although we are ultimately interested in characteristic classes of sheaves of relative differentials, we shall first prove the following results in the absolute case.

Let $X$ be a smooth $n$-dimensional complex scheme and let $E$ be a reduced divisor with strict normal crossings in $X$. Let us denote by $(E_i)_{i \in I}$ the irreducible components of $E$, and for any subset $J \subseteq I$, let us define:\footnote{In particular, $E_\emptyset = X$.}
$$E_J := \bigcap_{i \in J} E_i.$$
It is a smooth subscheme of codimension $\vert J \vert$ in $X$, and we shall denote by:
$$i_{E_J}: E_J \hlra X$$
the inclusion morphism.

As in \cite[II, 3.4]{Deligne70}, for every  non-negative integer $r$, we denote by  $E^r$ the subscheme of codimension $r$ of $X$ defined as the union of  the intersections of $r$ distinct components of $E$:
$$
E^r := \bigcup_{J \subset I, |J| = r} E_J.
$$
Observe that for every subset $J$ in $I$, the subscheme:
$$
E_J \cap E^{|J| + 1}=\bigcup_{i \in I - J} E_{J \cup \{i\}}
$$
is a divisor with normal crossings in $E_J$.

\begin{proposition}
\label{comparison Omega log and Omega easy cases general}
With the previous notation, for every subset $J$ of $I$, the following relation holds between total Chern classes in $\CH^\ast(E_J)$:
\begin{equation}\label{eq Omega (log) _E_J}
c\big(\Omega^1_X (\log E)_{\vert E_J}\big) = c\big(\Omega^1_{E_J} (\log E_J \cap E^{\vert J \vert +1})\big).
\end{equation}
Moreover the following equality holds in $\CH^\ast(X)$:
\begin{equation}\label{Chern class Omega1log}
c(\Omega_X^1(\log E)) = \sum_{J \subseteq I} i_{E_J \ast} c(\Omega^1_{E_J}),
\end{equation}
and the following equality holds in $\CH^\ast(X)_\Q$:
\begin{equation} \label{Td Omega1log}
    \Td(\Omega^1_X (\log E)) = \Td(\Omega^1_X) \prod_{i \in I} \td(-[E_i])^{-1},
\end{equation}
where $\td(x)$ denotes the formal series $x/(1 - e^{-x})$.
\end{proposition}

The relation \eqref{Chern class Omega1log} may be reformulated as the following equalities, where $r$ denotes an integer  such that $0 \leq r \leq n$:
\begin{equation}
c_r(\Omega^1_X(\log E)) = \sum_{J \subseteq I, \vert J \vert \leq r} i_{E_J \ast} c_{r - \vert J \vert}(\Omega^1_{E_J}).
\end{equation}

\subsection{Notation}
\label{Notation Chern classes omega^1 strict DNC}

For the rest of the statements, let us adopt the following notation.

Let $Y$ be a connected smooth projective complex scheme of pure dimension $N$, let $C$ be a connected smooth projective complex curve, and let
$$
g : Y \lra C
$$
be a surjective morphism that is smooth outside of $Y_\Delta$, where $\Delta$ is a reduced divisor in $C$.

Let us assume that the divisor $Y_\Delta$ is a divisor with strict normal crossings, and let us write it as:
$$
D = \sum_{i \in I} m_i D_i,
$$
where $I$ is a finite set and for every $i$ in $I$, $m_i$ is a positive integer, and where $D_i$ is a smooth connected divisor such that the $(D_i)_{i \in I}$ intersect each other transversally.

As before, for every subset $J$ of $I$, we denote:
$$
D_J := \bigcap_{i \in J} D_i,
$$
and  for every integer $r \geq 1$,  we define the subscheme:
$$
D^r := \bigcup_{J \subset I, |J| = r} D_J.
$$
The subscheme:
$$
D_J \cap D^{|J| + 1}=\bigcup_{i \in I - J} D_{J \cup \{i\}}
$$
is a divisor with normal crossings in $D_J$.

For every element $i$ in $I$, we denote by
$$
\cN_i := \cN_{D_i} Y \simeq \cO_Y(D_i) _{| D_i}
$$
 the normal vector bundle of the smooth divisor $D_i$ in $Y$.

Let $\preceq$ be a total order on $I$.

\subsection{Normal bundles to the strata $D_J$ and Chern classes}

\begin{proposition}
\label{omega^1 restricted to a strata}
For every non-empty subset $J$ in $I$, we have the equality of total Chern classes in $\CH^\ast(D_J)$:
\begin{equation}
\label{eq omega^1 restricted to a strata}
c\big(\omega^1_{Y/C | D_J}\big) = c\big(\Omega^1_{D_J} (\log D_J \cap D^{|J| + 1})\big).
\end{equation}
\end{proposition}

\begin{proposition}
\label{relation fibres DCN}
Let $i$ be an element in $I$, we have the equality in $\CH^1(D_i)$:
\begin{equation*}
    m_i c_1(\cN_i) = - \sum_{j \in I - \{i\}} m_j [D_{ij}],
\end{equation*}
where $D_{ij} := D_{\{i,j\}}.$
\end{proposition}

\begin{corollary}
\label{first consequence relation DCN}
Let $i$ be an element in $I$ and let $\beta$ be a class in $\CH^\ast(D_i)$, we have the equality in $\CH^\ast(Y)$:
\begin{equation*}
m_i i_{D_i \ast} (c_1(\cN_i) \cap  \beta) = - \sum_{j \in I - \{i\}} m_j i_{D_{ij} \ast} \beta_{| D_{ij}}.
\end{equation*}
\end{corollary}

\begin{corollary}
\label{N_i _D_ij and N_j _D_ij in terms of N_ij}
Let $(i,j)$ be a pair in $I^2$ such that $i \prec j$.

We have the equality in $\CH^1(D_{ij})$:
\begin{equation}
m_i c_1(\cN_{i | D_{ij}}) + m_j c_1(\cN_{j | D_{ij}}) = - \sum_{k\in I -\{i,j\}} m_k [D_{ijk}], \label{first equality N_i in terms of N_ij}
\end{equation}
where $D_{ijk} := D_{\{i,j,k\}}.$
\end{corollary}

\subsection{Comparing the characteristic classes of $[T_g]$ and ${\omega^{1\vee}_{Y/C}}$}
\label{Comparing the characteristic classes of T and omega}

Finally, we shall show the following results of comparison of the Chern classes of $\omega^1_{Y/C}$ and of the relative tangent class in $K$-theory\footnote{In the notation of  \cite[Expos\'e VIII]{SGA6} or \cite[B.7.6]{Fulton98}, this class would be denoted by $T_g$.}:
$$[T_g] := [T_Y] - g^\ast [T_C].$$
The duality operation on vector bundles on $Y$ extends to an involution $.^\vee$ of the abelian group $K^0(Y)$, which sends $[T_g]$ to the class $[\Omega^1_Y] - g^\ast [\Omega^1_C]$.
Observe that we have a sequence of coherent sheaves on $Y$:
$$0 \lra g^\ast \Omega^1_C \overset{D g}{\lra} \Omega^1_Y \lra \Omega^1_{Y/C} \lra 0.$$
This sequence is exact: everything but the injectivity of the differential morphism $D g$ is a standard property of the sheaf of K\"{a}hler differentials, and since the morphism $g$ is smooth on a dense open subset of $Y$, the morphism $D g$ does not vanish on a dense open subset, hence is injective.

Consequently, the following equality holds in $K^0(Y)$\footnote{As discussed in \cite[Expos\'e VIII, Section 2]{SGA6}, the class $[T_g]$ is the class in $K^0(Y)$ of the relative tangent complex $\mathbb{T}_g$, defined as the dual of the cotangent complex $\mathbb{L}_g$. The argument in the last paragraph actually proves that $\mathbb{L}_g$ is quasi-isomorphic to the sheaf of K\"ahler differentials $\Omega^1_g = \Omega^1_{Y/C}$.}:
$$[T_g] = [\Omega^1_{Y/C}]^\vee.$$

\begin{proposition}
\label{first terms Td/Td general}
We have the equality in $\CH^\ast(Y)_\Q$:
\begin{equation}
\label{Td/Td general}
\frac{\Td([T_g])}{\Td(\omega^{1 \vee}_{Y/C})}
= \prod_{i \in I} \td([D_i]) - 1/2 \sum_{i \in I} m_i [D_i],
\end{equation}
where $\td(x)$ denotes the formal series with rational coefficients $x/(1 - e^{-x})$.

In particular, we have the equality in $\CH^1(Y)_\Q$:
\begin{equation}
\label{Td/Td degree 1 general}
    \bigg [\frac{\Td([T_g])}{\Td(\omega^{1 \vee}_{Y/C})} \bigg ]^{(1)}
    = - 1/2 \sum_{i \in I} (m_i - 1) [D_i],
\end{equation}
and the equalities in $\CH^2(Y)_\Q$:
\begin{align}
   \bigg[\frac{\Td([T_g])}{\Td(\omega^{1 \vee}_{Y/C})} \bigg]^{(2)} &= 1/12 \sum_{i \in I} i_{D_i \ast} c_1(\cN_i) + 1/4 \sum_{(i,j) \in I^2, i \prec j} [D_{ij}], \label{eq Td/Td degree 2 with c_1(N_i)} \\
    &= 1/12 \sum_{(i,j) \in I^2, i \prec j} (3 - m_i/m_j - m_j/m_i) [D_{ij}]. \label{eq Td/Td degree 2 with multiplicities}
\end{align}
\end{proposition}
\begin{proposition} 
\label{comparison c_r when D^3 empty non-reduced}
For every non-negative integer $r$, we have the equality in $\CH^r(Y)$:
\begin{equation}
    \label{eq c_r omega^1 general}
c_r(\omega^{1 \vee}_{Y/C}) = c_r([T_g])
 + \sum_{i \in I} m_i i_{D_i \ast} [c_{r-1}(T_{ D_i}) + c_1(\cN_i) c_{r-2}(T_{D_i}) ]
+ \sum_{\emptyset \neq J \subset I} (-1)^{|J|} i_{D_J \ast} c_{r - |J|}(T_{D_J}).
\end{equation}
Furthermore, if the subscheme $D^3$ is empty, then for every non-negative integer $r$, we have the equality in $\CH^r(Y)$:
\begin{equation}
c_r(\omega^{1 \vee}_{Y/C})
= c_r([T_g])
 + \sum_{i \in I} (m_i - 1) i_{D_i \ast} c_{r-1}(T_{D_i})
- \sum_{(i,j) \in I^{2}, i \prec j} (m_i + m_j - 1) i_{D_{ij} \ast} c_{r-2}(T_{D_{ij}}).
\label{eq c_r omega and T_g when D^3 empty}
\end{equation}
\end{proposition}

The remainder of this chapter shall be devoted to the proofs of these results.

More specifically, Section \ref{Comparing Chern classes of differentials and of logarithmic differentials in the absolute case}  shall be devoted to the proof of Proposition \ref{comparison Omega log and Omega easy cases general}; Section \ref{Normal bundles to the strata of the singular fibers and Chern classes} shall be devoted to the proofs of Propositions \ref{omega^1 restricted to a strata}, \ref{relation fibres DCN} and Corollaries \ref{first consequence relation DCN} and \ref{N_i _D_ij and N_j _D_ij in terms of N_ij}; and Section \ref{Characteristic classes of T_g and omega^1 ^vee _Y/C} shall be devoted to the proofs of Propositions \ref{first terms Td/Td general} and \ref{comparison c_r when D^3 empty non-reduced}.

\section[Characteristic classes of logarithmic differentials: the absolute case]{Characteristic classes of differentials and logarithmic differentials: the absolute case }
\label{Comparing Chern classes of differentials and of logarithmic differentials in the absolute case}
Let us adopt the notation of Proposition \ref{comparison Omega log and Omega easy cases general}.

\subsection{Proof of \eqref{eq Omega (log) _E_J} }

Let $J$ be a subset of $I$. Let
$$
\cN_J := \cN_{E_J} X
$$
be the normal vector bundle of the subscheme $E_J$ in $X$.

One checks in local coordinates that the composition of morphisms of vector bundles on $E_J$:
$$
\cN_J ^\vee \lra \Omega^1_{X | E_J} \lra \Omega^1_X (\log E)_{| E_J}
$$
vanishes, so using the commutative diagram
$$
\xymatrix{  0 \ar[r] & \cN_{J}^\vee \ar[r]\ar[dr]^{0} &  \Omega^1_{X | E_J} \ar[r]\ar[d]  & \Omega^1_{E_J} \ar[r] & 0  \\
                     &                       & \Omega^1_X (\log E)_{| E_J}            &          & }
$$
and the exactness of the horizontal exact sequence of vector bundles, we obtain a map of vector bundles on $E_J$:
$$
\alpha_J : \Omega^1_{E_J} \lra \Omega^1_X (\log E)_{| E_J}.
$$
One checks in local coordinates that the map $\alpha_J$ factors through a map of vector bundles:
$$
\tilde{\alpha_{J}} : \Omega^1_{E_J} (\log E_J \cap E^{|J|+1}) \lra \Omega^1_X (\log E)_{| E_J}.
$$
On the other hand, let us define a map of vector bundles on $E_J$ by:
$$
\beta_{J} := (\mathrm{Res}_{E_i})_{i \in J} : \quad \Omega^1_X (\log E) _{| E_J} \lra \cO_{E_J}^{J} \simeq \cO_{E_J}^{\oplus |J|},
$$
where, for every $i$ in $J$, $\mathrm{Res}_{E_i}$ denotes the residue morphism relative to the divisor $E_i$.

Equality \eqref{eq Omega (log) _E_J} follows immediately from the following proposition.

\begin{proposition}
\label{exact sequence Omega _E_J }
The morphisms $\tilde{\alpha_J}$ and $\beta_J$ fit together in an exact sequence of vector bundles on $E_J$:
$$
0 \lra \Omega^1_{E_J} (\log E_J \cap E^{|J|+1}) \overset{\tilde{\alpha_J}}{\lra} \Omega^1_X (\log 
E)_{| E_J} \overset{\beta_J}{\lra} \cO_{E_J}^{J} \lra 0.
$$
\end{proposition}

Observe that this proposition follows from a log-structure argument (see \cite[Lemma 5.3.6]{Kato-Saito04}). Indeed we may assume that $J$ is non-empty and restrict ourselves to a formal neighborhood $U$ of the divisor $E$ in $X$, and consider the morphism of complex schemes:
$$g : U \lra S := \Spec \C[[t]]$$
defined by any equation of the divisor $E$ in $U$. With the notation of \cite[5.3]{Kato-Saito04}, the $\cO_U$-module $\Omega^1_{U/S}(\log)$ is precisely the quotient $\Omega^1_U(\log E)/g^\ast \Omega^1_S$.

Denoting by $s$ the closed point in $S$, the morphism of $\cO_{E_J}$-modules
$$(g^\ast \Omega^1_S)_{\mid E_J} \lra \Omega^1_U(\log E)_{\mid E_J}$$
factors through the morphism
$$(g^\ast \Omega^1_S)_{\mid E_J} \lra (g^\ast \Omega^1_S(\{s\}))_{\mid E_J}$$
which vanishes because $g$ maps $E_J$ to $s$, so that the canonical surjective morphism:
$$\Omega^1_U(\log E)_{\mid E_J} \lra \Omega^1_{U/S} (\log)_{\mid E_J} = \Omega^1_U(\log E)_{\mid E_J}/(g^\ast \Omega^1_S)_{\mid E_J}$$
is an isomorphism. 

Under this isomorphism, \cite[Lemma 5.3.6, 2]{Kato-Saito04} becomes Proposition \ref{exact sequence Omega _E_J }.

We present below a more direct proof for the convenience of the reader.

\begin{proof}
We can check this on a neighborhood of a point $P$ of $E_J$. Let $i_1,...,i_s$ be the elements of $J$, with $s := |J|$.

Let $r$ be the largest integer such that $P$ is in $E^r$, in particular, $s \leq r \leq n$. Let $E_{i_1}$,..., $E_{i_s}$, $E_{i_{s+1}}$,...,$E_{i_r}$ be the components of $E$ that meet in $P$, with $\{i_{s+1},...,i_r\}$ a subset of $I - J$.

Let us consider local coordinates $(x_1,...,x_n)$ of $X$ centered in $P$, such that for every integer $1 \leq l \leq r$, $E_{i_l}$ is locally defined by $(x_l = 0)$.

The vector bundle $\Omega^1_{E_J}$ admits a local frame $(\mathrm{d} x_{s+1},...,\mathrm{d} x_n)$, the vector bundle $\Omega^1_{X | E_J}$ admits a local frame ($\mathrm{d} x_1$,...,$\mathrm{d} x_n$), the vector bundle $\Omega^1_X (\log E)_{| E_J}$ admits a local frame ($\mathrm{d} x_1/x_1$,...,$\mathrm{d} x_r/x_r$, $\mathrm{d} x_{r+1}$,...,$\mathrm{d} x_n$), and the canonical map from $\Omega^1_{X | E_J}$ to $\Omega^1_X (\log E)_{| E_J}$ sends, for every integer $l$, $\mathrm{d} x_l$ to $x_l (\mathrm{d} x_l/x_l)$ if $l \leq r$, else to $\mathrm{d} x_l$.

Consequently, the map
$$
\alpha_J : \Omega^1_{E_J} \lra \Omega^1_X (\log E) _{| E_J}
$$
sends, for every integer $l \geq s+1$, $\mathrm{d} x_l$ to $x_l (\mathrm{d} x_l/x_l)$ if $l \leq r$, else to $\mathrm{d} x_l$.

The divisor $E_J \cap E^{|J|+1}$ in $E_J$ is locally defined by the equation $(x_{s+1} ... x_r = 0)$, so the vector bundle $\Omega^1_{E_J} (\log E_J \cap E^{|J|+1})$ admits a local frame $(\mathrm{d} x_{s+1}/x_{s+1},...,\mathrm{d} x_r/x_r, \mathrm{d} x_{r+1},...,\mathrm{d} x_n)$.

Consequently, the map
$$
\tilde{\alpha_{J}} : \Omega^1_{E_J} (\log E_J \cap E^{|J|+1}) \lra \Omega^1_X (\log E) _{| E_J}
$$
sends, for every integer $s+1 \leq l \leq r$, $\mathrm{d} x_l/x_l$ to itself, and for every integer $l > r$, $\mathrm{d} x_l$ to itself. It is an injective morphism of vector bundles, whose image is the vector subbundle of $\Omega^1_X (\log E)_{| E_J}$ generated by $(\mathrm{d} x_l/x_l)_{s+1 \leq l \leq r}$ and by $(\mathrm{d} x_l)_{l > r}$.

On the other hand, the morphism of vector bundles:
$$
\beta_{J} : \Omega^1_X (\log E) _{| E_J} \lra \cO_{E_J}^{J}
$$
sends, for every integer $l$ such that $1 \leq l \leq s$, $\mathrm{d}x_l/x_l$ to $(0,...,1,...,0)$ with the $1$ in $i_l$-th position; for every integer $l$ such that $s+1 \leq l \leq r $, $\mathrm{d}x_l/x_l$ to $0$, and for every integer $l$ such that $l > r$, $\mathrm{d} x_l$ to $0$.

Consequently, it is surjective, and its kernel is the vector subbundle of $\Omega^1_{X} (\log E)_{ | E_J}$ generated by $(\mathrm{d} x_l/x_l)_{s+1 \leq l \leq r}$ and by $(\mathrm{d} x_l)_{l > r}$, which is exactly the image of $\tilde{\alpha_J}$.

So the sequence is exact, as wanted.
\end{proof}

\subsection{Proof of \eqref{Chern class Omega1log}}

We shall reason by induction over the dimension $n$ of $X$.

If $n = 0$, then $X$ is a finite scheme and the divisor $E$ is empty, so that equality \eqref{Chern class Omega1log} holds.

Let $n \geq 1$ be an integer, let us assume that equality \eqref{Chern class Omega1log} holds for schemes of dimension at most $n-1$, and let us prove it for a scheme $X$ of dimension $n$.

One easily checks in local coordinates that we have an exact sequence of coherent sheaves on $X$:
\begin{equation} \label{exact sequence Omega Omega (log)}
0 \lra \Omega^1_X \overset{u}{\lra} \Omega^1_X (\log E) \overset{(\mathrm{Res}_{E_i})_{i \in I}}{\lra} \bigoplus_{i \in I} i_{E_i \ast} \cO_{E_i} \lra 0,
\end{equation} 
where $u$ is the canonical inclusion morphism, and for every $i$ in $I$, 
$\mathrm{Res}_{E_i}$ is the residue morphism.

Consequently, we have the equality in $\CH^\ast(X)$:
\begin{equation}
c(\Omega^1_X) = c(\Omega^1_X (\log E)) \prod_{i \in I} c([i_{E_i \ast} \cO_{E_i}])^{-1}.
\label{first equality comparison Omega Omega (log)}
\end{equation}
For every element $i$ in $I$, we have an exact sequence of coherent sheaves on $X$:
\begin{equation}
0 \lra \cO_X(-E_i) \lra \cO_X \lra i_{E_i \ast} \cO_{E_i} \lra 0, \label{exact sequence O_E_i}
\end{equation}
hence the equality in $\CH^\ast(X)$:
$$
c([i_{E_i \ast} \cO_{E_i}])^{-1} = c(\cO_X(-E_i))
= 1 - [E_i].
$$
Hence replacing in \eqref{first equality comparison Omega Omega (log)}:
\begin{align}
c(\Omega^1_X) &= c(\Omega^1_X (\log E))  \prod_{i \in I} (1 - [E_i]),\nonumber\\
&= c(\Omega^1_X (\log E))  \sum_{K \subset I} \prod_{i \in K} (- [E_i]),\nonumber\\
&= c(\Omega^1_X (\log E))  \sum_{K \subset I} (-1)^{|K|} [E_K] \textrm{ because the $(E_i)_{i \in I}$ intersect each other transversally,} \nonumber\\
&= \sum_{K \subset I} (-1)^{|K|} i_{E_K \ast} c(\Omega^1_X (\log E)_{| E_K}),\nonumber\\
&= \sum_{K \subset I} (-1)^{|K|} i_{E_K \ast} c(\Omega^1_{E_K} (\log E_K \cap E^{|K|+1})) \textrm{ using equality \eqref{eq Omega (log) _E_J},}\nonumber\\
&= c(\Omega^1_X (\log E)) + \sum_{\emptyset  \neq K \subset I} (-1)^{|K|} i_{E_K \ast} c(\Omega^1_{E_K} (\log E_K \cap E^{|K|+1})). \label{second equality comparison Omega Omega (log)}
\end{align}
Let $K$ be a non-empty subset of $I$. Applying the induction hypothesis to the scheme $E_K$ of dimension at most $n-1$ and to the divisor with strict normal crossings $E_K \cap E^{|K|+1}$, whose components are the $(E_{K \cup \{i\}})_{i \in I - K}$, yields the equality in $\CH^\ast(E_K)$:
$$
c(\Omega^1_{E_K} (\log E_K \cap E^{|K|+1})) = \sum_{K \subset J \subset I} i_{E_J, E_K \ast} c(\Omega^1_{E_J}).
$$
Hence replacing in \eqref{second equality comparison Omega Omega (log)}:
\begin{align*}
c(\Omega^1_X) &= c(\Omega^1_X (\log E)) + \sum_{\emptyset  \neq K \subset J \subset I} (-1)^{|K|} i_{E_J \ast} c(\Omega^1_{E_J}),\\
&= c(\Omega^1_X (\log E)) + \sum_{\emptyset  \neq J \subset I} \Big [\sum_{\emptyset  \neq K \subset J} (-1)^{|K|} \Big ] i_{E_J \ast} c(\Omega^1_{E_J}),\\
&= c(\Omega^1_X (\log E)) + \sum_{\emptyset  \neq J \subset I} \Big [\prod_{i \in J} (1-1) - 1 \Big ] i_{E_J \ast} c(\Omega^1_{E_J}),\\
&= c(\Omega^1_X (\log E)) - \sum_{\emptyset  \neq J \subset I} i_{E_J \ast} c(\Omega^1_{E_J}),
\end{align*}
hence we have the equality in $\CH^\ast(X)$:
\begin{align*}
c(\Omega^1_X (\log E)) &= c(\Omega^1_X) + \sum_{\emptyset  \neq J \subset I} i_{E_J \ast} c(\Omega^1_{E_J}),\\
&= \sum_{J \subset I} i_{E_J \ast} c(\Omega^1_{E_J}),
\end{align*}
which concludes the induction and the proof of equality \eqref{Chern class Omega1log}.

\subsection{Proof of \eqref{Td Omega1log}}
Similarly to the proof of \eqref{Chern class Omega1log}, using 
exact sequences \eqref{exact sequence Omega Omega (log)}, \eqref{exact sequence O_E_i} and the multiplicativity of the Todd class $\Td$, we obtain the equality in $\CH^\ast(X)_\Q$:
\begin{equation*}
\Td(\Omega^1_X (\log E)) = \Td(\Omega^1_X) \prod_{i \in I} \Td([i_{E_i \ast} \cO_{E_i}]) = \Td(\Omega^1_X) \prod_{i \in I} \td(-[E_i])^{-1},
\end{equation*}
as wanted.

\section{Normal bundles to the strata of the singular fibers and Chern classes}
\label{Normal bundles to the strata of the singular fibers and Chern classes}

\subsection{Proof of Proposition \ref{omega^1 restricted to a strata}}

Let $J$ be  a non-empty subset of $I$. Since the line bundle $\Omega^1_C (\log \Delta)_{| \Delta}$ can be trivialised on $\Delta$, the line bundle $(g^\ast \Omega^1_C (\log \Delta) )_{| D_J}$ can be trivialised on $D_J$, and we obtain the equality in $\CH^\ast(D_J)$:
$$
c(g^\ast \Omega^1_C (\log \Delta)) _{| D_J} = 1,
$$
hence by definition of the relative logarithmic vector bundle $\omega^1_{Y/C}$ on $Y$, the equality in $\CH^\ast(D_J)$:
$$c(\omega^1_{Y/C})_{| D_J} = c(\Omega^1_Y (\log Y_\Delta)) _{| D_J},$$
hence applying equality \eqref{eq Omega (log) _E_J} from Proposition \ref{comparison Omega log and Omega easy cases general} to the scheme $X := Y$ and the reduced divisor with strict normal crossings $E := |D| = \sum_{i \in I} D_i$, and to the subset $J$ in $I$:
$$ c(\omega^1_{Y/C})_{| D_J} = c(\Omega^1_{D_J} (\log D_J \cap D^{|J|+1}) ),$$
as wanted.

\subsection{Proofs of Proposition \ref{relation fibres DCN} and Corollaries \ref{first consequence relation DCN} and \ref{N_i _D_ij and N_j _D_ij in terms of N_ij}}
Let $i$ be an element of $I$. Since $\Delta$ is a scheme of dimension $0$, the line bundle $\cN_\Delta C$ on $\Delta$ can be trivialized. Consequently, we can trivialize the following line bundle on $D_i$:
\begin{align*}
(g^\ast \cN_\Delta C)_{| D_i} &\simeq \cO_Y(Y_\Delta) _{| D_i}\\
&\simeq \cO_Y(m_i D_i)_{| D_i} \otimes \bigotimes_{j \in I - \{i\}} \cO_Y(m_j D_j) _{| D_i}\\
&\simeq \cN_i^{\otimes m_i} \otimes \bigotimes_{j \in I - \{i\}} \cO_{D_i}(m_j D_{ij} ).
\end{align*}
Taking Chern classes, we obtain the following equality in $\CH^1(D_i)$:
\begin{equation}
\label{eq relation fibres DCN}
m_i c_1(\cN_i) = - \sum_{j \in I - \{i\}} m_j [D_{ij}],
\end{equation}
which shows Proposition \ref{relation fibres DCN}.

If $\beta$ is a class in $\CH^\ast(D_i)$, intersecting \eqref{eq relation fibres DCN} with $\beta$ then pushing forward by the inclusion $i_{D_i}$ yields Corollary \ref{first consequence relation DCN}.

Now, let $(i,j)$ be a pair in $I^2$ such that $i \prec j$. Applying equality \eqref{eq relation fibres DCN} with the element $i$ then restricting to $D_{ij}$ yields the equality in $\CH^1(D_{ij})$:
$$
m_i c_1(\cN_{i | D_{ij}}) = - \sum_{k \in I - \{i\}} m_k [D_{ik}]_{| D_{ij}},
$$
where $[D_{ik}]$ is seen as a divisor in $D_i$.

If $k \neq j$, this divisor intersects transversally the divisor $D_{ij}$, so that we have the equality in $\CH^1(D_{ij})$:
$$[D_{ik}]_{| D_{ij}} = [D_{ik} \cap D_{ij}] = [D_{ijk}].$$
If $k = j$, the self-intersection of the divisor $D_{ik}$ is given by:
$$[D_{ik}]_{| D_{ij}} = c_1(\cN_{D_{ij}} D_i) = c_1(\cN_{j | D_{ij}}).$$
We have used the canonical isomorphism of line bundles on $D_{ij}$:
$$\cN_{D_{ij}} D_i \simeq \cN_{j | D_{ij}}.$$
Hence replacing, we have the equality in $\CH^1(D_{ij})$:
$$
m_i c_1(\cN_{i | D_{ij}}) = - m_j c_1(\cN_{j | D_{ij}}) - \sum_{k \in I - \{i,j\}} m_k [D_{ijk}],
$$
which shows Corollary \ref{N_i _D_ij and N_j _D_ij in terms of N_ij}.

\section{Characteristic classes of $[T_g]$ and $\omega^{1 \vee} _{Y/C}$}
\label{Characteristic classes of T_g and omega^1 ^vee _Y/C}

\subsection{Proof of Proposition \ref{first terms Td/Td general}}

Using equality \eqref{Td Omega1log} from Proposition \ref{comparison Omega log and Omega easy cases general} applied to the scheme $X := Y$ and the reduced divisor with strict normal crossings $E := |D| = \sum_{i \in I} D_i$, we have the equality in $\CH^\ast(Y)_\Q$:
$$
\Td(\Omega^1_Y (\log D)) = \Td(\Omega^1_Y) \prod_{i \in I} \td(-[D_i])^{-1},
$$
hence the equality:
$$\Td(\Omega^1_Y) = \Td(\Omega^1_Y (\log D))  \prod_{i \in I} \td(-[D_i]),$$ 
hence, using the definition of the relative logarithmic bundle $\omega^1_{Y/C}$ and the multiplicativity of $\Td$:
$$\Td(\Omega^1_Y) =   \Td(\omega^1_{Y/C}) \, g^\ast \Td(\Omega^1_C (\log \Delta)) \, \prod_{i \in I} \td(-[D_i]).$$ 
Since $C$ is a curve, we have a canonical isomorphism of line bundles on $C$:
$$\Omega^1_C (\log \Delta) \simeq \Omega^1_C (\Delta),$$
hence replacing:
\begin{align}
\Td(\Omega^1_Y) &= \Td(\omega^1_{Y/C})\, g^\ast \td(c_1(\Omega^1_C) + [\Delta]) \, \prod_{i \in I} \td(-[D_i]). \nonumber
\end{align}
Taking duals, we obtain the equality in $\CH^\ast(Y)_\Q$:
\begin{equation}
\label{first equality td omega}
    \Td(T_Y) = \Td(\omega^{1 \vee} _{Y/C}) \, g^\ast \td(c_1(T_C) - [\Delta]) \, \prod_{i \in I} \td([D_i]).
\end{equation}
By definition of the relative tangent class $[T_g]$ in $K^0(Y)$, we have the identity in $\CH^\ast(Y)_\Q$:
\begin{align*}
\Td([T_g]) &= \Td(T_Y) g^\ast \Td(T_C)^{-1}, \\
&= \Td(\omega^{1 \vee}_{Y/C}) \, g^\ast [\td(c_1(T_C) - [\Delta])   \td(c_1(T_C))^{-1} ] \prod_{i \in I} \td([D_i]) \textrm{ using \eqref{first equality td omega},}
\end{align*}
hence the equality:
\begin{equation}
\label{second equality Td omega}
\frac{\Td([T_g])}{\Td(\omega^{1 \vee}_{Y/C})} = g^\ast [\td(c_1(T_C) - [\Delta])  \td(c_1(T_C))^{-1} ] \prod_{i \in I} \td([D_i]).
\end{equation}
Since $C$ is a curve, all the classes in $\CH^2(C)_\Q$ vanish. Consequently, we have the equality in $\CH^\ast(C)_\Q$:
\begin{align*}
\td(c_1(T_C) - [\Delta])  \td(c_1(T_C))^{-1} &= [1 + 1/2\, (c_1(T_C) - [\Delta])] [1 + 1/2\,  c_1(T_C)]^{-1}, \\
&= [1 + 1/2 (c_1(T_C) - [\Delta])] [1 - 1/2\ c_1(T_C)], \\
&= 1 - 1/2 \, [\Delta].
\end{align*}
Hence replacing in \eqref{second equality Td omega}, we obtain the equality in $\CH^\ast(Y)_\Q$:
\begin{align}
    \frac{\Td([T_g])}{\Td(\omega^{1 \vee}_{Y/C})} &= g^\ast [1 - 1/2\,  [\Delta]] \prod_{i \in I} \td([D_i]), \nonumber \\
    &= \prod_{i \in I} \td([D_i]) - 1/2\,  (g^\ast [\Delta] ) \prod_{i \in I} \td([D_i]). \label{third equality Td omega}
\end{align}
The line bundle $\cO_C(\Delta)_{| \Delta}$ on $\Delta$ can be trivialized, so for every element $i$ in $I$, the line bundle $(g^\ast \cO_C(\Delta))_{| D_i}$ on $D_i$ can be trivialized. Consequently, for every element $i$ in $I$, we have the equality in $\CH^2(Y)$:
$$(g^\ast [\Delta] ) [D_i] = 0,$$
hence we have the equality in $\CH^\ast(Y)_\Q$:
$$(g^\ast [\Delta] )  \prod_{i \in I} \td([D_i]) = g^\ast [\Delta] = \sum_{i \in I} m_i [D_i].$$
Hence replacing in \eqref{third equality Td omega}, we have the equality in $\CH^\ast(Y)_\Q$:
$$\frac{\Td([T_g])}{\Td(\omega^{1 \vee}_{Y/C})} = \prod_{i \in I} \td([D_i]) - 1/2 \sum_{i \in I} m_i [D_i],$$
which shows equality \eqref{Td/Td general}.

Taking the terms of codimension 1, we obtain the equality in $\CH^1(Y)_\Q$:
\begin{align*}
    \Big [ \frac{\Td([T_g])}{\Td(\omega^{1 \vee}_{Y/C})} \Big ]^{(1)} &= \sum_{i \in I} \td([D_i])^{(1)} - 1/2 \sum_{i \in I} m_i [D_i], \\
    &= \sum_{i \in I} (1/2 [D_i]) - 1/2 \sum_{i \in I} m_i [D_i], \\
    &= - 1/2 \sum_{i \in I} (m_i - 1) [D_i],
\end{align*}
which shows equality \eqref{Td/Td degree 1 general}.

Taking instead the terms of codimension 2, we obtain the equality in $\CH^2(Y)_\Q$:
\begin{align}
    \Big [ \frac{\Td([T_g])}{\Td(\omega^{1 \vee}_{Y/C})} \Big ]^{(2)} &= \big [\prod_{i \in I} \td([D_i]) \big ] ^{(2)}, \nonumber \\
    &= \sum_{i \in I} \td([D_i])^{(2)} + \sum_{(i,j) \in I^2, i \prec j} \td([D_i])^{(1)} . \td([D_j])^{(1)}, \nonumber \\
    &= \sum_{i \in I} (1/12\,  [D_i]^2) + \sum_{(i,j) \in I^2, i \prec j} (1/2\,  [D_i]) . (1/2 [D_j]), \nonumber \\
    &= 1/12\,  \sum_{i \in I} [D_i]^2 + 1/4 \sum_{(i,j) \in I^2, i \prec j} [D_i] . [D_j]. \label{first equality Td/Td degree 2}
\end{align}
For every element $i$ in $I$, the self-intersection of the divisor $D_i$ is given in $\CH^2(Y)$ by:
$$[D_i]^2 = i_{D_i \ast} c_1(\cN_i).$$
On the other hand, for every pair $(i,j)$ in $I^2$ such that $i \prec j$, the divisors $D_i$ and $D_j$ intersect transversally, so that we have the equality in $\CH^2(Y)$:
$$[D_i] . [D_j] = [D_{ij}].$$
Hence replacing in \eqref{first equality Td/Td degree 2}:
$$\Big [ \frac{\Td([T_g])}{\Td(\omega^{1 \vee}_{Y/C})} \Big ]^{(2)} = 1/12 \sum_{i \in I} i_{D_i \ast} c_1(\cN_i) + 1/4 \sum_{(i,j) \in I^2, i \prec j} [D_{ij}],$$
which shows equality \eqref{eq Td/Td degree 2 with c_1(N_i)}.

Finally, for every element $i$ in $I$, applying Proposition \ref{relation fibres DCN} and dividing by $m_i$, we have the equality in $\CH^1(D_i)_\Q$:
$$c_1(\cN_i) = - \sum_{j \in I - \{i\}} m_j/m_i [D_{ij}],$$
hence pushing forward by the inclusion of $D_i$ in $Y$, then summing over $i$, we have the equality in $\CH^2(Y)_\Q$:
$$
\sum_{i \in I} i_{D_i \ast} c_1(\cN_i)
= - \sum_{(i,j) \in I^2, i \prec j} (m_i/m_j + m_j/m_i) [D_{ij}],
$$
hence replacing, we obtain the equality in $\CH^2(Y)_\Q$:
\begin{align*}
    \Big [ \frac{\Td([T_g])}{\Td(\omega^{1 \vee}_{Y/C})} \Big ]^{(2)} &= - 1/12 \sum_{(i,j) \in I^2, i \prec j} (m_i/m_j + m_j/m_i) [D_{ij}] + 1/4 \sum_{(i,j) \in I^2, i \prec j} [D_{ij}], \\
    &= 1/12 \sum_{(i,j) \in I^2, i \prec j} (3 - m_i/m_j - m_j/m_i) [D_{ij}], 
\end{align*}
which shows equality \eqref{eq Td/Td degree 2 with multiplicities}.

This concludes the proof of Proposition \ref{first terms Td/Td general}.

\subsection{Proof of Proposition \ref{comparison c_r when D^3 empty non-reduced}}

Using equality \eqref{Chern class Omega1log} from Proposition \ref{comparison Omega log and Omega easy cases general} applied to the scheme $X := Y$ and the reduced divisor with strict normal crossings $E := |D| = \sum_{i \in I} D_i$, we have the equality in $\CH^\ast(Y)$:
$$c(\Omega^1_Y (\log D)) = c(\Omega^1_Y) + \sum_{\emptyset \neq J \subset I} i_{D_J \ast} c(\Omega^1_{D_J}),$$
hence, by definition of the relative logarithmic bundle $\omega^1_{Y/C}$, the equality in $\CH^\ast(Y)$:
\begin{align}
c(\omega^1_{Y/C}) &= \big[c(\Omega^1_Y) + \sum_{\emptyset \neq J \subset I} i_{D_J \ast} c(\Omega^1_{D_J})\big] \, g^\ast c(\Omega^1_C (\log \Delta))^{-1}. \nonumber 
\end{align}
Since $C$ is a curve, we have a canonical isomorphism of line bundles on $C$:
$$\Omega^1_C (\log \Delta) \simeq \Omega^1_C (\Delta),$$
hence replacing:
\begin{equation*}
c(\omega^1_{Y/C}) = \big[c(\Omega^1_Y) + \sum_{\emptyset \neq J \subset I} i_{D_J \ast} c(\Omega^1_{D_J})\big] \, g^\ast (1 + c_1(\Omega^1_C) + [\Delta])^{-1},
\end{equation*}
hence taking duals:
\begin{equation}
    c(\omega^{1 \vee}_{Y/C}) = \big[c(T_Y) + \sum_{\emptyset \neq J \subset I} (-1)^{|J|} i_{D_J \ast} c(T_{D_J})\big]\,  g^\ast (1 + c_1(T_C) - [\Delta])^{-1}. \label{first equality computation c_r(omega^1) general}
\end{equation}
The line bundles $T_{C | \Delta}$ and $\cO_C(\Delta)_{| \Delta}$ on $\Delta$ can be trivialized, so for every non-empty subset $J$ in $I$, the line bundles $(g^\ast T_C)_{| D_J}$ and $(g^\ast \cO_C(\Delta))_{| D_J}$ on $D_J$ can be trivialized. Consequently, for every non-empty subset $J$ in $I$, we have the equalities in $\CH^\ast(Y)$:
$$[i_{D_J \ast} c(T_{D_J}) ] (g^\ast c_1(T_C) ) = [i_{D_J \ast} c(T_{D_J}) ] (g^\ast [\Delta]) = 0,$$
hence we have the equality in $\CH^\ast(Y)$:
$$\bigg[\sum_{\emptyset \neq J \subset I} (-1)^{|J|} i_{D_J \ast} c(T_{D_J})\bigg] \, g^\ast (1 + c_1(T_C) - [\Delta])^{-1} = \sum_{\emptyset \neq J \subset I} (-1)^{|J|} i_{D_J \ast} c(T_{D_J}).$$
Hence replacing in \eqref{first equality computation c_r(omega^1) general}, we obtain the equality:
\begin{equation}
    c(\omega^{1 \vee}_{Y/C}) = c(T_Y) g^\ast (1 + c_1(T_C) - [\Delta])^{-1} + \sum_{\emptyset \neq J \subset I} (-1)^{|J|} i_{D_J \ast} c(T_{D_J}). \label{second equality computation c_r(omega^1) general}
\end{equation}
Since $C$ is a curve, all the classes in $\CH^2(C)$ vanish. Consequently, we have the equality in $\CH^\ast(C)$:
$$(1 + c_1(T_C) - [\Delta])^{-1} = (1 + c_1(T_C))^{-1} + [\Delta].$$
Hence replacing in \eqref{second equality computation c_r(omega^1) general}:
\begin{align}
    c(\omega^{1 \vee}_{Y/C}) &= c(T_Y) g^\ast [(1 + c_1(T_C))^{-1} + [\Delta]] + \sum_{\emptyset \neq J \subset I} (-1)^{|J|} i_{D_J \ast} c(T_{D_J}), \nonumber \\
    &= c(T_Y) g^\ast c(T_C)^{-1} + c(T_Y) \big(\sum_{i \in I} m_i [D_i]\big) + \sum_{\emptyset \neq J \subset I} (-1)^{|J|} i_{D_J \ast} c(T_{D_J}), \nonumber \\
    &= c([T_g]) + \sum_{i \in I} m_i i_{D_i \ast} c(T_{Y | D_i}) + \sum_{\emptyset \neq J \subset I} (-1)^{|J|} i_{D_J \ast} c(T_{D_J}) \textrm{ by definition of the class $[T_g]$.} \nonumber 
\end{align}
Let $r$ be a non-negative integer, taking the terms of codimension $r$, we obtain the equality in $\CH^r(Y)$:
\begin{equation}
c_r(\omega^{1 \vee}_{Y/C}) =  c_r([T_g]) + \sum_{i \in I} m_i i_{D_i \ast} c_{r-1}(T_{Y | D_i}) + \sum_{\emptyset \neq J \subset I} (-1)^{|J|} i_{D_J \ast} c_{r-|J|}(T_{D_J}).
 \label{third equality computation c_r(omega^1) general}
\end{equation}
For every element $i$ in $I$, we have an exact sequence of vector bundles on $D_i$:
$$0 \lra T_{D_i} \lra T_{Y | D_i} \lra \cN_i \lra 0,$$
hence the equality in $\CH^\ast(D_i)$:
$$c(T_{Y | D_i}) = c(T_{D_i}) (1 + c_1(\cN_i)),$$
hence the equality in $\CH^{r-1}(D_i)$,
$$c_{r-1}(T_{Y | D_i}) = c_{r-1}(T_{D_i}) + c_1(\cN_i) c_{r-2}(T_{D_i}).$$
Hence replacing in \eqref{third equality computation c_r(omega^1) general}, we obtain the equality in $\CH^r(Y)$:
$$c_r(\omega^{1 \vee}_{Y/C}) =  c_r([T_g]) + \sum_{i \in I} m_i i_{D_i \ast} [c_{r-1}(T_{ D_i}) + c_1(\cN_i) c_{r-2}(T_{D_i})] + \sum_{\emptyset \neq J \subset I} (-1)^{|J|} i_{D_J \ast} c_{r-|J|}(T_{D_J}),
$$
which shows equality \eqref{eq c_r omega^1 general}.

Now, let us assume that the subscheme $D^3$ is empty. In this case, equality \eqref{eq c_r omega^1 general} can be rewritten:
\begin{align}
c_r(\omega^{1 \vee}_{Y/C}) &=  c_r([T_g]) + \sum_{i \in I} m_i i_{D_i \ast} [c_{r-1}(T_{ D_i}) + c_1(\cN_i) c_{r-2}(T_{D_i})] \nonumber \\
&-  \sum_{i \in I} i_{D_i \ast} c_{r-1}(T_{D_i}) + \sum_{(i,j) \in I^2, i \prec j} i_{D_{ij} \ast} c_{r-2}(T_{D_{ij}}), \nonumber \\
&= c_r([T_g]) + \sum_{i \in I} (m_i-1) i_{D_i \ast} c_{r-1}(T_{ D_i}) \nonumber \\
&+ \sum_{i \in I} m_i i_{D_i \ast} (c_1(\cN_i) c_{r-2}(T_{D_i})) + \sum_{(i,j) \in I^2, i \prec j} i_{D_{ij} \ast} c_{r-2}(T_{D_{ij}}). \label{first eq c_r omega^1 D^3 empty}
\end{align}
For every element $i$ in $I$, applying Corollary \ref{first consequence relation DCN} to the class $\beta := c_{r-2}(T_{D_i})$ in $\CH^{r-2}(D_i)$ yields the equality in $\CH^r(Y)$:
\begin{equation}\label{sum mult Ni}
m_i i_{D_i \ast} (c_1(\cN_i) c_{r-2}(T_{D_i})) = - \sum_{j \in I - \{i\}} m_j i_{D_{ij} \ast} c_{r-2}(T_{D_i | D_{ij}}),
\end{equation}
hence summing over $i$, we obtain the equality in $\CH^r(Y)$:
$$\sum_{i \in I} m_i i_{D_i \ast} (c_1(\cN_i) c_{r-2}(T_{D_i})) = - \sum_{(i,j) \in I^2, i \prec j} i_{D_{ij} \ast} (m_j c_{r-2}(T_{D_i | D_{ij}}) + m_i c_{r-2}(T_{D_j | D_{ij}})).$$
For every pair $(i,j)$ in $I^2$ such that $i \prec j$, we have the following normal exact sequences of vector bundles on $D_{ij}$:
$$
0 \lra T_{D_{ij}} \lra T_{D_i | D_{ij}} \lra \cN_{D_{ij}} D_i \lra 0,
$$
$$
0 \lra T_{D_{ij}} \lra T_{D_j | D_{ij}} \lra \cN_{D_{ij}} D_j \lra 0,
$$
hence the following equalities in $\CH^{r-2}(D_{ij})$:
$$
c_{r-2}(T_{D_i | D_{ij}}) = c_{r-2}(T_{D_{ij}}) + c_1(\cN_{D_{ij}} D_i) c_{r-3}(T_{D_{ij}}),$$
$$
c_{r-2}(T_{D_j | D_{ij}}) = c_{r-2}(T_{D_{ij}}) + c_1(\cN_{D_{ij}} D_j) c_{r-3}(T_{D_{ij}}),$$
hence using the canonical isomorphisms of line bundles on $D_{ij}$ between $\cN_{D_{ij}} D_j$ and $\cN_{i | D_{ij}}$, as well as between $\cN_{D_{ij}} D_i$ and $\cN_{j | D_{ij}}$:
$$
c_{r-2}(T_{D_i | D_{ij}}) = c_{r-2}(T_{D_{ij}}) + c_1(\cN_{j | D_{ij}}) c_{r-3}(T_{D_{ij}}),
$$
$$c_{r-2}(T_{D_j | D_{ij}}) = c_{r-2}(T_{D_{ij}}) + c_1(\cN_{i | D_{ij}}) c_{r-3}(T_{D_{ij}}),
$$
hence replacing in \eqref{sum mult Ni}, we obtain the equality in $\CH^r(Y)$:
\begin{multline*} \sum_{i \in I} m_i i_{D_i \ast} (c_1(\cN_i) c_{r-2}(T_{D_i}))
     = - \sum_{(i,j) \in I^2, i \prec j} (m_j + m_i) i_{D_{ij} \ast} c_{r-2}(T_{D_{ij}}) \\
     - \sum_{(i,j) \in I^2, i \prec j} i_{D_{ij} \ast} [\{m_i c_1(\cN_{i | D_{ij}}) + m_j c_1(\cN_{j | D_{ij}})\} c_{r-3}(T_{D_{ij}}]; 
    \end{multline*}
using Corollary \ref{N_i _D_ij and N_j _D_ij in terms of N_ij}, and using that $D^3$ is empty, we get:
$$  \sum_{i \in I} m_i i_{D_i \ast} (c_1(\cN_i) c_{r-2}(T_{D_i}))= - \sum_{(i,j) \in I^2, i \prec j} (m_i + m_j) i_{D_{ij} \ast} c_{r-2}(T_{D_{ij}}).$$ 

Hence replacing in \eqref{first eq c_r omega^1 D^3 empty}, we obtain the equality in $\CH^r(Y)_\Q$:
\begin{align*}
c_r(\omega^{1 \vee}_{Y/C}) &=  c_r([T_g]) + \sum_{i \in I} (m_i-1) i_{D_i \ast} c_{r-1}(T_{ D_i}) \\
&-  \sum_{(i,j) \in I^2, i \prec j} (m_i + m_j) i_{D_{ij} \ast} c_{r-2}(T_{D_{ij}}) + \sum_{(i,j) \in I^2, i \prec j} i_{D_{ij} \ast} c_{r-2}(T_{D_{ij}}), \\
&= c_r([T_g]) + \sum_{i \in I} (m_i-1) i_{D_i \ast} c_{r-1}(T_{ D_i}) -  \sum_{(i,j) \in I^2, i \prec j} (m_i + m_j-1) i_{D_{ij} \ast} c_{r-2}(T_{D_{ij}}).
\end{align*}
This proves equality \eqref{eq c_r omega and T_g when D^3 empty}.

\chapter[Alternating product of Griffiths line bundles and SNC degenerations]{The alternating product of Griffiths line bundles associated to a pencil of  complete varieties with SNC degenerations}
\label{Alternating product of Griffiths line bundles and the Grothendieck-Riemann-Roch formula}
In this chapter, we work over a base field $k$ of characteristic zero, and by the word ``scheme", we mean a separated algebraic scheme of finite type over $k$.

\section{The characteristic classes $\rho$ and $\rho_r$ and the Gro\-then\-dieck-Riemann-Roch formula}

In this section, we will work on a fixed smooth $k$-scheme $X$.

\subsection{The classes $\lambda_y$, $\phi_y$, and $\rho$}

Recall that the \emph{Grothendieck $\lambda$-map}:
$$\lambda_y : K^0(X) \lra 1 + y \, K^0(X)[[y]], $$
where $y$ denotes an indeterminate, 
is defined as the only map that is multiplicative for exact sequences, and that is given on (the $K$-theory class of)  a vector bundle $V$ by:
$$\lambda_y([V]) := \sum \limits_{i = 0}^{\rk(V)} y^i \left[ \Exterior^i V \right] \in 1 + y \, K^0(X)[y];$$
see for instance \cite[app. to Expos\'e 0]{SGA6}.

\begin{definition}  \label{phi_y} 
For every vector  bundle $V$ on $X$, we define  $\phi_y(V)$ as the element of the polynomial ring $\CH^\ast(X)_\Q [y]$ given by:
\begin{equation}
\phi_y(V) := \ch\big(\lambda_y(V^{\vee})\big) \,  \, \Td(V)
= \sum \limits_{i = 0}^{\rk(V)} y^i \ch \left(\Exterior ^i V^{\vee} \right) \, \Td(V).
\end{equation}
\end{definition}

The ring $\CH^\ast(X)_\Q [y]$ has a new grading constructed in terms of the existing natural bigrading: 
$$\left(\CH^\ast(X)_\Q [y] \right) ^{(d)} := \overset{d}{\underset{e = 0}{\bigoplus}} \CH^e(X)_\Q \otimes _\Q \Q[y]_{d-e}.$$
For every vector bundle $V$ over $X$, the term of total degree $0$ of the element $\phi_y(V)$ is $1$, so this element is in $ 1 + (\CH^\ast(X)_\Q [y] )^{(\geq 1)}.$

\begin{proposition} \label{phi_y is multiplicative}
The map $\phi_y$ is multiplicative for short exact sequences.
\end{proposition}

\begin{proof} The $\lambda$-map $\lambda_y$ and the Todd class $\Td$ are both multiplicative for short exact sequences, and the Chern character $\ch$ is a morphism of rings. 
\end{proof}

\begin{remark}
 We could use this to descend $\phi_y$ to a morphism of groups from $K^0(X)$ to $1 + \left ( \CH^\ast(X)_\Q [[y]] \right )^{(\geq 1)}$, but we would then lose the property that $\phi_y(V)$ is a polynomial, which is necessary for our purpose.
\end{remark}

\begin{proposition} \label{phi_y of a line bundle}
Let $L$ be a line bundle on $X$, we have the identity in~$\CH^\ast(X)_\Q [y]$:
\begin{equation}\label{formula phi_y line bundle}
\phi_y(L) = \td(c_1(L))  \, \left[1 + y \,  e^{-c_1(L)}\right].
\end{equation}
In particular, replacing $y$ with $-1$:
\begin{equation} \phi_{-1}(L) = c_1(L). \label{formula phi_y line bundle 2} \end{equation}
\end{proposition}

Recall that, in the right-hand side of \eqref{formula phi_y line bundle}, we denote by 
$\td(x)$  the formal series $x/(1 - e^{-x})$.

\begin{proof} The first equality \eqref{formula phi_y line bundle} comes from a direct computation of $\phi _y(L)$, using the equality in $\CH^\ast(X)_\Q$:
$$\ch(L ^\vee) = e^{-c_1(L)}.$$
The second equality \eqref{formula phi_y line bundle 2} follows from the  definition of the series $\td$.
\end{proof}

The identity \eqref{formula phi_y line bundle 2} is also a particular case of \cite[Ex. 3.2.5]{Fulton98}. Indeed, the upcoming Proposition \ref{first terms of rho(V)} can be seen as a variant of this example.

\begin{definition}
\label{rho and rho_r}
If $V$ is a vector bundle on $X$, the class $\rho(V)$ is defined by:
$$\rho(V) := \dot{\phi}_{-1}(V) = \Big(\frac{\mathrm{d}}{\mathrm{d} y} \ch(\lambda_y(V^\vee))\Big)_{| y = -1} \Td(V) \in \CH^\ast(X)_\Q.$$
the derivative of $\phi$ applied to $y = -1$, it is well-defined since $\phi_y(V)$ is a polynomial. 

If $r \geq 1$ is an integer, we also define the characteristic class:
$$\rho_r := c_{r-1} - \frac{r}{2} c_r + \frac{1}{12} c_1 c_r : K^0(X) \lra \CH^{\ast}(X)_\Q.$$
\end{definition}

We are interested in comparing the classes $\rho(V)$ and $\rho_r([V])$ when $V$ is a vector bundle of rank~$r$.

\begin{proposition} \label{rho of a line bundle}
Let $L$ be a line bundle on $X$, we have the identity in $\CH^\ast(X)_\Q $:
\begin{equation}\label{formula rho line bundle}
\rho(L) = \td(-c_1(L)) =: \Td (L^\vee).
\end{equation}
\end{proposition}
\begin{proof}
We differentiate the equality \eqref{formula phi_y line bundle}. Since $\phi_y(L)$ is a polynomial of degree at most 1, its derivative is its coefficient of degree $1$, which gives the equality:
$$\rho(L) = \td(c_1(L)) \,  e^{-c_1(L)}.$$
Now, we have an equality of formal series in an indeterminate $x$:
$$\td(x) \,  e^{-x} = \td(-x),$$ 
that we can show using the definition of $\td$ in terms of exponentials, as follows:
\begin{align*}
\td(x) \, e^{-x} & = \bigg(\frac{x}{1 - e^{-x}}\bigg)\, e^{-x}, \\
& = \frac{x}{e^x - 1}, \\
& = - \frac{x}{1 - e^{- (-x)} }, \\
& = \td(-x).
\end{align*}
Replacing $x$ with $c_1(L)$, this becomes the result. 
\end{proof}

\begin{corollary}
\label{Leibniz rule}
Let $V$ be a vector bundle of rank $r$ on $X$, and let
$$V = V_r \supset ... \supset V_0 = 0$$
be a filtration such that $L_i := V_i/V_{i-1}$ is a line bundle for every integer $i\in \{1, \dots, r\}.$

We have the equality in $\CH^{\ast}(X)_\Q$:
$$\rho(V) = \sum \limits_{i=1}^r \td(-c_1(L_i)) \,  \prod \limits_{\substack{j = 1 \\ j \neq i}}^r c_1(L_j).$$
\end{corollary}
\begin{proof}
Using Proposition \ref{phi_y is multiplicative}, the map $\phi_y$ is multiplicative for exact sequences, so we have the equality:
$$\phi_y(V) = \prod \limits_{i=1}^r \phi_y(L_i).$$
Differentiating and applying the Leibniz rule, we obtain:
$$\dot{\phi}_{-1}(V) = \sum \limits_{i=1}^r \dot{\phi}_{-1}(L_i)  \, \prod \limits_{\substack{j = 1 \\ j \neq i}}^r \phi_{-1}(L_j).$$
Using \eqref{formula phi_y line bundle 2} and \eqref{formula rho line bundle}, we obtain the result.
\end{proof}

This allows us to compare $\rho$ and $\rho_r$:

\begin{proposition}[compare {\cite[p. 374]{BCOV94}}]  \label{first terms of rho(V)} Let $V$ be a vector bundle of rank $r$ on $X$, we have the identity in $\CH^\ast(X)_\Q$:
\begin{equation}\label{formula rho first terms}
\rho(V) = \rho_r([V]) + (\textrm{terms of codimension} \geq r+2).
\end{equation}
\end{proposition}

\begin{proof}
First, suppose that $V$ has a filtration by subbundles $V = V_r \supset ... \supset V_0 = 0$, such that for $1 \leq i \leq r$, $L_i := V_i/V_{i-1}$ is a line bundle.
Using Corollary \ref{Leibniz rule}, we have the equality:
$$\rho(V) = \sum \limits_{i=1}^r \td(-c_1(L_i)) \,  \prod \limits_{\substack{j = 1 \\ j \neq i}}^r c_1(L_j).$$

Recall the first terms of the development of $\td$:
$$\td(x) = 1 + \frac{x}{2} + \frac{x^2}{12} + ... \quad .$$
Accordingly we have:
\begin{align*}
\rho(V) & = \sum \limits_{i=1}^r \Bigg[ \Big(1 - \frac{c_1(L_i)}{2} + \frac{c_1(L_i)^2}{12} + (\textrm{terms of codimension} \geq 3)\Big)  \prod \limits_{\substack{j = 1 \\ j \neq i}}^r c_1(L_j) \Bigg], \\
& = \sum \limits_{i=1}^r \prod \limits_{\substack{j = 1 \\ j \neq i}}^r c_1(L_j) 
- \frac{1}{2} \sum \limits_{i=1}^r \prod \limits_{j=1}^r c_1(L_j) 
+ \frac{1}{12} \sum \limits_{i=1}^r \Big(c_1(L_i) \,  \prod \limits_{j=1}^r c_1(L_j) \Big)
+ (\textrm{codim} \geq r+2)), \\
& = \sum \limits_{i=1}^r \prod \limits_{\substack{j = 1 \\ j \neq i}}^r c_1(L_j) 
- \frac{r}{2} \prod \limits_{j=1}^r c_1(L_j) 
 + \frac{1}{12} \Big(\sum \limits_{i=1}^r c_1(L_i)   \Big)  \prod \limits_{j=1}^r c_1(L_j)
 + (\textrm{codim} \geq r+2),
\end{align*}
where by $(\textrm{codim} \geq r+2)$ we mean an element in $\CH^{\geq r+2}(X)_\Q$.

By multiplicativity of the total Chern class, the elementary symmetric polynomials of $(c_1(L_i))_i$ are simply the Chern classes of $V$, which gives the equality:
\begin{align*}
    \rho(V) &= c_{r-1}(V) - \frac{r}{2} c_r(V) + \frac{1}{12} c_1(V) c_r(V) +(\textrm{codim} \geq r+2), \\
    &= \rho_r([V]) + (\textrm{codim} \geq r+2)).
\end{align*}
This shows the desired formula when $V$ has such a filtration.

In the general case, we are reduced to this situation by the ``splitting principle" (\cite[Remark 3.2.3]{Fulton98}). 
\end{proof}

\subsection{The class $\rho$ and the Grothendieck-Riemann-Roch formula}

The Riemann-Roch formula implies the following result.

\begin{proposition}
\label{RRG and rho}
Let $X, S$ be two smooth $k$-schemes of pure dimension, let:
$$
f : X \lra S 
$$
be a proper morphism of relative dimension $m \geq 0$, and let $V$ be a vector bundle on $X$.
Let us denote by 
$$
[T_f]:= [T_{X/k}] - f^\ast [T_{S/k}] \in K^0(X)
$$
the relative tangent class associated with the l.c.i. morphism $f$.\footnote{see for instance  \cite[Expos\'e VIII]{SGA6} or \cite[B.7.6]{Fulton98}.}

We have the equality in $\CH^\ast(S)_\Q$:
$$
f_\ast \left (\rho(V)  \frac{\Td([T_f])}{\Td(V)} \right )
= \sum \limits_{ \substack{0 \leq p \leq \rk(V),\\ 0 \leq q \leq m}} p (-1)^{p+q-1} \ch(R^q f_\ast \Exterior^p V^\vee).
$$
In particular, taking the terms of codimension $1$, we have the equality in $\CH^1(S)_\Q$:
$$
f_\ast \left [ \left (\rho(V)  \frac{\Td([T_f])}{\Td(V)} \right) ^{(m + 1)} \right ]
= \sum \limits_{\substack{0 \leq p \leq \rk(V),\\ 0 \leq q \leq m}} p (-1)^{p+q-1} c_1(R^q f_\ast \Exterior^p V^\vee).
$$
\end{proposition}

\begin{proof}
We have the equality:
\begin{align}
f_\ast \left (\rho(V)  \frac{\Td([T_f])}{\Td(V)} \right)
&= f_\ast \left (\Big(\frac{\mathrm{d}}{\mathrm{d}y} \ch(\lambda_y(V^\vee))\Big) _{| y = -1} \Td(V)  \frac{\Td[T_f]}{\Td(V)} \right ), \label{first eq push rho} \\
&= \frac{\mathrm{d}}{\mathrm{d}y} f_\ast \left (\ch(\lambda_y(V^\vee))  \Td([T_f]) \right) _{| y = -1}, \nonumber \\
&= \frac{\mathrm{d}}{\mathrm{d}y} f_\ast \left (\sum_{0 \leq p \leq \rk(V)} y^p \ch(\Exterior^p V^\vee)  \Td([T_f]) \right )_{| y = -1}, \label{second eq push rho} \\
&= \frac{\mathrm{d}}{\mathrm{d}y} \left (\sum_{0 \leq p \leq \rk(V)} y^p \ch(R^{\bullet} f_\ast \Exterior^p V^\vee) \right )_{| y = -1}, \label{third eq push rho} \\
&= \sum_{0 \leq p \leq \rk(V)} p (-1)^{p-1} \ch(R^{\bullet} f_\ast \Exterior^p V^\vee), \nonumber \\
&= \sum \limits_{\substack{0 \leq p \leq \rk(V) \\ 0 \leq q \leq m}} p (-1)^{p-1} (-1)^q \ch(R^q f_\ast \Exterior^p V^\vee), \label{fourth eq push rho} 
\end{align}
where in \eqref{first eq push rho}, we have used the definition of the class $\rho$; in \eqref{second eq push rho}, we have used the definition of the class $\lambda_y$; in \eqref{third eq push rho}, we have used the Grothendieck-Riemann-Roch formula; and in \eqref{fourth eq push rho}, we have used the definition of the morphism of abelian groups $R^{\bullet} f_\ast$.
\end{proof}

\section{Application to Griffiths line bundles of fibrations over curves with SNC fibers}

\subsection{An application of Steenbrink's theory}

We are now able to compute a certain combination of the Griffiths line bundles in terms of the function $\rho_r$.

\begin{theorem}
\label{GK and rho_r}
Let $C$ be a connected smooth projective complex curve with generic point $\eta$, $Y$ be a connected smooth projective $N$-dimensional complex scheme, and let
$$
g : Y \lra C
$$
be a surjective morphism of complex schemes. Let us assume that there exists a finite subset $\Delta$ in $C$ such that $g$ is smooth over $C - \Delta$, and such that the divisor $Y_\Delta$ is a divisor with strict normal crossings in $Y$.

We have the equality in $\CH^1(C)_\Q$:
\begin{equation*}
\sum_{n=0}^{2(N-1)} (-1)^{n-1} c_1(\GK_{C,-}(\H^n(Y_\eta / C_\eta)))
= g_\ast \left [ \left (\rho_{N-1}(\omega^{1 \vee}_{Y/C})  \frac{\Td([T_g])}{\Td(\omega^{1 \vee}_{Y/C})} \right )^{(N)} \right ].
\end{equation*}
\end{theorem}

\begin{proof}
Using Proposition \ref{lower Griffiths and logarithmic cohomology}, for every integer $n$, we have an isomorphism of line bundles over $C$:
$$\GK_{C,-}(\H^n(Y_\eta / C_\eta)) \simeq \bigotimes_{0 \leq p \leq n} \big(\det R^{n-p} g_\ast \omega^p_{Y/C}\big)^{\otimes p}.$$
Consequently, taking Chern classes and taking the alternate sum, we obtain the equality in $\CH^1(C)_\Q$:
$$\sum_{n=0}^{2(N-1)} (-1)^{n-1} c_1(\GK_{C,-}(\H^n(Y_\eta / C_\eta))) = \sum_{n=0}^{2(N-1)} (-1)^{n-1} \sum_{0 \leq p \leq n} p c_1(R^{n-p} g_\ast \omega^p_{Y/C}).$$
We reindex this sum on $p$ and $q := n-p$, and we get: 
$$\sum_{n=0}^{2(N-1)} (-1)^{n-1} c_1(\GK_{C,-}(\H^n(Y_\eta / C_\eta))) = \sum \limits_{\substack{p,q \geq 0,\\ p+q \leq 2(N-1)}} p (-1)^{p+q-1} c_1(R^q g_\ast \omega^p_{Y/C}).$$
Since the morphism $g$ is of relative dimension $N-1$, the term in the sum vanishes unless $p,q \leq N-1$, so we can rewrite:
\begin{align*}
\sum_{n=0}^{2(N-1)} (-1)^{n-1} c_1(\GK_{C,-}(\H^n(Y_\eta / C_\eta))) & = \sum_{0 \leq p,q \leq N-1} p (-1)^{p+q-1} c_1(R^q g_\ast \omega^p_{Y/C}), \\
&= \sum_{0 \leq p,q \leq N-1} p (-1)^{p+q-1} c_1(R^q g_\ast \Exterior^p (\omega^{1 \vee}_{Y/C})^\vee ).
\end{align*}
Now, we apply Proposition \ref{RRG and rho} to the morphism $g$ of relative dimension $N-1$ and the locally free sheaf
$$
V := \omega^{1 \vee}_{Y/C}
$$
of rank $N-1$. This yields the equality:
$$ \sum_{n=0}^{2(N-1)} (-1)^{n-1} c_1(\GK_{C,-}(\H^n(Y_\eta / C_\eta))) = g_\ast \left [ \left (\rho(\omega^{1 \vee}_{Y/C})  \frac{\Td([T_g])}{\Td(\omega^{1 \vee}_{Y/C})} \right )^{(N)} \right ],$$
hence applying Proposition \ref{first terms of rho(V)} to the vector bundle $V$ of rank $N-1$:
\begin{align*}
 & \sum_{n=0}^{2(N-1)} (-1)^{n-1} c_1(\GK_{C,-}(\H^n(Y_\eta / C_\eta))) \\ 
&= g_\ast \left [ \left (\{\rho_{N-1}(\omega^{1 \vee}_{Y/C}) + (\textrm{terms of codimension} \geq N+1) \}  \frac{\Td([T_g])}{\Td(\omega^{1 \vee}_{Y/C})} \right )^{(N)} \right ],\\
&= g_\ast \left [ \left (\rho_{N-1}(\omega^{1 \vee}_{Y/C})  \frac{\Td([T_g])}{\Td(\omega^{1 \vee}_{Y/C})} \right )^{(N)} \right ]. \qedhere
\end{align*}
\end{proof}

\subsection{The alternating product  of Griffiths line bundles and the geometry of singular fibers}

We adopt the notation of the previous subsection, and assume, without loss of generality, that the divisor $\Delta$ in $C$ is reduced.

Let us adopt the notation of Subsection \ref{Notation Chern classes omega^1 strict DNC}. Namely, we write the divisor with strict normal crossings $Y_\Delta$ as:
$$
D = \sum_{i \in I} m_i D_i,
$$
where $I$ is a finite set, for every $i$ in $I$, $m_i \geq  1$ is an integer, and $D_i$ is a smooth connected divisor, such that the $(D_i)_{i \in I}$ intersect each other transversally. We denote by $\cN_i$ the normal bundle $N_{D_i} Y$ of $D_i$ in $Y$.

The set $I$ may be written as the disjoint union:
$$I = \bigcup_{x \in \Delta} I_x,$$
where, for every $x \in \Delta,$ $I_x$ denotes the non-empty subset of $I$ defined by:
$$I_x := \big\{ i \in I \mid g(D_i) = \{x \} \big\}.$$

For every subset $J$ of $I$, let us denote:
$$
D_J := \bigcap_{i \in J} D_i,
$$
it is a smooth subscheme of codimension $|J|$.

We adopt the notation of \cite[II, 3.4]{Deligne70}. Namely, we denote, for every integer $r \geq 1$, $D^r$ the subscheme of codimension $r$ of $Y$ defined by the union of all the intersections of $r$ different components $(D_i)$:
$$
D^r := \bigcup_{J \subset I, |J| = r} D_J.
$$
Let us choose a total order $\preceq$ on $I$.

For every element $i$ in $I$, let us define an open subset $\Dcirc_i$ of the divisor $D_i$ by:
$$\Dcirc_i := D_i - D_i \cap D^2.$$
Similarly, for every pair $(i,j)$ in $I^2$ such that $i \prec j$, let us define an open subset $\Dcirc_{ij}$ of the subscheme
$$D_{ij} := D_{\{i,j\}}$$
by:
$$\Dcirc_{ij} := D_{ij} - D_{ij} \cap D^3.$$
Let $\chi_{\mathrm{top}}$ denote the topological Euler characteristic.

First recall the following classical theorem:

\begin{theorem} \label{chi_top complement of DNC 2}
Let $X$ be a proper smooth $n$-dimensional complex scheme and $E$ be a reduced divisor with strict normal crossings in $X$. We have the equality of integers:
\begin{equation*}
    \chi_{\mathrm{top}}(X - E) := \sum_{i = 0}^{n} (-1)^i \dim_\C H^i(X-E, \C) = \int_X c_n(\Omega^{1 \vee}_X (\log E)).
\end{equation*}
\end{theorem}
\begin{proof}
We can give a proof of this theorem using the results of this paper. Let us adopt the notation of Proposition \ref{comparison Omega log and Omega easy cases general}, we have the equality of integers:
\begin{align*}
\int_X c_n(\Omega^{1 \vee}_X (\log E)) &= (-1)^n \int_X c_n(\Omega^1_X (\log E)),\\
&= (-1)^n \sum_{J \subset I} \int_{E_J} c_{n-|J|}(\Omega^1_{E_J}) \textrm{ using equality \eqref{Chern class Omega1log} from Proposition \ref{comparison Omega log and Omega easy cases general},}\\
&= \sum_{J \subset I} (-1)^{|J|} \int_{E_J} c_{n-|J|}(T_{E_J}),\\
&= \sum_{J \subset I} (-1)^{|J|} \chi_{\mathrm{top}}(E_J),\\
&= \chi_{\mathrm{top}}(X) - \chi_{\mathrm{top}}(E) \textrm{ using the inclusion-exclusion formula,}\\
&= \chi_{\mathrm{top}}(X - E). \qedhere
\end{align*}
\end{proof}

\begin{theorem}
\label{GK in terms of classes in D non-reduced}
With the notation of this subsection, we have the equality in $\CH_0(C)_\Q$:
\begin{equation}
    \sum_{n=0}^{2(N-1)} (-1)^{n-1} c_1(\GK_{C,-}(\H^n(Y_\eta/C_\eta))) = \frac{1}{12} g_\ast (c_1 c_{N-1})(\omega^{1 \vee}_{Y/C}) + \sum_{x \in \Delta} \alpha_x [x],
\end{equation}
where for every $x$ in $\Delta$, $\alpha_x$ is the rational number given by:
\begin{align}
&\alpha_x = \frac{N-1}{4} \sum_{i \in I_x} (m_i - 1) \chi_{\mathrm{top}}(\Dcirc_i) + \frac{1}{12} \sum_{\substack{(i,j) \in I_x^2, \\ i \prec j}} (3 - m_i/m_j - m_j/m_i) \chi_{\mathrm{top}}(\Dcirc_{ij}), \label{geommult}\\
&= \frac{1}{12} \sum_{i \in I_x} \Big[3 (N-1) (m_i - 1) \chi_{\mathrm{top}}(\Dcirc_i) + \int_{D_i} c_1(\cN_i) c_{N-2}(\Omega^{1 \vee} _{D_i} (\log D_i \cap D^2)) \Big] + \frac{1}{4} \sum_{\substack{(i,j) \in I_x^2, \\ i \prec j}} \chi_{\mathrm{top}}(\Dcirc_{ij}). \label{Z/12}
\end{align}
\end{theorem}

The expression \eqref{geommult} for $\alpha_x$ involves only the topology of the open strata $\Dcirc_i$ and $\Dcirc_{ij}$ of the singular fibers of $g$ and the multiplicities $m_i$ of their components $D_i$. The expression \eqref{Z/12} makes clear that $\alpha_x$ belongs to $(1/12)\Z$.

\begin{proof}
By Theorem \ref{GK and rho_r} and by  definition of the class $\rho_{N-1}$, the following equalities hold in $\CH_0(C)_\Q$:
\begin{align}
\sum_{n=0}^{2(N-1)} &(-1)^{n-1} c_1(\GK_{C,-}(\H^n(Y_\eta / C_\eta))) \nonumber\\
&= g_\ast \Big [\Big (\rho_{N-1}(\omega^{1 \vee}_{Y/C}) . \frac{\Td([T_g])}{\Td(\omega^{1 \vee}_{Y/C})} \Big )^{(N)} \Big ], \nonumber\\
&= g_\ast \Big  [ \Big ( (c_{N-2}(\omega^{1 \vee}_{Y/C}) - \frac{N-1}{2} c_{N-1}(\omega^{1 \vee}_{Y/C}) + \frac{1}{12} (c_1 c_{N-1})(\omega^{1 \vee}_{Y/C}) )  \frac{\Td([T_g])}{\Td(\omega^{1 \vee}_{Y/C})} \Big )^{(N)} \Big ], \nonumber\\
&= \frac{1}{12} g_\ast (c_1 c_{N-1})(\omega^{1 \vee}_{Y/C}) - \frac{N-1}{2} g_\ast \Big [c_{N-1}(\omega^{1 \vee}_{Y/C})  \, \Big (\frac{\Td([T_g])}{\Td(\omega^{1 \vee}_{Y/C})} \Big )^{(1)} \Big ]  \nonumber\\
&+ g_\ast \Big [c_{N-2}(\omega^{1 \vee}_{Y/C}) \Big (\frac{\Td([T_g])}{\Td(\omega^{1 \vee}_{Y/C})} \Big  )^{(2)} \Big ]. \label{first equality computation rho omega}
\end{align}
Using successively equality \eqref{Td/Td degree 1 general} from Proposition \ref{first terms Td/Td general} and  Proposition \ref{omega^1 restricted to a strata},
we get the following equality in $\CH^N(Y)_\Q$:
\begin{align*}
c_{N-1}(\omega^{1 \vee}_{Y/C}) \, \Big (\frac{\Td([T_g])}{\Td(\omega^{1 \vee}_{Y/C})} \Big ) ^{(1)} 
&= -\frac{1}{2} \sum_{i \in I} (m_i - 1) i_{D_i \ast} c_{N-1}(\omega^{1 \vee}_{Y/C | D_i} ),\\
&= -\frac{1}{2} \sum_{i \in I} (m_i - 1) i_{D_i \ast} c_{N-1}(\Omega^{1 \vee}_{D_i} (\log D_i \cap D^2))
\end{align*}
hence pushing forward by $g$, the equality in $\CH_0(C)_\Q$:
\begin{align*}
g_\ast \Big [ c_{N-1}(\omega^{1 \vee}_{Y/C}) \, \Big (\frac{\Td([T_g])}{\Td(\omega^{1 \vee}_{Y/C})} \Big )^{(1)} \Big ]
&= -\frac{1}{2} \sum_{i \in I} (m_i - 1) g_\ast i_{D_i \ast} c_{N-1}(\Omega^{1 \vee}_{D_i} (\log D_i \cap D^2)),\\
&= \sum_{x \in \Delta} \alpha_{1,x} [x],
\end{align*}
where for every point $x$ in $\Delta$, $\alpha_{1,x}$ is the rational number given by:
\begin{align*}
\alpha_{1,x} &= -\frac{1}{2} \sum_{i \in I_x} (m_i - 1) \int_{D_i} c_{N-1}(\Omega^{1 \vee}_{D_i} (\log D_i \cap D^2)),\\
&= -\frac{1}{2} \sum_{i \in I_x} (m_i - 1)\,  \chi_{\mathrm{top}}(\Dcirc_i) \textrm{ \quad \quad using Theorem \ref{chi_top complement of DNC 2}.}
\end{align*}
Similarly, using equality \eqref{eq Td/Td degree 2 with multiplicities} from Proposition \ref{first terms Td/Td general} and Proposition \ref{omega^1 restricted to a strata},  the following  equalities hold in $\CH^N(Y)_\Q$:
\begin{align*}
c_{N-2}(\omega^{1 \vee}_{Y/C}) \Big (\frac{\Td([T_g])}{\Td(\omega^{1 \vee}_{Y/C})} \Big  )^{(2)}
&= \frac{1}{12} \sum_{\substack{(i,j) \in I^2, \\ i \prec j}} (3 - m_i/m_j - m_j/m_i) i_{D_{ij} \ast} c_{N-2}(\omega^{1 \vee}_{Y/C | D_{ij}} ),\\
&= \frac{1}{12} \sum_{\substack{(i,j) \in I^2, \\ i \prec j}} (3 - m_i/m_j - m_j/m_i) i_{D_{ij} \ast} c_{N-2}(\Omega^{1 \vee}_{D_{ij}} (\log D_{ij} \cap D^3) ),
\end{align*}
hence pushing forward by $g$, the equality in $\CH_0(C)_\Q$:
$$
g_\ast \Big [c_{N-2}(\omega^{1 \vee}_{Y/C})  \Big (\frac{\Td([T_g])}{\Td(\omega^{1 \vee}_{Y/C})} \Big )^{(2)} \Big ]
= \sum_{x \in \Delta} \alpha_{2,x} [x],
$$
where, for every point $x$ in $\Delta$, $\alpha_{2,x}$ is the rational number given by:
\begin{align*}
\alpha_{2,x} &= \frac{1}{12} \sum_{\substack{(i,j) \in I_x^2, \\ i \prec j}} (3 - m_i/m_j - m_j/m_i) \int_{D_{ij}} c_{N-2}(\Omega^{1 \vee}_{D_{ij}} (\log D_{ij} \cap D^3) ),\\
&= \frac{1}{12} \sum_{\substack{(i,j) \in I_x^2, \\ i \prec j}} (3 - m_i/m_j - m_j/m_i) \chi_{\mathrm{top}}(\Dcirc_{ij} ) \textrm{ using again Theorem \ref{chi_top complement of DNC 2}.}
\end{align*}
Hence replacing in \eqref{first equality computation rho omega}, we obtain the equality in $\CH_0(C)_\Q$:
$$
\sum_{n=0}^{2(N-1)} (-1)^{n-1} c_1(\GK_{C,-}(\H^n(Y_\eta / C_\eta)))
= \frac{1}{12} g_\ast (c_1 c_{N-1})(\omega^{1 \vee}_{Y/C}) + \sum_{x \in \Delta} \alpha_x [x],
$$
where for every point $x$ in $\Delta$, the rational number $\alpha_x$ is given by:
\begin{align*}
&\alpha_x = - \frac{N-1}{2} \alpha_{1,x} + \alpha_{2,x},\\
&= \frac{N-1}{4} \sum_{i \in I_x} (m_i - 1) \chi_{\mathrm{top}}(\Dcirc_i) + \frac{1}{12} \sum_{\substack{(i,j) \in I_x^2, \\ i \prec j}} (3 - m_i/m_j - m_j/m_i) \chi_{\mathrm{top}}(\Dcirc_{ij}).
\end{align*}
For the other expression of the rational number $\alpha_x$, one reasons similarly using the expression from \eqref{eq Td/Td degree 2 with c_1(N_i)} from Proposition \ref{first terms Td/Td general} instead of \eqref{eq Td/Td degree 2 with multiplicities}.
\end{proof}

\chapter[Griffiths line bundles and non-degenerate critical points]{The alternating product of Griffiths line bundles associated to a proper fibration with  non-degenerate critical points}

The objective of this chapter is to prove the following theorem.

\begin{theorem}
\label{GK critical points}
Let $C$ be a connected smooth projective complex curve with generic point $\eta$, $H$ be a smooth projective $N$-dimensional complex scheme, and let
$$
f : H \lra C
$$
be a morphism of complex schemes. Let us assume that there exists a finite subset $\Sigma$ in $H$ such that $f$ is smooth on $H - \Sigma$ and admits a non-degenerate critical point at any point of $\Sigma$.

Let $[\Sigma]$ be the reduced $0$-cycle in $H$ associated with the finite subset $\Sigma$. 

We have the equalities in $\CH_0(C)_\Q$:
\begin{equation}
\label{pas intro GK- points doubles}
\sum_{n = 0}^{2(N-1)} (-1)^{n-1} c_1\big(\GK_{C, -}(\H^n(H_\eta / C_\eta))\big)
= \frac{1}{12} f_\ast \big((c_1 c_{N-1})([\Omega^1_{H/C}]^\vee)\big) + u^-_N f_\ast [\Sigma],
\end{equation}
and
\begin{equation}
\sum_{n = 0}^{2(N-1)} (-1)^{n-1} c_1\big(\GK_{C, +}(\H^n(H_\eta / C_\eta))\big)
= \frac{1}{12} f_\ast \big((c_1 c_{N-1}\big)([\Omega^1_{H/C}]^\vee)\big) + u_N^+ f_\ast [\Sigma], \label{pas intro GK+ points doubles}
\end{equation}
where $u^-_N$ and $u^+_N$ are the rational numbers defined by: 
\begin{equation*}
u^-_N := 
\begin{cases}
 (5N-3)/24 & \text{if $N$ is odd} \\
 N/24 & \text{if $N$ is even},
\end{cases}
\end{equation*}
and:
\begin{equation*}
u^+_N := 
\begin{cases}
 -(7N-9)/24 & \text{if $N$ is odd} \\
 N/24 & \text{if $N$ is even}.
\end{cases}
\end{equation*}
\end{theorem}

Observe that, as the morphism of non-singular complex schemes $f:H \ra C$ is smooth on an open dense subset of $H$, we have the following equality in $K^0(H) \simeq K_0(H)$:
$$
[T_f] = [\Omega^1_{H/C}] ^\vee,
$$
where $.^\vee$ denotes the duality involution on $K^0(H)$, by the same reasoning as in 
 Subsection \ref{Comparing the characteristic classes of T and omega}.

\section[Reduction of the proof to a computation in $\CH_0(H)$]{Reduction of the proof of Theorem \ref{GK critical points} to  a computation in $\CH_0(H)$}
\label{Reduction of the proof}

As in Theorem \ref{GK critical points}, let $C$ be a connected smooth projective complex curve, let $H$ be a smooth projective complex scheme of pure dimension $N \geq 1$, and let
$$
f : H \lra C
$$
be a proper surjective morphism with a finite subset $\Sigma$ of critical points, all of which are non-degenerate.

Let $\Delta := f(\Sigma)$ be the set-theoretic image of $\Sigma$, that we see as a reduced divisor in $C$. We can rephrase the hypothesis as the divisor in $H$:
$$
H_\Delta := f^\ast \Delta
$$
having only finitely many singularities, all of which are ordinary double points.

Let
$$
\nu : \tilde{H} \lra H
$$
be the blow-up of $H$ at $\Sigma$, and let $E$ be the exceptional divisor. Let
$$
g := f \circ \nu : \tilde{H} \lra C
$$
be its composition with $f$. Over the generic point $\eta$ of $C$, this morphism can be identified with the morphism $f$.

For every point $P$ in $\Sigma$, let us define a subscheme in $\tilde{H}$ by:
$$E_P := \nu^{-1}(\{P\}),$$
so that the $(E_P)_{P \in \Sigma}$ are the connected components of the exceptional divisor $E$. In particular, they are disjoint divisors, isomorphic to the complex projective space $\PP^{N-1}$; see for instance Proposition \ref{nu_* eta_p^N}  below.

Let us define a divisor in $\tilde{H}$ by:
$$
\tilde{H}_\Delta := g^\ast \Delta = \nu^\ast H_\Delta.
$$
Since the divisor $H_\Delta$ has its only singularities at points of $\Sigma$, and since all of these singularities are ordinary double points, its pullback divisor $\tilde{H}_\Delta$ can be written as follows: 
\begin{equation}
\tilde{H}_\Delta = 2 \sum_{P \in \Sigma} E_P + W \label{decomposition Htilde_Delta}
\end{equation}
where $W$ is the proper transform in $\tilde{H}$ of the divisor $H_\Delta$ in $H$; it is a non-singular divisor intersecting transversally the components $(E_P)_{P\in \Sigma}$ of the exceptional divisor, and for every point $P$, the intersection $E_P \cap W$ is a smooth quadric in the projective space $E_P$.

In particular, the divisor $\tilde{H}_\Delta$ is a divisor with strict normal crossings, which allows us to define the relative logarithmic bundle $\omega^1_{\tilde{H}/C}$.

We are going to apply Theorem \ref{GK in terms of classes in D non-reduced} to the morphism $g$ and to the variations of Hodge structures of relative cohomology:
$$\H^n({H}_\eta/C_\eta)  \lrasim \H^n(\widetilde{H}_\eta/C_\eta).$$ 
For this, we are going to prove the following lemma using the results of Chapter \ref{section comparing Omega^1 and omega^1}.

\begin{lemma} \label{nu_ast c_1 c_N-1 omega points doubles}
With the above notation, the following formula holds in $\CH^N(H)$:
\begin{equation}
\nu_\ast\big( (c_1 c_{N-1})(\omega^{1\vee}_{\widetilde{H}/C})\big) = (c_1 c_{N-1})([T_f]) - (N-2) \frac{1 - (-1)^N}{2} [\Sigma].
\end{equation}
\end{lemma}

Let us prove Theorem \ref{GK critical points} using this lemma. 

For every point $P$ in $\Sigma$, the complex scheme $E_P$ is isomorphic to the projective space $\PP^{N-1}$, hence its topological Euler characteristic is given by:
$$\chi_{\mathrm{top}}(E_P) = N.$$
Furthermore, the intersection $E_P \cap W$ is isomorphic to a smooth quadric in $\PP^{N-1}$, hence its topological Euler characteristic is given by:
$$\chi_{\mathrm{top}}(E_P \cap W) = \frac{1}{2} [(-1)^N + 2 N - 1].$$
This follows for instance from equalities  \eqref{eq c_r tangent quadric} and \eqref{eq nu_* eta_p^N} in Proposition \ref{nu_* eta_p^N} below, applied to the integers $m_P := 2$ and  $r := N-2$ and the hypersurface $Q_P:= E_P \cap W$, and from equality \eqref{non square denominator} in Proposition \ref{formal identities} below, applied to the integers $n := N$ and $r := 2$ and the complex number $a := 2$.

By additivity of the topological Euler characteristic, the characteristic of the complement $E_P - E_P \cap W$ is given by:
$$\chi_{\mathrm{top}}(E_P - E_P \cap W) = N - \frac{1}{2} [(-1)^N + 2 N - 1] = \frac{1}{2} [1 - (-1)^N].$$
Consequently,  for every point $x$ in $\Delta$, the expression \eqref{geommult} for the rational number $\alpha_x$ introduced in  Theorem \ref{GK in terms of classes in D non-reduced} reduces to:
\begin{align*}
    \alpha_x &=  \frac{N-1}{4} \sum_{P \in \Sigma_x} \chi_{\mathrm{top}}(E_P - E_P \cap W) + \frac{1}{12} \sum_{P \in \Sigma_x} \frac{1}{2} \chi_{\mathrm{top}}(E_P \cap W), \\
    &= \frac{N-1}{8} [1 - (-1)^N] |\Sigma_x| + \frac{1}{48} [(-1)^N + 2N - 1] |\Sigma_x|, \\
    &= \frac{1}{48} [6 (N-1) (1 - (-1)^N) + (-1)^N + 2N - 1] |\Sigma_x|, \\
    &= \frac{1}{48} [(6 N - 7) (1 - (-1)^N) + 2N] |\Sigma_x|.
\end{align*}
Hence replacing in Theorem \ref{GK in terms of classes in D non-reduced} and using Lemma \ref{nu_ast c_1 c_N-1 omega points doubles}, we have the equality in $\CH_0(C)_\Q$:
\begin{align*}
\sum_{n=0}^{2(N-1)} (-1)^{n-1} c_1(\GK_{C,-}(\H^n(H_\eta/C_\eta))) &= \frac{1}{12} g_\ast (c_1 c_{N-1})(\omega^{1 \vee}_{Y/C}) + \sum_{x \in \Delta} \alpha_x [x], \\
&= \frac{1}{12} f_\ast (c_1 c_{N-1})([T_f]) -  \frac{1}{12} (N-2) \frac{1 - (-1)^N}{2} f_\ast [\Sigma] \\
&+ \frac{1}{48} [(6 N - 7) (1 - (-1)^N) + 2N] \sum_{x \in \Delta}  |\Sigma_x| [x], \\
&= \frac{1}{12} f_\ast (c_1 c_{N-1})([T_f]) + u_N^- f_\ast [\Sigma], 
\end{align*}
where $u_N^-$ is the rational number given by:
\begin{align*}
u_N^- &= - \frac{1}{24} (N-2) (1 - (-1)^N) + \frac{1}{48} [(6 N - 7) (1 - (-1)^N) + 2N], \\
&= \frac{1}{48} [- 2 (N-2) (1 - (-1)^N) + (6 N - 7) (1 - (-1)^N) + 2N], \\
&= \frac{1}{48} [(4N - 3) (1 - (-1)^N) + 2 N], \\
&= \begin{cases}
 (5N-3)/24 & \text{if $N$ is odd} \\
 N/24 & \text{if $N$ is even},
\end{cases}
\end{align*}
which shows equality \eqref{pas intro GK- points doubles}.

Furthermore, for every integer $n$, equality \eqref{c1 GK + non-degenerate 1} from Corollary \ref{GK bundles with elementary exponents for non-degenerate singularites} gives the equality in $\CH_0(C)_\Q$:
$$
c_1(\GK_{C,+}(\H^n(H_\eta / C_\eta))) = c_1(\GK_{C,-}(\H^n(H_\eta / C_\eta))) + \delta^{n,N-1} \eta_N \frac{N-1}{2} f_\ast [\Sigma], 
$$
where $\eta_N$ is $1$ if $N$ is odd and $0$ if $N$ is even.

Taking the alternating sum over $n$ yields the equality in $\CH_0(C)_\Q$:
\begin{align*}
\sum_{n=0}^{2(N-1)} (-1)^{n-1} c_1(\GK_{C,+}(\H^n(H_\eta/C_\eta))) &= \sum_{n=0}^{2(N-1)} (-1)^{n-1} c_1(\GK_{C,-}(\H^n(H_\eta/C_\eta))) \\
&+ (-1)^N \eta_N \frac{N-1}{2} f_\ast [\Sigma], \\
&= \frac{1}{12} f_\ast (c_1 c_{N-1})([T_f]) + u_N^- f_\ast [\Sigma] - \eta_N \frac{N-1}{2} f_\ast [\Sigma], \\
&= \frac{1}{12} f_\ast (c_1 c_{N-1})([T_f]) + u_N^+ f_\ast [\Sigma],
\end{align*}
where $u_N^+$ is the rational number given by:
$$u_N^+ = u_N^- - \eta_N \frac{N-1}{2}
    = \begin{cases}
 -(7N-9)/24 & \text{if $N$ is odd} \\
 N/24 & \text{if $N$ is even},
\end{cases}
$$
which shows equality \eqref{pas intro GK+ points doubles} and completes the proof of Theorem \ref{GK critical points}.

The rest of this chapter shall be devoted to proving Lemma \ref{nu_ast c_1 c_N-1 omega points doubles}.

\section{Blowing up Chern classes of relative tangent classes over a curve}

We shall use the following notation. Let $k$ be a field, $H$ a smooth $k$-scheme of pure dimension $N \geq 1$ and $\Sigma$ a finite subset of $H$, and let:
$$
\nu : \tilde{H} \lra H
$$
be the blow-up of $H$ at $\Sigma$, and  $E$ its  exceptional divisor.

For every point $P$ in $\Sigma$, let us define a subscheme in $\tilde{H}$ by:
$$E_P := \nu^{-1}(\{P\}),$$
so that the $(E_P)_{P \in \Sigma}$ are the connected components of the exceptional divisor $E$. In particular, they are disjoint divisors.

For every point $P$ in $\Sigma$, let us denote the class of the divisor $E_P$ by:
$$
\eta_P := [E_P] = c_1(\cO_{\tilde{H}}(E_P)) \in \CH^1(\tilde{H}).
$$

Observe that for every point $P$ in $\Sigma$, the restriction $\nu_{\mid E_P}: E_P \ra H$ factors through the zero-dimensional scheme $P$. Therefore, for every cycle $\alpha$ in $H$ of positive codimension, the following equality holds in $\CH^\ast(\tilde{H})$:
\begin{equation}\label{pullback restricted to E_P}
\eta_P \cdot \nu^\ast \alpha = 0. 
\end{equation}

Furthermore, for every pair $(P,P')$ in $\Sigma^2$ such that $P \neq P'$, the divisors $E_P$ and $E_{P'}$ in $\tilde{H}$ are disjoint, so we obtain the equality in $\CH^2(\tilde{H})$:
\begin{equation}\label{product eta_P eta_P'}
\eta_P \cdot \eta_{P'} = 0. 
\end{equation}

\begin{proposition}\label{nu_* eta_p^N}
Let $P$ be a point in $\Sigma$. The scheme $E_P$ is isomorphic to a projective space of dimension $N-1$.

We have the equality in $\CH^N(H)$:
\begin{equation}
\nu_\ast \eta_P^N = (-1)^{N-1} [P]. \label{eq nu_* eta_p^N}
\end{equation}
Moreover, for every non-negative integer $r$, we have the equality in $\CH^r(E_P)$:
\begin{equation}
c_r(T_{E_P}) = (-1)^r \binom{N}{r} \eta_{P | E_P}^r. \label{eq c_r tangent exceptional divisor}
\end{equation}

Finally, let $m_P$ be a positive integer, and let $Q_P$ be a smooth hypersurface of degree $m_P$ in $E_P$ and let $i_{Q_P}: Q_P \ra E_P$ be the inclusion morphism.
For every non-negative integer $r$, we have the equality in $\CH^{r+1}(E_P)$:
\begin{equation}
i_{Q_P \ast} c_r(T_{Q_P}) = (-1)^{r+1} m_P \bigg [\frac{(1 + y)^N}{1 + m_P y} \bigg ]^{[r]} \eta_{P | E_P}^{r+1},\label{eq c_r tangent quadric}
\end{equation}
where $f(y)^{[r]}$ denotes the coefficient of $y^r$ in some formal series $f(y) \in \C[[y]]$.
\end{proposition}

\begin{proof}
Since $E_P$ is a connected component of the exceptional divisor $E$, there is a canonical isomorphism of $k$-schemes:
$$
\phi : E_P \lra \PP_k(T_P H),
$$
and a canonical isomorphism of line bundles on $E_P$:
$$
\phi^\ast \cO_{T_P H}(-1) \simeq \cO_{\tilde{H}}(E_P)_{| E_P}.
$$
Consequently, we have the equality in $\CH^1(E_P)$:
\begin{equation} \label{eq eta_P c_1(O(1))}
\eta_{P | E_P} = \phi^\ast c_1(\cO_{T_P H}(-1)).
\end{equation}

Now, let us compute in $\CH^N(\tilde{H})$:
\begin{align*}
\eta_P^N &= \eta_P^{N-1} [E_P],\\
&= i_{E_P \ast} \eta_{P | E_P}^{N-1} \textrm{ by the projection formula,}\\
&= (-1)^{N-1} \, i_{E_P \ast} \phi^\ast c_1(\cO_{T_P H}(1))^{N-1}.
\end{align*}
Hence pushing forward, we have the equality in $\CH^N(H)$:
\begin{align*}
\nu_\ast \eta_P^N &= (-1)^{N-1}\, \nu_\ast i_{E_P \ast} \phi^\ast c_1(\cO_{T_P H}(1))^{N-1}, \\
&= (-1)^{N-1} \,\Big [\int_{\PP(T_P H)} c_1(\cO_{T_P H}(1))^{N-1} \Big ] [P],\\
&= (-1)^{N-1}\,  [P],
\end{align*}
which shows equality \eqref{eq nu_* eta_p^N}.

Let $r$ be a non-negative integer. Since $\phi$ is an isomorphism, we have the equality in $\CH^r(E_P)$:
$$
c_r(T_{E_P}) = \phi^\ast c_r(T_{\PP(T_P H)}).
$$
Using the exact sequence of vector bundles on $\PP(T_P H)$ from \cite[B.5.8]{Fulton98}:
$$
0 \lra \cO_{\PP(T_P H)} \lra T_P H \otimes_k \cO_{T_P H}(1) \lra T_{\PP(T_P H)} \lra 0,
$$
we can replace:
$$
c_r(T_{E_P}) = \phi^\ast [(1 + c_1(\cO_{T_P H}(1) ) )^N ]^{(r)}
= \binom{N}{r} \phi^\ast c_1(\cO_{T_P H}(1))^{r}
= (-1)^r \binom{N}{r} \eta_{P | E_P}^r \textrm{ using \eqref{eq eta_P c_1(O(1))}, }
$$
which shows equality \eqref{eq c_r tangent exceptional divisor}.

Now, let $m_P$ be a positive integer and $Q_P$ be an hypersurface of degree $m_P$ in $E_P$.

We have the following exact sequence of vector bundles on $Q_P$:
$$
0 \lra T_{Q_P} \lra T_{E_P | Q_P} \lra \cN_{Q_P} E_P \lra 0,
$$
hence the following equality in $\CH^\ast(Q_P)$:
\begin{align*}
c(T_{Q_P}) &= \frac{c(T_{E_P | Q_P})}{c(\cN_{Q_P} E_P)},\\
&= \frac{\phi_{| Q_P}^\ast (1 + c_1(\cO_{T_P H}(1) ))^N}{1 + c_1(\cO_{E_P}(Q_P)_{| Q_P})} ,\\
&= \frac{\phi_{| Q_P}^\ast (1 + c_1(\cO_{T_P H}(1) ))^N}{1 + m_P \phi_{| Q_P}^\ast c_1(\cO_{T_P H}(1))} \textrm{ because $Q_P$ is of degree $m_P$,}\\
&= \phi_{|Q_P}^\ast \bigg[\frac{(1 + c_1(\cO_{T_P H}(1) ))^N}{1 + m_P c_1(\cO_{T_P H}(1) )} \bigg].
\end{align*}
Hence, for every integer $r$, we have the equality in $\CH^r(Q_P)$:
\begin{equation}
c_r(T_{Q_P}) = \bigg[\frac{(1 + y)^N}{1 + m_P y} \bigg]^{[r]} \phi_{| Q_P}^\ast c_1(\cO_{T_P H}(1))^r
= (-1)^r \bigg[\frac{(1 + y)^N}{1 + m_P y} \bigg]^{[r]} \eta_{P | Q_P}^r \textrm{ using \eqref{eq eta_P c_1(O(1))}.} \label{eq c_r T Q_p}
\end{equation}

Pushing forward \eqref{eq c_r T Q_p} by the inclusion map of $Q_P$ in $E_P$, and using the projection formula and the equality in $\CH^1(E_P)$:
$$
[Q_P] = \phi^\ast c_1(\cO_{T_P H}(m_P)) = - m_P \eta_{P | E_P},
$$
yields equality \eqref{eq c_r tangent quadric}.
\end{proof}

  Recall the following consequence of the ``Grothendieck-Riemann-Roch without denominators'' formula:
\begin{proposition}[{\cite[Example 15.4.2, (c)]{Fulton98}}]
\label{comparison Omega^1 and nu^ast Omega^1}
We have the equality in $\CH^\ast(\tilde{H})$:
\begin{equation} \label{eq blowing up total Chern class}
c(T_{\tilde{H}}) = \nu^\ast c(T_H) + \sum_{P \in \Sigma} \big[ (1 + \eta_P) (1-\eta_P)^N - 1\big].
\end{equation}
\end{proposition}

\begin{corollary}\label{comparison T rel blowup}
With the above notation, let $C$ be a smooth projective $k$-curve, let
$$f : H \lra C$$
be a flat morphism, and let
$$g := f \circ \nu : \tilde{H} \lra  C$$
be its composition with the blow-down morphism $\nu$. The following equality holds in $\CH^\ast(\tilde{H})$:
$$c([T_g]) = \nu^\ast c([T_f]) + \sum_{P \in \Sigma} \big[ (1 + \eta_P) (1-\eta_P)^N - 1\big].$$
 \end{corollary} 

\begin{proof}
By definition of the relative tangent class $[T_g]$ in $K^0(\tilde{H})$, the following equality holds in $\CH^\ast(\tilde{H})$:
\begin{align}
c([T_g]) &= c(T_{\tilde{H}}). g^* c(T_C)^{-1} \nonumber \\
& = \big(\nu^* c(T_H) + \sum_{P \in \Sigma} \big[ (1 + \eta_P) (1-\eta_P)^N - 1\big] \big) . g^* c(T_C)^{-1} \label{first eq T rel blow} \\
&= \nu^* (c(T_H) . f^* c(T_C)^{-1} ) + \sum_{P \in \Sigma} \big[ (1 + \eta_P) (1-\eta_P)^N - 1\big] . (1 - g^* c_1(T_C) ) \nonumber \\
&= \nu^* (c(T_H) . f^* c(T_C)^{-1} ) + \sum_{P \in \Sigma} \big[ (1 + \eta_P) (1-\eta_P)^N - 1\big] \label{second eq T rel blow} \\
&= \nu^* c([T_f]) + \sum_{P \in \Sigma} \big[ (1 + \eta_P) (1-\eta_P)^N - 1\big], \label{third eq T rel blow}
\end{align}
where in \eqref{first eq T rel blow} we used Proposition \ref{comparison Omega^1 and nu^ast Omega^1}, in \eqref{second eq T rel blow} we applied equality \eqref{pullback restricted to E_P} to the cycle of codimension one $\alpha := f^* c_1(T_C)$ and used that the polynomial $(1 + x) (1-x)^N - 1$ is a multiple of $x,$ and in \eqref{third eq T rel blow} we used the definition of the relative tangent class $[T_f]$ in $K^0(H).$
\end{proof}

The following corollary is a reformulation of Corollary \ref{comparison T rel blowup}, and we leave its derivation from Corollary \ref{comparison T rel blowup} to the reader.

\begin{corollary}\label{comparison cr T rel blowup} With the same notation as in Corollary \ref{comparison T rel blowup}, for every positive integer $r,$ the following equality holds in $\CH^r(\tilde{H})$:
$$c_r([T_g]) = \nu^\ast c_r([T_f]) + (-1)^r \Bigg [\binom{N}{r} - \binom{N}{r-1} \Bigg ] \sum_{P \in \Sigma} \eta_P^r.$$
\end{corollary}

 \section{The Chern classes $c_r(\omega^{1\vee}_{\widetilde{H}/C})$ in terms of~$\nu^\ast c_r([T_f])$}
In order to prove Lemma \ref{nu_ast c_1 c_N-1 omega points doubles}, we want to compare the cycle classes $\nu_\ast (c_1 c_{N-1})(\omega^{1 \vee}_{\tilde{H}/C})$ and $(c_1 c_{N-1})([T_f])$ on~$H$.

Using Corollary \ref{comparison T rel blowup},
we can compare the cycle classes $c_r([T_g])$ and $\nu^\ast c_r([T_f])$ on $\tilde{H}$, where $r$ is any positive integer.
In this Section, we shall apply equality \eqref{eq c_r omega and T_g when D^3 empty} from Proposition \ref{comparison c_r when D^3 empty non-reduced} to compare the cycle classes $c_r(\omega^{1 \vee} _{\tilde{H}/C})$ and $c_r([T_g])$ on $\tilde{H}$ (this will be the object of Proposition \ref{c_r omega for blow-up}), then we shall combine this with Corollary \ref{comparison T rel blowup} to compare the cycle classes $c_r(\omega^{1 \vee}_{\tilde{H}/C})$ and  $\nu^\ast c_r([T_f])$ (this will be the object of Corollary \ref{c_r omega for blow-up and T}). Next, we shall apply these expressions with the integers $r := 1$ and $r := N-1$ to compare the cycle classes $(c_1 c_{N-1})(\omega^{1 \vee}_{\tilde{H}/C})$ and $\nu^\ast (c_1 c_{N-1})([T_f])$ (this will be the object of Corollary \ref{c_1 c_N-1 omega T rel}). After pushing forward by $\nu$, this will establish Lemma \ref{nu_ast c_1 c_N-1 omega points doubles} and complete the proof of Theorem \ref{GK critical points}.

If $y$ denotes an indeterminate and $f(y)$ a formal series in $\C[[y]]$, for every integer $p$, we shall denote by $f(y)^{[p]}$ its coefficient of degree $p$. It vanishes if $p$ is negative.

\begin{proposition}
\label{formal identities}
For every $a$ in $\C^\ast$, and every $n$ and $r$ in $\N$, we have the formulas:
\begin{align}
\left[ \frac{(1 + y)^n}{1 + a y} \right]^{[n-r]} & = (-1)^{r} \sum_{r \leq k \leq n} \binom{n}{k} (-1)^k a^{k-r} = \frac{(-1)^{n+r}}{a^r} \big [(a-1)^n - \sum_{0 \leq k \leq r-1} \binom{n}{k} (-1)^{n-k} a^k \big ],
\label{non square denominator} \\
\intertext{and:}
\label{square denom lourd} \left[ \frac{(1 + y)^{n+1}}{(1 + a y)^2} \right ]^{[n-r]} &= (-1)^{r+1} \sum_{r+1 \leq k \leq n+1} (k-r) \binom{n+1}{k} (-1)^k a^{k-r-1} \\
\label{square denominator} &= \frac{(-1)^{n+r}}{a^{r+1}} \Big [\big(r + (n+1-r) a\big) (a-1)^n - \sum_{0 \leq k \leq r} (k - r) \binom{n+1}{k} (-1)^{n+1-k} a^k \Big ].
\end{align}
\end{proposition}

\begin{proof}
The equality between the left-hand side and the middle side in \eqref{non square denominator} follows from the computation:
\begin{align}
\left [\frac{(1+y)^n}{1+a y} \right ]^{[n-r]} &= \sum_{0 \leq i \leq n-r} [(1+y)^n]^{[i]} \Big [\frac{1}{1+ a y} \Big]^{[n-r-i]} \nonumber \\
&= \sum_{0 \leq i \leq n-r} \binom{n}{i} (-a)^{n-r-i} \nonumber \\
\label{first eq non sq den} &= \sum_{r \leq k \leq n} \binom{n}{n-k} (-a)^{k-r} \\
&= (-1)^r \sum_{r \leq k \leq n} \binom{n}{k} (-1)^k a^{k-r}, \nonumber
\end{align}
where in \eqref{first eq non sq den} we have introduced $k := n-i$.

The equality between the middle side and the right-hand side in \eqref{non square denominator} follows from the binomial formula applied to $(a-1)^n$.

Equality \eqref{square denom lourd} follows from the equality of formal series in $\C[[a,y]]$:
$$
\frac{\mathrm{d}}{\mathrm{d} a} \left [\frac{(1+y)^{n+1}}{1+a y} \right ]
= - y \frac{(1+y)^{n+1}}{(1+a y)^2},
$$
which implies the equality of complex numbers:
\begin{align}
\left [\frac{(1+y)^{n+1}}{(1+a y)^2} \right ]^{[n-r]}
&= - \bigg [y^{-1} \frac{\mathrm{d}}{\mathrm{d} a} \left (\frac{(1+y)^{n+1}}{1 + a y} \right ) \bigg ]^{[n-r]} \nonumber \\
&= - \frac{\mathrm{d}}{\mathrm{d} a} \left [\frac{(1 + y)^{n+1}}{1 + a y} \right ]^{[n+1-r]} \nonumber \\
\label{non sq for sq} &= - \frac{\mathrm{d}}{\mathrm{d} a} \big [(-1)^r \sum_{r \leq k \leq n+1} \binom{n+1}{k} (-1)^k a^{k-r} \big ] \\
&= (-1)^{r+1} \sum_{r+1 \leq k \leq n+1} (k-r) \binom{n+1}{k} (-1)^k a^{k-r-1}, \nonumber
\end{align}
where in \eqref{non sq for sq}, we have used \eqref{non square denominator} applied to $n' := n+1$.

Equality \eqref{square denominator} follows from the  expansion:
\begin{align}
\big(r + (n+1-r) a\big) (a-1)^n
&= (n+1) a (a-1)^n - r (a-1)^{n+1} \nonumber \\
&= (n+1) \sum_{1 \leq k \leq n+1} \binom{n}{k-1} (-1)^{n-k+1} a^k \nonumber \\
&- r \sum_{0 \leq k \leq n+1} \binom{n+1}{k} (-1)^{n+1-k} a^k \nonumber \\
&= \sum_{0 \leq k \leq n+1} \bigg ( (n+1) \binom{n}{k-1} - r \binom{n+1}{k} \bigg ) (-1)^{n+1-k} a^k \nonumber \\
\label{binom mult} &= \sum_{0 \leq k \leq n+1} (k-r) \binom{n+1}{k} (-1)^{n+1-k} a^k,
\end{align}
where in \eqref{binom mult}, we have used the classical identity, for every integer $k$:
$$
(n+1) \binom{n}{k-1} = k \binom{n+1}{k}. \qedhere
$$
\end{proof}

 Let us adopt the notation of Section \ref{Reduction of the proof}.

For every point $P$ in $\Sigma$, let us denote the class of the divisor $E_P$ by:
$$
\eta_P := [E_P] = c_1(\cO_{\tilde{H}}(E_P)) \in \CH^1(\tilde{H}).
$$

Observe that equalities \eqref{pullback restricted to E_P} and \eqref{product eta_P eta_P'} still hold in this setting.

\begin{proposition}
\label{c_r omega for blow-up}
With the above notation, for every non-negative integer $r$, we have the equality in $\CH^r(\tilde{H})$:
$$
c_r(\omega^{1 \vee}_{\tilde{H}/C}) = c_r([T_g]) + \alpha(N,r) \sum_{P \in \Sigma} \eta_P^r,
$$
where $\alpha(N,r)$ is the integer given by:
$$
\alpha(N,r) := (-1)^{r-1} \Bigg(\binom{N}{r-1} - 4 \Big[\frac{(1 + y)^N}{1 + 2 y} \Big]^{[r-2]} \Bigg).
$$
\end{proposition}

Observe that the integer $\alpha(N,1)$ is simply 1.
\begin{proof}
Let us apply equality \eqref{eq c_r omega and T_g when D^3 empty} from Proposition \ref{comparison c_r when D^3 empty non-reduced} to the scheme
$
Y := \tilde{H},
$
and to the morphism $g$, whose divisor of singular fibers $\tilde{H}_\Delta$ is of the form given by \eqref{decomposition Htilde_Delta}. Observe that with the notation of that proposition, the subscheme $D^3$ is empty.

We obtain, for every integer $r$, the equality in $\CH^r(\tilde{H})$:
\begin{equation}
c_r(\omega^{1 \vee}_{\tilde{H}/C}) = c_r([T_g]) + \sum_{P \in \Sigma} i_{E_P \ast} c_{r-1}(T_{E_P}) - 2 \sum_{P \in \Sigma} i_{E_P \cap W \ast} c_{r-2}(T_{E_P \cap W}). \label{first equality c_r omega blow-up}
\end{equation}

Using \eqref{eq c_r tangent exceptional divisor} from Proposition \ref{nu_* eta_p^N}, we have the equality in $\CH^{r-1}(E_P)$:
$$
c_{r-1}(T_{E_P}) = (-1)^{r-1} \binom{N}{r-1} \eta_{P | E_P}^{r-1},
$$
hence pushing forward by the inclusion in $\tilde{H}$ and using the projection formula, we have the equality in $\CH^r(\tilde{H})$:
$$
i_{E_P \ast} c_{r-1}(T_{E_P}) = (-1)^{r-1} \binom{N}{r-1} \eta_P^r.
$$

On the other hand, the intersection $Q_P := E_P \cap W$ is a smooth hypersurface of degree $m_P := 2$ in the projective space $E_P$, hence using equality \eqref{eq c_r tangent quadric} from Proposition \ref{nu_* eta_p^N}, we obtain the equality in $\CH^{r-1}(E_P)$:
$$
i_{E_P \cap W \ast} c_{r-2}(T_{E_P \cap W}) = 2 (-1)^{r-1} \Big[\frac{(1 + y)^N}{1 + 2 y} \Big]^{[r-2]} \eta_{P | E_P}^{r-1},
$$
hence pushing forward by the inclusion in $\tilde{H}$ and using the projection formula, we have the equality in $\CH^r(\tilde{H})$:
$$
i_{E_P \cap W \ast} c_{r-2}(T_{E_P \cap W}) = 2 (-1)^{r-1} \Big[\frac{(1 + y)^N}{1 + 2 y} \Big]^{[r-2]} \eta_P^r.
$$

Replacing in \eqref{first equality c_r omega blow-up} yields the result.
\end{proof}

 \begin{corollary}\label{c_r omega for blow-up and T}
With the above notation, for every positive integer $r$, the following equality holds in $\CH^r(\tilde{H})$:
$$
c_r(\omega^{1 \vee}_{\tilde{H}/C}) = \nu^\ast c_r([T_f]) + \beta(N,r) \sum_{P \in \Sigma} \eta_P^r,
$$
where $\beta(N,r)$ is the integer given by:
$$\beta(N,r) := (-1)^r \bigg [\binom{N}{r} - 2 \binom{N}{r-1} + 4 \Big [\frac{(1 + y)^N}{1 + 2 y} \Big]^{[r-2]} \bigg ].$$
\end{corollary}

  Observe that the integer $\beta(N,1)$ is equal to $-N+2.$

 \begin{proof} For every positive integer $r$, combining Corollary \ref{comparison cr T rel blowup} and Proposition \ref{c_r omega for blow-up} yields the equality in $\CH^r(\tilde{H})$:
$$c_r(\omega^{1 \vee}_{\tilde{H}/C}) = \nu^\ast c_r([T_f]) + \gamma(N,r) \sum_{P \in \Sigma} \eta_P^r,$$
where $\gamma(N,r)$ is the integer given by:
\begin{align*}
\gamma(N,r) &= \alpha(N,r) + (-1)^r \Bigg [\binom{N}{r} - \binom{N}{r-1} \Bigg ] \\
& = (-1)^{r-1} \Bigg[\binom{N}{r-1} - 4 \Big[\frac{(1 + y)^N}{1 + 2 y} \Big]^{[r-2]} \Bigg] 
+ (-1)^r \Bigg [\binom{N}{r} - \binom{N}{r-1} \Bigg ] \\
&= \beta(N,r). \qedhere
\end{align*}
\end{proof}

 \begin{corollary}\label{c_1 c_N-1 omega T rel}
With the above notation, the following equality holds in $\CH^N(\tilde{H)}$:
\begin{equation}\label{eq c_1 c_N-1 omega T rel} 
(c_1 c_{N-1})(\omega^{1 \vee}_{\tilde{H} / C})
= \nu^* (c_1 c_{N-1})([T_f])
- (N-2) \frac{1 - (-1)^N}{2} \sum_{P\in \Sigma} \eta_P^N.
\end{equation}
\end{corollary}

\begin{proof}
If the integer $N$ is 1, this follows directly from Corollary \ref{c_r omega for blow-up and T} applied to the integer $r := 1$ (observe that the second term in the right-hand side of \eqref{eq c_1 c_N-1 omega T rel} is $\sum_{P\in \Sigma} \eta_P$).

  Let us assume that the integer $N$ is at least 2.
Applying Corollary \ref{c_r omega for blow-up and T} to the positive integers $r = 1$ and $r= N-1,$ then multiplying, yields the equality in $\CH^N(\tilde{H})$:
\begin{align}
(c_1 c_{N-1})(\omega^{1 \vee}_{\tilde{H} / C})
&= (\nu^* c_1([T_f]) + \beta(N,1) \sum_{P \in \Sigma}\eta_P) . (\nu^* c_{N-1}([T_f]) + \beta(N, N-1) \sum_{P\in \Sigma} \eta_P^{N-1} ) \nonumber \\
&= \nu^* (c_1 c_{N-1})([T_f])
+ \beta(N, N-1) \sum_{P \in \Sigma} \eta_P^{N-1} \nu^* c_1([T_f]) \nonumber \\
&\quad + \beta(N,1) \sum_{P\in \Sigma} \eta_P \nu^* c_{N-1}([T_f])
+ \beta(N,1) \beta(N, N-1) \sum_{(P, P') \in \Sigma^2}\eta_P \eta_{P'}^{N-1}. \label{first eq c_1 c_N-1 omega T rel}
\end{align}

Since we have assumed that $N$ was at least $2$, we can apply equality \eqref{pullback restricted to E_P} to the cycle $\alpha := c_{N-1}([T_f])$ of positive codimension to obtain, for every point $P$ in $\Sigma$, the equality in $\CH^N(\tilde{H})$:
$$
\eta_P \, \nu^\ast c_{N-1}([T_f]) = 0.
$$

Furthermore, we can apply this same equality to the cycle $\alpha := c_1([T_f])$ to obtain, for every point $P$ in $\Sigma$, the equality in $\CH^N(\tilde{H})$:
$$
\eta_P^{N-1} \, \nu^\ast c_1([T_f]) = \eta_P^{N-2} \big(\eta_P \, \nu^\ast c_1([T_f])\big) = 0.
$$

Finally, applying equality \eqref{product eta_P eta_P'}, we obtain the equality in $\CH^N(\tilde{H})$:
$$
\sum_{P, P' \in \Sigma} \eta_P\, \eta_{P'}^{N-1} = \sum_{P \in \Sigma} \eta_P^N.
$$

Hence replacing in \eqref{first eq c_1 c_N-1 omega T rel}:
\begin{equation}
(c_1 c_{N-1})(\omega^{1 \vee}_{\tilde{H} / C})
= \nu^* (c_1 c_{N-1})([T_f])
+ \beta(N,1) \beta(N, N-1) \sum_{P \in \Sigma}\eta_P^N. \label{second eq c_1 c_N-1 omega T rel}
\end{equation}

 Now let us compute the combinatorial coefficient $\beta(N,1) \beta(N,N-1).$ Applying equality \eqref{non square denominator} from Proposition \ref{formal identities} to the integers $r = 3$ and $n = N$ and to the complex number $a = 2$ yields the equality:
\begin{align}
\left[ \frac{(1 + y)^N}{1 + 2 y} \right]^{[N-3]} & = \frac{(-1)^{N+3}}{2^3} \Big[(2-1)^N - \sum_{0 \leq k \leq 2} \binom{N}{k} (-1)^{N-k} 2^k\Big] \nonumber \\
&= \frac{(-1)^{N+1}}{8}  \big[1 + (-1)^{N+1} + 2 N (-1)^N + 4 N (N-1)/2 (-1)^{N+1} \big] \nonumber \\
&= \frac{(-1)^{N+1}}{8}\big[1 + (-1)^{N+1} (1 - 2 N + 2 N (N-1) ) \big] \nonumber \\
&= \frac{1}{8} \big[(-1)^{N+1} + 2 N^2 - 4 N + 1\big]. \nonumber 
\end{align}

Consequently, the combinatorial coefficient $\beta(N,1) \beta(N, N-1)$ is given by:
\begin{align*}
\beta(N,1) \beta(N, N-1)
&= (-N+2)  (-1)^{N-1} \bigg[\binom{N}{N-1} - 2 \binom{N}{N-2} + 4 \Big[\frac{(1 + y)^N}{1 + 2 y} \Big]^{[N-3]} \bigg] \\
&= (-1)^N (N-2) \big[N - 2 N (N-1)/2 + 4/8 [(-1)^{N+1} + 2 N^2 - 4 N + 1] \big] \\
&= (-1)^N (N-2) \big[N - N (N-1) + (1 + (-1)^{N+1})/2 + N^2 - 2 N\big] \\
&= (-1)^N (N-2) (1 + (-1)^{N+1})/2 \\
&= - (N-2) (1 - (-1)^N)/2.
\end{align*}
  
  Replacing this expression in \eqref{second eq c_1 c_N-1 omega T rel} yields the result.
\end{proof}

\begin{proof}[Proof of Lemma \ref{nu_ast c_1 c_N-1 omega points doubles}]
By Corollary \ref{c_1 c_N-1 omega T rel}, we have the equality in $\CH^N(\tilde{H})$:
$$
(c_1 c_{N-1})(\omega^{1 \vee}_{\tilde{H}/C})
= \nu^\ast (c_1 c_{N-1})([T_f]) - (N-2) \frac{1 - (-1)^{N}}{2} \sum_{P \in \Sigma} \eta_P^N,
$$
hence pushing forward by the morphism $\nu$, whose generic degree is one, and using equality \eqref{eq nu_* eta_p^N} from Proposition \ref{nu_* eta_p^N}, we have the equality in $\CH^N(H)$:
\begin{align*}
\nu_\ast (c_1 c_{N-1})(\omega^{1 \vee}_{\tilde{H}/C}) &= (c_1 c_{N-1})([T_f]) - (N-2) \frac{1 - (-1)^N}{2} (-1)^{N-1} [\Sigma] \\
&= (c_1 c_{N-1})([T_f]) - (N-2) \frac{1 - (-1)^N}{2} [\Sigma],
\end{align*} 
as wanted.
\end{proof}

\chapter[The Griffiths height of a pencil of hypersurfaces]{The Griffiths height of the middle-dimensional cohomology of a pencil of hypersurfaces}

\section{Families of ample divisors in a smooth pencil}
\label{Ample hyp}

The objective of this section is to prove the following proposition.

\begin{proposition}
\label{technical result}
Let $C$ be a connected smooth projective complex curve with generic point $\eta$, $X$ be a smooth projective complex scheme of pure dimension $N+1$, and let
$$\pi : X \lra C$$
be a smooth surjective morphism of complex schemes. Let $H$ be a non-singular hypersurface in $X$ such that the morphism
$$\pi_{| H} : H \lra C$$
is flat\footnote{or equivalently, is non-constant on every connected component of $H$, or has fibers of pure dimension $N-1$.} and has a finite set $\Sigma$ of critical points, all of which are non-degenerate.

If we denote by $L$ the line bundle $\cO_X(H)$ on $X$,  then the following equality holds in $\CH_0(X)$:
\begin{equation}\label{pas intro SigmaLX}
[\Sigma]= \big( (1-c_1(L))^{-1} c(\Omega^1_{X/C}) \big)^{(N+1)},
\end{equation}

If moreover the line bundle $L$ is ample relatively to~$\pi$, 
then the following equalities hold in the Chow group with rational coefficients $\CH_0(C)_\Q$:
\begin{multline}\label{pas intro GK+XL}
c_1(\GK_{C,+}(\H^{N-1}(H_\eta / C_\eta)))  \\ = c_1(\GK_C(\H^{N-1}(X/C))) + c_1(\GK_C(\H^{N+1}(X/C))) - c_1(\GK_C(\H^N(X/C))) \\
+ \frac{1}{12} \pi_\ast \Big [ \big( (1-c_1(L))^{-1}c_1(\Omega^1_{X/C}) c(\Omega^1_{X/C})\big)^{(N+1)} \Big ]
-\frac{1}{12} \pi_\ast \big( c_1(L)c_N(\Omega^1_{X/C})\big) + v^+_N\,  \pi_\ast [\Sigma],
\end{multline}
and:
\begin{multline}\label{pas intro GK-XL}
c_1(\GK_{C,-}(\H^{N-1}(H_\eta / C_\eta))) \\ = c_1(\GK_C(\H^{N-1}(X/C))) + c_1(\GK_C(\H^{N+1}(X/C))) - c_1(\GK_C(\H^N(X/C)))  \\
+ \frac{1}{12} \pi_\ast \Big [ \big( (1-c_1(L))^{-1}c_1(\Omega^1_{X/C}) c(\Omega^1_{X/C})\big)^{(N+1)} \Big ]
-\frac{1}{12} \pi_\ast \big( c_1(L)c_N(\Omega^1_{X/C})\big) + v^-_N\,  \pi_\ast [\Sigma],
\end{multline}
where:
\begin{equation*}
v^+_N := 
\begin{cases}
 7(N-1)/24 & \text{if $N$ is odd} \\
 (N+2)/24 & \text{if $N$ is even}.
\end{cases}
\end{equation*}
and:
\begin{equation*}
v^-_N := 
\begin{cases}
 - 5(N-1)/24 & \text{if $N$ is odd} \\
 (N+2)/24 & \text{if $N$ is even}.
\end{cases}
\end{equation*}
\end{proposition}

\subsection{Computation of $i_\ast [\Sigma]$ and of $i_\ast (c_1 c_{N-1})([T_f])$}\label{Sigma pencil smooth}

Let $C$ be a connected smooth projective complex curve with generic point $\eta$, let $X$ be a smooth projective complex scheme of pure dimension $N+1$, and let
$$\pi : X \lra C$$
be a smooth surjective morphism of complex schemes. Let $H$ be a non-singular hypersurface in $X$ such that the morphism
$$f := \pi_{| H} : H \lra C$$
is flat and has a finite set $\Sigma$ of critical points, all of which are non-degenerate.

Let
$$
i : H \lra X
$$
be the inclusion map. Let us also denote by $L$ the line bundle $\cO_X(H)$ on $X$.

Observe that the line bundle $L^\vee_{| H}$ on $H$ can be identified with the conormal line bundle $\cN_H X^\vee$ of the hypersurface $H$ in $X$.

Let us define a map $s$ of vector bundles on $H$ by the composition:
$$
s : L^\vee _{| H} \simeq \cN_H X^\vee \lra \Omega^1_{X | H} \lra \Omega^1_{X/C | H}.
$$
We will also see $s$ as a section of the vector bundle $\Omega^1_{X/C | H} \otimes L_{| H}$.

Some of the content of the upcoming results, Lemmas \ref{subscheme s} and \ref{ex sequence hyp}, can be summarized in the commutative diagram below, where all the horizontal and vertical sequences are exact sequences of coherent sheaves on $H$:
\begin{equation} \label{diagr hyp omega}
    \xymatrix{
    & & f^\ast \Omega^1_C \ar[d]_{{}^t (D \pi)}\ar@{=}[r]  &  f^\ast \Omega^1_C \ar[d]^{{}^t (D f)} & \\
    0 \ar[r] & \cN_H X^\vee \ar@{=}[d] \ar[r]  & \Omega^1_{X \mid H} \ar[r]^{{}^t (D i)} \ar[d] & \Omega^1_H \ar[r] \ar[d] & 0 \\
       0 \ar[r] & \cN_H X^\vee \ar[r]^{s} & \Omega^1_{X/C \mid H} \ar[r]^{{}^t (D i)} \ar[d] & \Omega^1_{H/C} \ar[r] \ar[d] & 0 \\
    & & 0 & 0 &}
\end{equation}
In \eqref{diagr hyp omega}, ${}^t (D f)$ (resp. ${}^t (D \pi)$, ${}^t (D i)$) denotes the morphism between coherent sheaves of K\"ahler differentials induced by the morphism of schemes $f$ (resp. $\pi$, $i$).

\begin{lemma}
\label{subscheme s}
The section $s$ does not vanish on the open dense subset $H - \Sigma$ in $H$, and for every point $P$ in $\Sigma$, there exists a  system $(z_1,...,z_N)$ of local analytic coordinates on $H$ near $P$ such that, in some local frame of the vector bundle $\Omega^1_{X/C | H} \otimes L_{| H}$, the section $s$ is given by:
$$
s = (2 z_i)_{1 \leq i \leq N}.
$$
\end{lemma} 
\begin{proof}
By definition of the subset $\Sigma$, the morphism of non-singular complex schemes $f$ is smooth on the open dense subset $H - \Sigma$, so that by elementary properties of the sheaf of K\"ahler differentials, the section $s$ does not vanish on $H - \Sigma$.

Now, let $P$ be a point in $\Sigma$, and let $y$ in $C$ be its image by $f$. The hypersurfaces $H$ and $X_y$ in $X$ are both non-singular and meet in $P$. By definition of $\Sigma$, the morphism $f$ has a critical point in $P$, so we have the equality of hyperplanes in $T_{X,P}$:
\begin{equation}
T_{H,P} = T_{X_y, P} \label{eq tangent two hyp}
\end{equation}
Let $t$ be a local coordinate of $C$ in some analytic neighborhood of $y$. Since the morphism $\pi$ is smooth, we can extend the function $\pi^\ast t$ into a local coordinate system $(\pi^\ast t, z_1,...,z_N)$ in some analytic neighborhood $V$ of $P$.

Using \eqref{eq tangent two hyp} and  the implicit function theorem, after possibly shrinking $V$, there exists a diffeomorphism:
$$
\phi : X_y \cap V \lra H \cap V
$$
which sends $P$ to $P$ and whose differential at $P$ is the identity map 
$$
\Id : T_{X_y,P} \lra T_{H, P}.
$$
In other terms, there exists an analytic neighborhood $U$ of $0$ in $\C^N$, and a function: 
$$
u : U \lra \C,
$$
vanishing in the origin, and whose derivative vanishes in the origin, and such that $H \cap V$ is defined by the equation:
\begin{equation}
\pi^\ast t = u(z_1,...,z_N). \label{eq hyp implicit theorem}
\end{equation}
The morphism $\phi$ and the function $u$ are related by the equality of morphisms from $X_y$ to $\C^{N+1}$:
\begin{equation}
   (\pi^\ast t, z_1,...,z_N) \circ \phi = (u(z_1,...,z_N),z_1,...,z_N). \label{phi and u in coordinates}
\end{equation}
Observe that $H$ admits a local coordinate system given by $(z_{1 | H},...,z_{N | H})$.

By hypothesis, $y$ is a non-degenerate critical point of the morphism $f = \pi_{| H}$ on $H$, hence it is also a non-degenerate critical point of the morphism:
$$
\pi_{| H \cap V} \circ \phi : X_y \cap V \lra C.
$$
Using \eqref{phi and u in coordinates}, this morphism is given in local coordinates by:
$$
(\pi^\ast t)_{| H \cap V} \circ \phi = u(z_1,...,z_N),
$$
so we obtain that the Hessian at $0$ of the function $u$ is non-degenerate. Consequently, after possibly shrinking the neighborhood $V$ and changing the coordinates $z_1,...,z_N$, we can assume that the function $u$ is given by:
\begin{equation}
u(z_1,...,z_N) = z_1^2 + ... + z_N^2. \label{u non-degenerate}
\end{equation}
The vector bundle $\Omega^1_X$ admits a local frame given by $(\pi^\ast \mathrm{d} t, \mathrm{d} z_1,..., \mathrm{d} z_N)$. Using that the hypersurface $H$ is defined by the equation \eqref{eq hyp implicit theorem}, the conormal line bundle $L_{| H}^\vee \simeq \cN_H X^\vee$, as embedded in the vector bundle $\Omega^1_{X | H}$, is generated by the non-vanishing section:
\begin{equation*}
s_0 := \frac{\partial u }{\partial z_1}  \mathrm{d} z_1 + ... + \frac{\partial u }{\partial z_N} \mathrm{d} z_N - \pi^\ast \mathrm{d} t 
= 2 z_{1 | H} \mathrm{d} z_1 + ... + 2 z_{N | H} \mathrm{d} z_N - \pi^\ast \mathrm{d} t \textrm{ using \eqref{u non-degenerate}}.
\end{equation*}
The vector bundle $\Omega^1_{X/C}$ admits a local frame given by $([\mathrm{d} z_1],...,[\mathrm{d} z_N])$. In this frame, the image by the projection
$$
\Omega^1_{X | H} \lra \Omega^1_{X/C | H}
$$
of the section $s_0$ is given by:
$$
2 z_{1 | H} [\mathrm{d} z_1] + ... + 2 z_{N | H} [\mathrm{d} z_N].
$$
This image is precisely the section $s$, which proves the lemma.
\end{proof}

\begin{lemma}
\label{ex sequence hyp} 
We have an exact sequence of coherent sheaves on $H$:
\begin{equation}
0 \lra L^\vee _{| H} \overset{s}{\lra} \Omega^1_{X/C | H} \overset{{}^t (D i)}{\lra} \Omega^1_{H/C} \lra 0, \label{eq ex sequence hyp}
\end{equation}
where ${}^t (D i)$ is the map between sheaves of K\"{a}hler differentials associated with the morphism of $C$-schemes $i$.

Furthermore, the subscheme of $H$ defined by the vanishing of $s$ is precisely the reduced subscheme defined by the finite subset $\Sigma$.
\end{lemma}

\begin{proof}
For the exactness of \eqref{eq ex sequence hyp}, everything but the injectivity of the morphism $s$ is a standard property of sheaves of relative K\"{a}hler differentials. By Lemma \ref{subscheme s}, the morphism $s$ does not vanish on the open dense subset $H - \Sigma$ in $H$. Since both coherent sheaves $L^\vee _{| H}$ and $\Omega^1_{X/C | H}$ are locally free on $H$, the morphism of coherent sheaves $s$ is therefore injective, which completes the proof of the exactness of \eqref{eq ex sequence hyp}.

Furthermore, by Lemma \ref{subscheme s}, the morphism $s$ vanishes with order $1$ in every point of $\Sigma$. Consequently, the subscheme defined by the vanishing of the morphism $s$ is precisely the reduced subscheme defined by $\Sigma$, as wanted.
\end{proof}

\begin{proposition}
\label{Sigma for hyp}
We have the following equalities in $\CH_0(X)$:
\begin{equation}
i_\ast [\Sigma] = [(1 - c_1(L))^{-1} c(\Omega^1_{X/C}) ]^{(N+1)}, \label{eq Sigma for hyp}
\end{equation}
and:
\begin{multline}
i_\ast (c_1 c_{N-1})([T_f]) = (c_1 c_N)([T_\pi]) \\+ (-1)^N \big([(1 - c_1(L))^{-1} c_1(\Omega^1_{X/C})
c(\Omega^1_{X/C}) ]^{(N+1)} + i_\ast [\Sigma] - c_1(L) c_N(\Omega^1_{X/C}) \big). \label{eq class tangent hyp}
\end{multline}
\end{proposition}

\begin{proof}
Using Lemma \ref{ex sequence hyp}, the subscheme defined by the vanishing of the section $s$ of the vector bundle $\Omega^1_{X/C | H} \otimes L_{| H}$ on $H$ is exactly the reduced subscheme defined by $\Sigma$, in particular, it is of dimension $0$, hence of codimension $N$ in $H$. The rank of the vector bundle $\Omega^1_{X/C | H} \otimes L_{| H}$ on $H$ is exactly $N$, so that $s$ is a regular section in the sense of \cite[B.3.4]{Fulton98}.

Consequently, using \cite[Example 3.2.16, (ii)]{Fulton98}, we have the equality in $\CH_0(H)$:
\begin{align*}
[\Sigma] &= c_N(\Omega^1_{X/C | H} \otimes L_{| H}),\\
&= i^\ast c_N(\Omega^1_{X/C} \otimes L),\\
&= i^\ast \sum_{0 \leq k \leq N} c_1(L)^{N-k} c_k(\Omega^1_{X/C}),
\end{align*}
where the last expression follows from the formula for the top Chern class of the tensor product of a vector bundle by a line bundle (\cite[Remark 3.2.3, (b)]{Fulton98}).

Pushing forward by the inclusion $i$ and applying the projection formula, we obtain the equality in $\CH_0(X)$:
\begin{align}
i_\ast [\Sigma] &= [H] \cdot \sum_{0 \leq k \leq N} c_1(L)^{N-k} c_k(\Omega^1_{X/C}), \nonumber \\
&= c_1(L) \sum_{0 \leq k \leq N} c_1(L)^{N-k} c_k(\Omega^1_{X/C}), \nonumber\\
&= \sum_{0 \leq k \leq N} c_1(L)^{N+1-k} c_k(\Omega^1_{X/C}),\label{isig}\\
&= \sum_{0 \leq k \leq N+1} c_1(L)^{N+1-k} c_k(\Omega^1_{X/C}) \textrm{ because the vector bundle $\Omega^1_{X/C}$ is of rank $N$,} \nonumber \\
&= [(1 - c_1(L))^{-1} c(\Omega^1_{X/C}) ]^{(N+1)}, \nonumber
\end{align}
which shows equality \eqref{eq Sigma for hyp}.

Now let us show equality \eqref{eq class tangent hyp}. Using exact sequence \eqref{eq ex sequence hyp} and the multiplicativity of total Chern classes, we have the equality in $\CH^\ast(H)$:
$$
c([\Omega^1_{H/C}]) = c(\Omega^1_{X/C | H}) c(L^\vee_{| H})^{-1}
= i^\ast [(1 - c_1(L))^{-1} c(\Omega^1_{X/C})]. 
$$
Taking the terms of codimension $1$, we obtain the equality in $\CH^1(H)$:
\begin{equation}
c_1([\Omega^1_{H/C}]) = i^\ast [c_1(\Omega^1_{X/C}) + c_1(L)], \label{c1 Omega hypersurface}
\end{equation}
and taking the terms of codimension $N-1$, we obtain the equality in $\CH^{N-1}(H)$:
\begin{equation}
c_{N-1}([\Omega^1_{H/C}]) = i^\ast \sum_{0 \leq k \leq N-1} c_1(L)^{N-1-k} c_k(\Omega^1_{X/C}). \label{cN-1 Omega hypersurface}
\end{equation}
Multiplying \eqref{c1 Omega hypersurface} and \eqref{cN-1 Omega hypersurface}, we obtain the equality in $\CH^N(H) \simeq \CH_0(H)$:
\begin{align*}
(c_1 c_{N-1})([\Omega^1_{H/C}]) 
&= i^\ast \Big[\sum_{0 \leq k \leq N-1} c_1(L)^{N-k} c_k(\Omega^1_{X/C})+ \sum_{0 \leq k \leq N-1} c_1(L)^{N-1-k} (c_1 c_k)(\Omega^1_{X/C})\Big].
\end{align*}
Pushing forward by $i$ and applying the projection formula, we obtain the equality in $\CH_0(X)$:
\begin{align}
i_\ast (c_1 c_{N-1})([\Omega^1_{H/C}]) &= c_1(L) . \Big[\sum_{0 \leq k \leq N-1} c_1(L)^{N-k} c_k(\Omega^1_{X/C}) + \sum_{0 \leq k \leq N-1} c_1(L)^{N-1-k} (c_1 c_k)(\Omega^1_{X/C})\Big],\nonumber \\
&= \sum_{0 \leq k \leq N-1} c_1(L)^{N+1-k} c_k(\Omega^1_{X/C}) + \sum_{0 \leq k \leq N-1} c_1(L)^{N-k} (c_1 c_k)(\Omega^1_{X/C}). \label{first eq computation c tangent hyp}
\end{align}
The first sum in \eqref{first eq computation c tangent hyp} is given by:
\begin{align*}
\sum_{0 \leq k \leq N-1} c_1(L)^{N+1-k} c_k(\Omega^1_{X/C}) 
&= \sum_{0 \leq k \leq N} c_1(L)^{N+1-k} c_k(\Omega^1_{X/C}) - c_1(L) c_N(\Omega^1_{X/C}),\\
&= i_\ast [\Sigma] - c_1(L) c_N(\Omega^1_{X/C}) \textrm{ using equality \eqref{isig} 
}.
\end{align*}
Similarly, the second sum in \eqref{first eq computation c tangent hyp} is given by:
\begin{align*}
\sum_{0 \leq k \leq N-1} c_1(L)^{N-k} (c_1 c_k)(\Omega^1_{X/C})
&= \sum_{0 \leq k \leq N} c_1(L)^{N-k} (c_1 c_k)(\Omega^1_{X/C}) - (c_1 c_N)(\Omega^1_{X/C}),\\
&= [(1 - c_1(L))^{-1} c_1(\Omega^1_{X/C}) c(\Omega^1_{X/C}) ]^{(N+1)} - (c_1 c_N)(\Omega^1_{X/C}).
\end{align*}
Hence replacing in \eqref{first eq computation c tangent hyp}, we obtain the equality in $\CH_0(X)$:
$$
i_\ast (c_1 c_{N-1})([\Omega^1_{H/C}]) = i_\ast [\Sigma] - c_1(L) c_N(\Omega^1_{X/C}) + [(1 - c_1(L))^{-1} c_1(\Omega^1_{X/C}) c(\Omega^1_{X/C}) ]^{(N+1)} - (c_1 c_N)(\Omega^1_{X/C}).
$$
Since the morphism of non-singular complex schemes $f$ is smooth on an open dense subset of $H$, similarly to Subsection \ref{Comparing the characteristic classes of T and omega}, we have the following equality in $K^0(H) \simeq K_0(H)$:
$$
[T_f] = [\Omega^1_{H/C}] ^\vee,
$$
where $.^\vee$ denotes the duality involution on $K^0(H)$.

Consequently, we have the equality in $\CH_0(X)$:
\begin{align*}
i_\ast (c_1 c_{N-1})([T_f]) &= (-1)^N i_\ast (c_1 c_{N-1})([\Omega^1_{H/C}]),\\
&= (c_1 c_N)([T_\pi]) \\
&+ (-1)^N \big([(1 - c_1(L))^{-1} c_1(\Omega^1_{X/C}) c(\Omega^1_{X/C})]^{(N+1)} + i_\ast [\Sigma] - c_1(L) c_N(\Omega^1_{X/C})\big),
\end{align*}
which shows equality \eqref{eq class tangent hyp}.
\end{proof}

\subsection{An application of Lefschetz's weak theorem}

We can also prove the following consequence of Lefschetz's weak theorem and of Poincar\'{e} duality.

\begin{proposition} 
\label{GK Lefschetz}
If the line bundle $L$ is ample relatively to the morphism $\pi$, then for every integer $n$ such that $n < N-1$, the local monodromy of $\H^n(H_\eta / C_\eta)$ is unipotent and we have the equality in $\CH_0(C)$:
\begin{equation}
c_1(\GK_C(\H^n(H_\eta / C_\eta))) = c_1(\GK_C(\H^n(X/C))), \label{eq GK Lefschetz small n}
\end{equation}
and for every integer $n$ such that $n > N-1$, the local monodromy of $\H^n(H_\eta / C_\eta)$ is unipotent and we have the equality in $\CH_0(C)_\Q$:
\begin{equation}
c_1(\GK_C(\H^n(H_\eta / C_\eta))) = c_1(\GK_C(\H^{n+2}(X/C))). \label{eq GK Lefschetz large n}
\end{equation}
\end{proposition}

A minor variant of the proof below would   actually show that the local monodromy of $\H^n(H_\eta / C_\eta)$ is trivial when $n \neq N-1$.  

\begin{proof}

Let us assume that the line bundle $L$ is ample relatively to the morphism $\pi$. Let $\Delta$ be the set-theoretic image:
$$
\Delta := f(\Sigma),
$$
and let: 
$$\Cc:= C - \Delta$$
be its complement in $C$. 

Let $n$ be an integer such that $n < N-1$.

Since the line bundle $L = \cO_X(H)$ is ample relatively to the morphism $\pi$, using Lefschetz's weak theorem (see for instance \cite[Theorem 1.29]{Voisin03}), we obtain that the pullback morphism of variations of Hodge structures on $\Cc$:
$$
\H^n(i^\ast) : \H^n(X - X_\Delta / \Cc) \lra \H^n(H - H_\Delta / \Cc)
$$
is an isomorphism.

Since $X$ is smooth over $C$, the VHS $\H^n(X - X_\Delta /\Cc)$ on $\Cc$ can be extended into the VHS $\H^n(X/C)$ on $C$, hence its local monodromy at every point of $\Delta$ is trivial.

Consequently, the local monodromy of the VHS $\H^n(H - H_\Delta / \Cc)$ at every point of $\Delta$ is trivial, in particular unipotent, and the isomorphism of VHS $\H^n(i^\ast)$ extends into an isomorphism on the (upper or lower) Deligne extension.

Consequently, it induces an isomorphism of line bundles on $C$:
$$
\GK_C(\H^n(i^\ast)) : \GK_C(\H^n(X/C)) \lra \GK_C(\H^n(H_\eta / C_\eta)).
$$
Taking Chern classes, we obtain immediately equality \eqref{eq GK Lefschetz small n}.

Now, let $n$ be an integer such that $n > N-1$. Applying Proposition \ref{GK Poincare duality} to the blow-up $Y$ of $H$ at every point of $\Sigma$, satisfying that the divisor $Y_\Delta$ is a divisor with strict normal crossings, and to the integer $$n' := 2(N-1)-n < N -1,$$ and using that the local monodromy at every point of $\Delta$ of the VHS 
$\H^{n'}(Y - Y_\Delta / \Cc)$ is unipotent (which we showed above), we obtain that the local monodromy at every point of $\Delta$ of the VHS $\H^n(Y - Y_\Delta / \Cc)$, hence of the VHS $\H^n(H_\eta / C_\eta)$, is unipotent and that the line bundle on $C$:
\begin{multline*}
\GK_C(\H^n(Y - Y_\Delta / \Cc)) \otimes \GK_C(\H^{2(N-1)-n}(Y - Y_\Delta / \Cc))^\vee \\
\simeq \GK_C(\H^n(H_\eta / C_\eta)) \otimes \GK_C(\H^{2(N-1)-n}(H_\eta / C_\eta))
\end{multline*}
is of $2$-torsion. Consequently, we have the equality in $\CH_0(C)_\Q$,
$$
c_1(\GK_C(\H^n(H_\eta / C_\eta)))
= c_1(\GK_C(\H^{2(N-1)-n}(H_\eta / C_\eta))).
$$
Applying equality \eqref{eq GK Lefschetz small n} to the integer $n' := 2(N-1)-n < N-1$, we obtain:
$$
c_1(\GK_C(\H^n(H_\eta / C_\eta))) = c_1(\GK_C(\H^{2(N-1)-n}(X/C))).
$$
Applying Proposition \ref{GK Poincare duality} to $X$, with the set of critical values being empty, we obtain:
$$
c_1(\GK_C(\H^n(H_\eta / C_\eta))) = c_1(\GK_C(\H^{2N - (2(N-1)-n)}(X/C))) = c_1(\GK_C(\H^{n+2}(X/C))),
$$
which shows equality \eqref{eq GK Lefschetz large n}.
\end{proof}

Now we can conclude the proof of Proposition \ref{technical result}. Let us assume that the line bundle $L$ is ample relatively to the morphism $\pi$.

Applying equality \eqref{pas intro GK- points doubles} from Theorem \ref{GK critical points} to the $C$-scheme $H$, then applying equality \eqref{eq class tangent hyp} from Proposition \ref{Sigma for hyp}, we have the equalities in $\CH_0(C)_\Q$:
\begin{align}
&\sum_{n=0}^{2(N-1)} (-1)^{n-1} c_1\big(\GK_{C, -}(\H^n(H_\eta / C_\eta))\big) \nonumber \\
&= \frac{1}{12} f_\ast \big((c_1 c_{N-1})([\Omega^1_{H/C}]^\vee)\big) + u_N^- f_\ast [\Sigma],\nonumber \\
&= \frac{1}{12} \pi_\ast (c_1 c_N)([T_\pi]) + (u_N^- + \frac{(-1)^N}{12}) f_\ast [\Sigma] \nonumber \\
&+ \frac{(-1)^N}{12} \pi_\ast \big(
[(1 - c_1(L))^{-1} c_1(\Omega^1_{X/C}) c(\Omega^1_{X/C}) ]^{(N+1)} - c_1(L) c_N(\Omega^1_{X/C}) \big), \label{first eq technical result}
\end{align}
where $u_N^-$ is the rational number defined by: 
\begin{equation*}
u_N^- := 
\begin{cases}
 (5N-3)/24 & \text{if $N$ is odd} \\
 N/24 & \text{if $N$ is even}.
\end{cases}
\end{equation*}

Applying Theorem \ref{GK critical points} to the smooth $C$-scheme $X$ yields the equality in $\CH_0(C)_\Q$:
\begin{equation}
\sum_{n=0}^{2N} (-1)^{n-1} c_1(\GK_C(\H^n(X/C)))
= \frac{1}{12} \pi_\ast (c_1 c_N)([T_\pi]). \label{second eq technical result}
\end{equation}
Finally, it follows from Proposition \ref{GK Lefschetz} that we have the equality in $\CH_0(C)_\Q$:
\begin{multline}
\sum_{n=0}^{2(N-1)} (-1)^{n-1} c_1\big(\GK_{C, -}(\H^n(H_\eta / C_\eta))\big) - \sum_{n=0}^{2N} (-1)^{n-1} c_1(\GK_C(\H^n(X/C))) \\= (-1)^N [c_1(\GK_{C,-}(\H^{N-1}(H_\eta / C_\eta))) - c_1(\GK_C(\H^{N-1}(X/C))) \\
+ c_1(\GK_C(\H^N(X/C))) - c_1(\GK_C(\H^{N+1}(X/C)))]. \label{third eq technical result}
\end{multline}
Equality \eqref{pas intro GK-XL} follows by combining equalities \eqref{first eq technical result}, \eqref{second eq technical result} and \eqref{third eq technical result}. Equality \eqref{pas intro GK+XL} is proved in the same way, using equality \eqref{pas intro GK+ points doubles} from Theorem \ref{GK critical points} instead of equality \eqref{pas intro GK- points doubles}.

This concludes the proof of Proposition \ref{technical result}.

\section{Pencils of hypersurfaces in the projective space}

In this section, we shall use  Proposition \ref{technical result} to compute the Griffiths height of the middle -dimensional cohomology  of pencils of projective hypersurfaces. Before stating our results, we introduce some notation and recall some basic facts concerning projective bundles over a projective curve, their Chow groups, and their hypersurfaces.

\subsection{Preliminary: projective bundles over a curve and horizontal hypersurfaces}
\label{Hypersurface in a projective bundle over a projective curve: notation and definitions}

Let $C$ be a connected smooth projective complex curve with generic point $\eta$, let $E$ be a vector bundle of rank $N+ 1\geq 1$ over $C$, and let:
$$
\pi: \PP(E) := \Proj S^\bullet E^\vee \lra C,
$$
be the associated projective bundle over $C$.

We denote by $\cO_E(1)$ the tautological quotient line bundle over $\PP(E)$; it is the dual of the tautological subbundle of rank $1$ of $\pi^\ast E$.

We shall use the fact (see for instance \cite[B.5.8]{Fulton98}) that the relative tangent bundle $$T_{\PP(E)/C}:= T_\pi$$ fits into a short exact sequence:
\begin{equation}
0 \lra  \cO_{\PP(E)} \lra \pi^\ast E \otimes \cO_E(1) \lra T_{\PP(E)/C} \lra 0. \label{tangentbundleexactsequence}
\end{equation}

Let $H$ be an hypersurface, namely an effective Cartier divisor, in $\PP(E)$. It is a projective complex scheme of dimension $N$. Its generic fiber $H_\eta$ is an hypersurface in the projective space $\PP(E)_\eta$ of dimension~$N$ over $\C(C)$. We shall denote its degree by $d$.

We shall also denote by:
$$
i : H \lra \PP(E)
$$ 
the inclusion morphism of  $H$ in the projective bundle $\PP(E)$, and by: 
$$
f:=\pi \circ i : H \lra C
$$
its projection to the curve $C$.
 
As is well-known, the hypersurface $H_\eta$ in the projective space $\PP(E)_\eta$ may be defined, as a subscheme, by the vanishing of an homogeneous polynomial of degree $d$ --- namely by a non-zero element of $S^d E^\vee_\eta$ --- unique up to the multiplication by an element of $\C(C)^\ast.$
 
This description of $H_\eta$ extends to a description of the  ``relative hypersurface'' $H$ over $C$ as follows.
  
Since $\PP(E)$ is a regular scheme, its Picard group is isomorphic to its Chow group of codimension one $\CH^1(\PP(E))$. According to the structure of the Chow groups of a projective bundle (see \cite[Th. 3.3 (b)]{Fulton98}, with $k = \dim(\PP(E)) - 1$), there is an isomorphism:
$$\Z \oplus \CH^1(C) \overset{\sim}{\lra} \CH^1(\PP(E)), \quad (\delta, \alpha) \longmapsto \delta c_1(\cO_E(1)) + \pi ^\ast \alpha.$$
Consequently the line bundle $\cO_{\PP(E)}(H)$ is isomorphic to $\cO_E(\delta) \otimes \pi^\ast M$ for some integer $\delta$ and some line bundle $M$ over $C$. By considering the restriction of this isomorphism to $\PP(E)_\eta$, we see that $\delta$ coincides with the generic degree $d$ of $H$. Moreover the line bundle $M$ is unique up to isomorphism. 

We will denote by:
\begin{equation}\label{defsigma}
\sigma: \cO_{\PP(E)}(H) \lrasim \cO_E(d) \otimes \pi^\ast M
\end{equation}
the isomorphism obtained by this construction. This isomorphism may be seen as a morphism of line bundles over $\PP(E)$:
\begin{equation}\label{defsigmabis}
\sigma: \pi^\ast M^\vee \lra \cO_E(d),
\end{equation}
the vanishing of which defines $H$ scheme-theoretically. In turn, the data of the morphism \eqref{defsigmabis} is equivalent to the data of a morphism of sheaves of $\cO_C$-modules:
\begin{equation}\label{deftau}
\tau : M^\vee \lra \pi_\ast \cO_{E}(d) \simeq S^d E^\vee,
\end{equation}
which is clearly injective.

For every point $x$ of $C$, the image $\tau_x(y)$ of any non-zero element $y$ of $M^\vee_x$ is an element of $S^d E^\vee_{x}$, which is a (scheme-theoretic) equation for the fiber $H_x$ of $H$ in the projective space $\PP(E)_x$ over $\C$.
This shows that the hypersurface $H$ is horizontal --- namely, that no fiber of $\pi$ is a component of $H$, or equivalently that the morphism $f : H \ra C$ is flat --- if and only if $\tau_x$ is non-zero for every point $x$ of $C$, or equivalently if and only if $\tau$ is an isomorphism onto a locally direct summand of $S^d E^\vee$.

In the sequel, we shall use the following notation for the first Chern classes of $E$ and $M$,  and of~$\cO_E(1)$, in $\CH^1(C)$ and $\CH^1(\PP(E))$ respectively:
$$e := c_1(E), \quad m := c_1(M),  \quad \mbox{and} \quad h := c_1(\cO_E(1)).$$ 
Using the isomorphism \eqref{defsigma}, we obtain the equality in $\CH^1(\PP(E))$:
\begin{equation}\label{H as a cycle}
[H] = d \,h + \pi^\ast m.
\end{equation}

Using  the definition of Segre classes and their relation to Chern classes (\cite[3.1, 3.2]{Fulton98}), and the fact that $C$ is one-dimensional, we obtain the following equality, for every $i \in \N$:
\begin{equation}\label{Segre class of E}
\pi_\ast h^i =: s_{i-N}(E)=
\begin{cases}
 0  &\quad \mbox{if } i < N, \\ 
  [C] &\quad \mbox{if } i = N, \\ 
  - e &\quad \mbox{if } i = N+1, \\
 0 &\quad  \mbox{if } i> N+1. 
\end{cases}
\end{equation}

From the relations above and the projection formula, we deduce:
\begin{equation}\label{HhN-1}
\pi_\ast (h^{N-1} \cap [H]) = d [C],
\end{equation}
and 
\begin{equation}\label{HhN}
\pi_\ast (h^{N} \cap [H]) = -d e + m.
\end{equation}

In computations, we will  use that $\CH^i(C)$ vanishes for $i >1$; in particular, we have:
\begin{equation}
m \cdot m = e \cdot e = m \cdot e =0. \label{product classes on C}
\end{equation}

\begin{remark} 
\label{invariance_produittensoriel} 
The classes $m$ and $e$ depend on the choice of the bundle $E$, and not only on the $C$-scheme~$\PP{} (E)$.  

More specifically, if $L$ is a line bundle on $C$, and $E$ is replaced by $E' := E \otimes L$, then we may identify the $C$-schemes $\PP(E')$ and $\PP(E)$, but $e$ is replaced by:
$$
e' := c_1(E') = c_1(E \otimes L) = e + (N+1) \textrm{ } c_1(L).
$$
The line bundle $\cO_E(1)$ is replaced by:
$$
\cO_{E'}(1) \simeq \cO_E(1) \otimes \pi^\ast L^\vee,
$$
and the class $h$ is replaced by:
$$
h' := c_1(\cO_{E'}(1)) = h - \pi^\ast c_1(L).
$$
Moreover, the bundle $S^d E'^\vee$ may be identified with $(S^d E^\vee) \otimes L^{-d}$. The morphism of vector bundles: 
$$
\tau' := \tau \otimes \Id_{L^{-d}} : M^\vee \otimes L^{-d} \hookrightarrow S^d E'^\vee
$$ 
clearly represents the same equation as $\tau$, hence it also defines the hypersurface $H$ in $\PP(E')$. 

Finally the line bundle $M$ is replaced by $M' := M \otimes L^d$, and consequently the class $m$ is replaced by:
$$
m ' := c_1(M') = m + d \textrm{ }c_1(L).
$$

Observe that this implies the equality of classes in $\CH^1(C)_\Q$:
$$ (N+1) m' - de' = (N+1) m - d e.$$ 
\end{remark}

We shall denote the relative canonical sheaf of the smooth morphism 
$$\pi : \PP(E) \lra C$$
by:
\begin{equation}\label{def canonical sheaf}
\omega_{\PP(E)/C} \simeq \Omega^N_{\PP(E)/C} \simeq \det T_{\PP(E)/C} ^\vee.
\end{equation}

\begin{defpropositionEng} 
    \label{ht_int}
Let $h^{\ast} \in \CH^1(\PP(E))_\Q$ be the class defined by:
$$h^{\ast} := - \frac{1}{N+1} c_1(\omega_{\PP(E)/C}),$$
and let $\hgt_{int}(H/C)$ be the rational number defined by:
$$\hgt_{int}(H/C) := \int_{\PP(E)} h^{\ast N} \cap [H].$$
They satisfy the following relations:
\begin{equation}\label{h*}
h^\ast = h +  \frac{1}{N+1} \pi^{\ast} e,
\end{equation}
\begin{align}
   \hgt_{int}(H/C) & = \int_{\PP(E)} h ^N \cap [H] + \frac{dN}{N+1} \deg E, \label{htinteth} \\
& = \deg M - \frac{d}{N+1} \deg E . \label{htintsurC}
\end{align}
In particular, the rational number $(N+1) \hgt_{int}(H/C)$ is an integer.
\end{defpropositionEng}  

\begin{proof}
 The equality \eqref{h*} comes from taking the determinant of the exact sequence (\ref{tangentbundleexactsequence}). This gives isomorphisms:
$$\det T _{\PP(E)/C}   \simeq  \det\big(\pi^{\ast} E \otimes \cO_E(1)\big) \simeq \cO_E(N+1) \otimes \pi^{\ast} \det E.$$
Together with \eqref{def canonical sheaf}, this establishes \eqref{h*}. 

By using successively \eqref{product classes on C}, \eqref{H as a cycle} and \eqref{product classes on C} again, we have:
\begin{align}
h^{\ast  N} \cap [H] & = \left(h + \frac{1}{N+1} \pi^{\ast} e\right)^N  \cap [H], \notag \\
& = \left(h^N + \frac{N}{N+1} h^{N-1} \pi^{\ast} e \right)  \cap [H], \notag \\
& = h^N  \cap [H] + \frac{N}{N+1} h^{N-1} \pi^{\ast} e \textrm{ } . \textrm{ } (d h + \pi^{\ast} m), \notag\\
& = h^N  \cap [H] + \frac{dN}{N+1} h^N \pi^{\ast} e. \label{h*^N cap H}
\end{align}
After pushing forward by $\pi$ and taking the degree, using \eqref{Segre class of E}, this becomes equality \eqref{htinteth}. 

From \eqref{h*^N cap H} and \eqref{H as a cycle}, we obtain:
\begin{align*}
 h^{\ast N} \cap [H] & =  h^N . \textrm{ }(d h + \pi^{\ast} m) + \frac{dN}{N+1} h^N \pi^{\ast} e, \\
& = d h^{N+1} + h^N \pi^{\ast} m + \frac{dN}{N+1} h^N \pi^{\ast} e, \\
\end{align*}
so pushing it forward by $\pi$ and using \eqref{Segre class of E}, we get:
$$\pi_{\ast} (h^{\ast N} \cap [H]) = - d e + m + \frac{dN}{N+1} e = m - \frac{d}{N+1} e.$$
Consequently, we have:
$$\int_{\PP(E)} h^{\ast N} \cap [H] = \deg (\pi_{\ast} (h^{\ast N} \cap [H]) ) = \deg M  - \frac{d}{N+1} \deg E.$$ 
This shows equality \eqref{htintsurC}.
\end{proof}

\begin{remark}\label{invariance htint}

 The subscript stands for ``intersection-theoretic." The height $\hgt_{int}(H/C)$ clearly only depends on $\PP(E)$ as a scheme over $C$ and on its subscheme $H$. We can also check this property on the expression \eqref{htintsurC}: according to Remark \ref{invariance_produittensoriel}, 
 if $E$ is replaced by $E' := E \otimes L$, with $L$ a line bundle, the class $m - \frac{d}{N+1} e$ is unchanged. 
 \end{remark}
We can also express this height in terms of slopes. The slope of the line bundle $M$ is  $\mu(M) = \deg M$, the slope of $E$ is: $$\mu(E) = \frac{\deg E }{N+1},$$ and the one of $S^d E^\vee$ is: $$\mu(S^d E^\vee) = - d \, \mu(E) = - \frac{d \deg E }{N+1},$$ which gives the following expression:
\begin{equation}\label{htintslope}
\hgt_{int}(H/C) = \mu(S^d E^{\vee}) - \mu(M^\vee).
\end{equation}  

The existence of an effective section of the line bundle $M^{\otimes (N+1) (d-1)^N} \otimes (\det E)^{\otimes -d (d-1)^N}$ on $C$ defined by the discriminant of the horizontal hypersurface $H$ (see \eqref{discriminant section H} below), along with equality \eqref{htintsurC}, will prove that the rational number $\hgt_{int}(H/C)$ is non-negative when $d \geq 2$.

When the vector bundle $S^d E^\vee$ is semistable, the non-negativity of $\hgt_{int}(H/C)$ would follow from equation \eqref{htintslope}. Indeed, since $M^\vee$ is naturally embedded as a coherent subsheaf of $S^d E^\vee$, using the expression \eqref{htintslope} for $\hgt_{int}(H/C)$, if the vector bundle $S^d E^\vee$ is stable (resp. semistable), then $\hgt_{int}(H/C)$ is positive (resp. non-negative). 

Combining this result with the compatibility of vector bundle stability with tensor operations in characteristic 0 (see for instance \cite{Narasimhan-Seshadri65}), we immediately obtain that if the vector bundle $E$ is semistable, then $\hgt_{int}(H/C)$ is non-negative. 

The non-negativity of $\hgt_{int}(H/C)$ when the vector bundle $E$ is semistable would also follow from \cite[Th.~3.1]{Miyaoka87}: the vector bundle $E$ is semistable if and only if the class $h^\ast$ (which Miyaoka denotes by $\lambda_{\mathcal{E}}$) is nef, and when this holds,   the rational number 
$$\hgt_{int}(H/C) := \int_{\PP(E)} h^{\ast N} \cap [H]$$
is non-negative.

\subsection{The Griffiths height of the middle-dimensional cohomology of a pencil of hypersurfaces in projective spaces}

The objective of this subsection is to show the following theorem, which immediately implies Theorem \ref{intro GK hyp P(E)} using equality \eqref{htintsurC}.
As observed after the statement of Theorem \ref{intro GK hyp P(E)}, this theorem implies the validity of the same formulas for the Griffiths heights of the primitive part of the middle-dimensional cohomology~$\H^{N-1}(H_\eta/C_\eta)$. 

\begin{theorem}
\label{GK hyp P(E)} 
Let $C$ be a connected smooth projective complex  curve with generic point $\eta$, $E$ a vector bundle of rank~$N+1$ over $C$, and $H \subset \PP(E) $ an horizontal hypersurface of relative degree $d$, smooth over $\C.$  

If $\pi_{\mid H}$ has only a finite number of critical points, all of which are non-degenerate, then the reduced $0$-cycle in $H$ defined by the set of critical points $\Sigma$ satisfies the equality in $\CH_0(C)$:
\begin{equation} \label{f(sigma) P(E)}
f_\ast [\Sigma] = (d-1)^N ((N+1) m - d e).
\end{equation}

Moreover, we have the following equalities in $\CH_0(C)_\Q$:
\begin{equation} \label{pas intro GK+XL P(E)}
c_1(\GK_{C,+}( \H^{N-1}(H_\eta/C_\eta))) = F_+(d,N) (m - \frac{d}{N+1} e),
\end{equation}
and:
\begin{equation} \label{pas intro GK-XL P(E)}
c_1(\GK_{C,-}( \H^{N-1}(H_\eta/C_\eta))) = F_-(d,N) (m - \frac{d}{N+1} e),
\end{equation}
and the following equality of rational numbers:
\begin{equation} \label{pas intro GKstab P(E)}
\mathrm{ht}_{GK,stab}( \H^{N-1}(H_\eta/C_\eta)) = F_{stab}(d,N) \, \mathrm{ht}_{int}(H/C),
\end{equation}
where $F_+(d,N)$, $F_-(d,N)$ and $F_{stab}(d,N)$ are the elements of $(1/12) \Z$ given when  $N$ is odd by:
$$
F_+(d,N) := \frac{N+1}{24 d^2} \left[ (d-1)^N  (7 d^2 N - 7 d^2 - 2 d N - 2 ) + 2 (d^2 - 1) \right] ,
$$
$$F_-(d,N) := \frac{N+1}{24 d^2} \left[ (d-1)^N (- 5 d^2 N + 5 d^2 - 2 d N - 2 ) + 2 (d^2 - 1) \right] ,$$
and:
$$
F_{stab}(d,N) := \frac{N+1}{24 d^2} \left[ (d-1)^N (d^2 N - d^2 - 2 d N - 2 ) + 2 (d^2-1) \right] ,
$$
and, when $N$ is even, by:
$$
F_+(d,N) = F_-(d,N) = F_{stab}(d,N) := \frac{N+1}{24 d^2} \left [ (d-1)^N  (d^2 N + 2 d^2 - 2 d N - 2) - 2 (d^2-1) \right ]. 
$$
\end{theorem}

Observe that the morphism of complex schemes:
$$
\pi : X := \PP(E) \lra C
$$
is smooth and surjective, and using the isomorphism \ref{defsigma}, the line bundle $L$ is relatively ample. Consequently, we are in the situation of Section \ref{Ample hyp}. 
\begin{proposition}
\label{cycles Proj E} 
With the above notation, and denoting $L$ the line bundle $\cO_{\PP(E)}(H)$ on $\PP(E)$, we have the equalities in $\CH_0(\PP(E))$:
\begin{equation}
i_\ast [\Sigma] = (d-1)^N h^N [(d-1) h + (N+1) \pi^\ast m - \pi^\ast e], \label{sigma in PE}
\end{equation}
\begin{equation}
c_1(L) c_N(\Omega^1_{\PP(E)/C}) = (-1)^N h^N [d (N+1) h + (N+1) \pi^\ast m + d N \pi^\ast e], \label{c1 cN Omega in PE}
\end{equation}
and:
\begin{equation}
    [(1 - c_1(L))^{-1} c_1(\Omega^1_{\PP(E)/C}) c(\Omega^1_{\PP(E)/C}) ]^{(N+1)} = h^N (a_{N,d} h + b_{N,d} \pi^\ast m + c_{N,d} \pi^\ast e), \label{quotient in PE}
\end{equation}
with the rational numbers $a_{N,d}$, $b_{N,d}$ and $c_{N,d}$ given by:
$$a_{N,d} := \frac{N+1}{d} \big (- (d-1)^{N+1} + (-1)^{N+1} \big ),$$
$$b_{N,d} := \frac{N+1}{d^2} \big (- (d-1)^N (d N + 1) + (-1)^N \big ),$$
$$c_{N,d} := \frac{1}{d}  \big ( - (d-1)^N (d-N-2) + (-1)^{N+1} (N+2) \big ).  $$
\end{proposition}
\begin{proof}
Applying equality \eqref{eq Sigma for hyp} from Proposition \ref{Sigma for hyp} with the above notation yields the equality in $\CH_0(\PP(E))$:
\begin{align}
i_\ast [\Sigma] &= \big [(1 - c_1(L))^{-1} c(\Omega^1_{\PP(E)/C} ) \big ]^{(N+1)} \nonumber \\
&= (-1)^{N+1} \big [(1 + c_1(L))^{-1} c(T_{\PP(E)/C}) \big ]^{(N+1)} \nonumber \\
&= (-1)^{N+1} \big [(1 + d h + \pi^\ast m)^{-1} c(\pi^\ast E \otimes \cO(1)) \big ]^{(N+1)} \label{first eq sigma in PE} \\
&= (-1)^{N+1} \Big [ \big ((1 + d h)^{-1} - \pi^\ast m \, (1 + d h)^{-2} \big ) \big ((1 + h)^{N+1} + \pi^\ast e\,  (1 + h)^N \big ) \Big ]^{(N+1)} \label{second eq sigma in PE} \\
&= (-1)^{N+1} [(1 + d h)^{-1} (1 + h)^{N+1} ]^{(N+1)} 
-(-1)^{N+1} \pi^\ast m \, [(1 + d h)^{-2} (1 + h)^{N+1} ]^{(N)} \nonumber \\
&+ (-1)^{N+1} \pi^\ast e \,  [(1 + d h)^{-1} (1 + h)^N ]^{(N)}  \label{third eq sigma in PE} \\
&= (-1)^{N+1} [ (-1)^{N+1} (d-1)^{N+1}] h^{N+1}
- (-1)^{N+1} \pi^\ast m \, [(-1)^N (N+1) (d-1)^N] h^N \nonumber \\
 &+ (-1)^{N+1} \pi^\ast e  \, [(-1)^N (d-1)^N] h^N \label{fourth eq sigma in PE} \\
&= (d-1)^N h^N [(d-1) h + (N+1) \pi^\ast m - \pi^\ast e], \nonumber
\end{align}
where in \eqref{first eq sigma in PE}, we used exact sequence \eqref{tangentbundleexactsequence} and equality \eqref{H as a cycle}; in \eqref{second eq sigma in PE}, we used \eqref{product classes on C} and \cite[Ex. 3.2.2]{Fulton98} ; in \eqref{third eq sigma in PE}, we used again \eqref{product classes on C}; and in \eqref{fourth eq sigma in PE}, we used Proposition \ref{formal identities}. 

This shows \eqref{sigma in PE}. The other equalities are proved similarly with straightforward computations.
\end{proof}

\begin{proof}[Proof of Theorem \ref{GK hyp P(E)}]

Pushing equality \eqref{sigma in PE} forward by the morphism $\pi$ yields the equality in $\CH^1(C)$:
\begin{align*}
f_\ast [\Sigma] &= (d-1)^N [(d-1) \pi_\ast h^{N+1} + (N+1) \pi_\ast (h^N \pi^\ast m) - \pi_\ast (h^N \pi^\ast e) ], \\
&= (d-1)^N [- (d-1) e + (N+1) m - e] \textrm{ using the projection formula and \eqref{Segre class of E},} \\
&= (d-1)^N [(N+1) m - d e],
\end{align*}
which shows equality \eqref{f(sigma) P(E)}.

Let us show equality \eqref{pas intro GK+XL P(E)}. Since $X= \PP(E)$ is a projective bundle over $C$, all its relative Hodge vector bundles on $C$ are trivial, and for every integer $n$, the Griffiths line bundle $\GK_{C}(\H^n(X / C))$ is trivial. Since the line bundle $L$ is ample relatively to the morphism $\pi$, we may apply the equality \eqref{pas intro GK+XL}  established in Proposition \ref{technical result}. With the above notation, this 
yields the equality in $\CH_0(C)_\Q$:
\begin{multline}
c_1(\GK_{C,+}(\H^{N-1}(H_\eta / C_\eta))) = \frac{1}{12} \pi_\ast \Big [ \big ((1 - c_1(L))^{-1} c_1(\Omega^1_{\PP(E)/C}) c(\Omega^1_{\PP(E)/C}) \big )^{(N+1)} \Big ] \\
- \frac{1}{12} \pi_\ast \big (c_1(L) c_N(\Omega^1_{\PP(E)/C}) \big )
 + v_N^+ f_\ast [\Sigma]. \label{first eq GK pencil P(E)}
\end{multline}
Using \eqref{quotient in PE}, the first term in \eqref{first eq GK pencil P(E)} is given by:
\begin{align}
&\frac{1}{12} \pi_\ast \Big [ \big ((1 - c_1(L))^{-1} c_1(\Omega^1_{\PP(E)/C}) c(\Omega^1_{\PP(E)/C}) \big )^{(N+1)} \Big ] \nonumber \\
&= \frac{1}{12} \pi_\ast ( a_{N,d} h^{N+1} + b_{N,d} h^N \pi^\ast m + c_{N,d} h^N \pi^\ast e), \nonumber\\
&= \frac{1}{12} b_{N,d} m + \frac{1}{12} (c_{N,d} - a_{N,d}) e,\label{second eq GK pencil P(E)} \\
&= \frac{N+1}{12 d^2} \big [ - (d-1)^N (d N + 1) + (-1)^N \big ] m \nonumber \\
&+ \frac{1}{12 d} \big [- (d-1)^N (d-N-2) + (-1)^{N+1} (N+2) + (N+1) (d-1)^{N+1} - (N+1) (-1)^{N+1} \big ] e,\nonumber\\
&= \frac{N+1}{12 d^2} \big [ - (d-1)^N (d N + 1) + (-1)^N \big ] m - \frac{1}{12 d} \big [ - (d-1)^N (d N + 1) + (-1)^N \big ] e,\nonumber\\
&= \frac{N+1}{12 d^2} \big ( - (d-1)^N (d N + 1) + (-1)^N \big ) \big (m - \frac{d}{N+1} e \big ), \label{third eq GK pencil P(E)}
\end{align}
where in \eqref{second eq GK pencil P(E)}, we have used the projection formula and \eqref{Segre class of E}.

Reasoning similarly using \eqref{c1 cN Omega in PE}, the second term in \eqref{first eq GK pencil P(E)} is given by:
\begin{align}
\frac{1}{12} \pi_\ast \big (c_1(L) c_N(\Omega^1_{\PP(E)/C}) \big )
&= \frac{(-1)^N}{12} \pi_\ast (d (N+1) h^{N+1} + (N+1) h^N \pi^\ast m + d N h^N \pi^\ast e),\nonumber\\
&= \frac{(-1)^N}{12} (- d (N+1) e + (N+1) m + d N e),\nonumber\\
&= \frac{(-1)^N (N+1)}{12} \big  (m - \frac{d}{N+1} e \big ). \label{fourth eq GK pencil P(E)}
\end{align}
Replacing \eqref{f(sigma) P(E)}, \eqref{third eq GK pencil P(E)} and \eqref{fourth eq GK pencil P(E)} in \eqref{first eq GK pencil P(E)} yields the equality:
$$
c_1(\GK_{C,+}(\H^{N-1}(H_\eta / C_\eta))) = F_+(d, N) \big (m - \frac{d}{N+1} e \big ),
$$
where $F_+(d,N)$ is the rational number given by:
$$
F_+(d,N) = \frac{N+1}{12 d^2} \big ( - (d-1)^N ( d N + 1) + (-1)^N \big ) - \frac{(-1)^N (N+1)}{12} + v_N^+ (N+1) (d-1)^N.
$$
If the integer $N$ is odd, $F_+(d,N)$ is given by:
\begin{align*}
F_+(d,N) &= - \frac{N+1}{12 d^2} \big ((d-1)^N (d N + 1) + 1 \big ) + \frac{N+1}{12} + \frac{7 (N-1)}{24} (N+1) (d-1)^N, \\
&= \frac{N+1}{24 d^2} \Big [ - 2 \big ((d-1)^N (d N + 1) + 1 \big ) + 2 d^2 + 7 d^2 (N-1) (d-1)^N \Big ],\\
&= \frac{N+1}{24 d^2} \Big [ (d-1)^N (7 d^2 N - 7 d^2 - 2 d N - 2) + 2 (d^2 - 1) \Big ],
\end{align*}
as wanted, and if the integer $N$ is even, it is given by:
\begin{align*}
F_+(d,N) &= \frac{N+1}{12 d^2} \big (- (d-1)^N (d N + 1) + 1 \big ) - \frac{N+1}{12} + \frac{N+2}{24} (N+1) (d-1)^N,\\
&= \frac{N+1}{24 d^2} \Big [2 \big (- (d-1)^N (d N + 1) + 1 \big ) - 2 d^2 + d^2 (N+2) (d-1)^N \Big ],\\
&= \frac{N+1}{24 d^2} \Big [ (d-1)^N (d^2 N + 2 d^2 - 2 d N - 2) - 2 (d^2 - 1) \Big ],
\end{align*}
as wanted. This shows equality \eqref{pas intro GK+XL P(E)}.

For equality \eqref{pas intro GK-XL P(E)}, we reason similarly using \eqref{pas intro GK-XL}.

For equality \eqref{pas intro GKstab P(E)}, it follows from equality \eqref{f(sigma) P(E)} and the definition of the height $\hgt_{int}(H/C)$ that we have the following equality of integers:
$$
\vert \Sigma \vert = (d-1)^N [(N+1) \deg(m) - d \deg(e)] = (N+1) (d-1)^N \hgt_{int}(H/C).
$$
Equality \eqref{pas intro GKstab P(E)} is a simple consequence of this equality combined with equality \eqref{ht stab non-degenerate 1} from Corollary \ref{GK bundles with elementary exponents for non-degenerate singularites} and with equality \eqref{pas intro GK-XL P(E)}.
\end{proof}

\subsection{Complements: horizontal hypersurfaces in a projective bundle and discriminant}
\label{Disc hyp P(E)}

Equality \eqref{f(sigma) P(E)} from Theorem \ref{GK hyp P(E)} can also be deduced from the classical theory of the discriminant, which we shall now recall.

Recall the definition (\cite{Demazure12}) of the discriminant of a homogeneous polynomial in an arbitrary number of indeterminates.

Let us denote by $n \geq 1$  the number of indeterminates and by $d$  the degree of the homogeneous polynomial. Let:
$$I_{n,d} := \left\{\alpha = (\alpha_1,...,\alpha_n) \in \N ^n | \sum_{i=1}^n \alpha_i = d \right\}$$
be the set of indexes, let:
$$U_{n,d} := \Z \left[(T_\alpha)_{\alpha \in I_{n,d}} \right] $$
be the universal polynomial ring, and let:
$$P_{n,d} := \sum \limits_{\alpha \in I_{n,d}} X_1^{\alpha_1} ... X_n^{\alpha_n}  \, T_\alpha \in U_{n,d} [X_1,...,X_n]$$
be the \emph{universal homogeneous polynomial} of degree $d$ in $n$ indeterminates. 
It satisfies the following tautological property. 

For $R$ a ring, let us denote by $R[X_1,...,X_n]_d$ the submodule of $R[X_1,...,X_n]$ of homogeneous polynomials of degree $d$. For every polynomial $F \in R[X_1,...,X_n]_d$, there is a unique morphism of rings:
$$h_F : U_{n,d} \lra R,$$
such that the polynomial $h_F(P_{n,d})$ deduced from $P_{n,d}$ by replacing its coefficients in $U_{n,d}$ by their images by $h_F$ --- in other terms, its base change  by $h_F$ --- is $F$ itself.

\begin{defpropositionEng}[{\cite[5]{Demazure12}; see also \cite[Chapter 13]{GKZ08}}]
\label{discriminant}

The \emph{universal discriminant} $\disc_{n,d}(P_{n,d})$ is an element of $U_{n,d}$ given by:
$$\disc_{n,d}(P_{n,d}) := d^{-a(n,d)} \Res \left(\frac{\partial P_{n,d}}{\partial X_1}, \cdots, \frac{\partial P_{n,d}}{\partial_{X_n}} \right),$$
where $\Res$ denotes the resultant of $n$ polynomials in $n$ indeterminates, and the integer~$a(n,d)$ is defined by:
$$a(n,d) := \frac{(d-1)^n - (-1)^n}{d}.$$

For every ring $R$, if $F$ is a polynomial in $R[X_1,...,X_n]_d$, if $h_F : U_{n,d} \ra R$ denotes the morphism of rings such that  $h_F(P_{n,d})=F$,  then: $$\disc_{n,d}(F) := h_F(\disc_{n,d}(P_{n,d})) \in R$$ defines the \emph{discriminant} of F.

The universal discriminant $\disc_{n,d}(P_{n,d})$ is characterized (up to a sign) by the following two properties:
\begin{enumerate}
\item It is a prime element of $U_{n,d}$.\footnote{see \cite[6, Cor. 1]{Demazure12}.} \label{discriminant is prime}
\item  If $k$ is a field, and $F$ is a polynomial in $k[X_1,..., X_n]_d$, then $\disc_{n,d}(F) = 0$ if and only if the subscheme $(F = 0)$ in $\PP^{n-1} _k$ is not smooth.\footnote{see \cite[Prop. 12]{Demazure12}.} 
\label{discriminant and smooth subschemes}
\end{enumerate}

Moreover, for every ring $R$ and for every homogeneous polynomial $F$ in $R[X_1,...,X_n]_d$, 
the following equalities hold in $U_{n,d}:$\footnote{see  \cite[ Prop. 11, c) and e)]{Demazure12}.}
\begin{enumerate} [resume]
\item If $\gamma$ is a scalar in $R$, then we have:
\begin{equation}\disc_{n,d}(\gamma F) = \gamma^{n (d-1)^{n-1}} \disc_{n,d}(F). \label{homogeneity discriminant} \end{equation} 
\item  If $A = (a_{i,j})_{1 \leq i,j \leq n}$ is a matrix with entries in $R$, then the discriminant of the polynomial $F' := F \left(\sum \limits_j a_{1,j} X_j, \sum \limits_j a_{2,j} X_j,... \right)$ is given by:
\begin{equation} \disc_{n,d}(F') = (\det A)^{d (d-1)^{n-1}} \disc_{n,d}(F). \label{discriminant variable change} \end{equation} 
\end{enumerate}

\end{defpropositionEng}

We can extend this definition to families of polynomials parametrised by a scheme in the following way.

\begin{defpropositionEng} \label{family discriminant}
Let $X$ be a scheme, $M$ be a line bundle over $X$, $E$ be a vector bundle of rank $n$ on $X$, and let $\tau$ be a section over $X$ of the vector bundle $S^d E^\vee \otimes M$.

There exists a unique  section $\disc_{n,d}(\tau)$ over $X$ of the line bundle: $$M^{ \otimes n (d-1)^{n-1}} \otimes (\det E)^{ \otimes - d (d-1)^{n-1}}$$ satisfying the following property:
\begin{enumerate}
    \item If $V = \Spec R$ is an affine open subset of $X$, if
    $$\phi : \cO_V \overset{\sim}{\lra} M_V, \quad \psi : \cO_V ^{\oplus n} \overset{\sim}{\lra} E_V$$
    are local trivializations of the line bundle $M$ and the vector bundle $E$, then the scalar
    $$\left ((\phi^{-1})^{ \otimes n (d-1)^{n-1}} \otimes (\det \psi^{-1}) ^{ \otimes -d (d-1)^{n-1}} \right ) (\disc_{n,d}(\tau)) \in \Gamma(V, \cO_V) \simeq R,$$
    is precisely the discriminant $\disc_{n,d}(F)$ of the polynomial
    $$F := \left ( (S^d \ {}^t \psi) \otimes \phi^{-1} \right ) (\tau_{| V}) \in \Gamma(V, S^d \cO_V ^{\oplus n}) \simeq R[X_1,...,X_n]_d.$$
\end{enumerate}
This section also satisfies the following property:
\begin{enumerate}[resume]
\item If $x$ is a  point in $X$ with residue field $\kappa(x)$, then the section $\disc_{n,d}(\tau)$ vanishes at $x$ if and only if the hypersurface in $\PP(E_{\kappa(x)})$ defined by the vanishing of the polynomial $\tau(x) \in S^d E_{\kappa(x)}^\vee \otimes M_{\kappa(x)}$ is not smooth over $\kappa(x)$. \label{family discriminant and smooth subschemes}
\end{enumerate}
\end{defpropositionEng}

The proof is straightforward.

Consider now   a connected smooth projective complex curve $C$ with generic point $\eta$, a vector bundle $E$ of rank $N+1$ over $C,$ and  a horizontal hypersurface $H \subset \PP(E)$ such that $H_\eta$ is smooth over $\eta$, and let us adopt the notation of the previous subsections. 

In particular, we denote by  $\tau$ a section of $S^d E^\vee \otimes M$ representing the ``equation" of $H$.

By applying to $\tau$ the construction in the previous proposition, we get a section:
\begin{equation} \label{discriminant section H}
\disc_{N+1,d}(\tau) \in \Gamma \left(C, M^{\otimes (N+1) (d-1)^N} \otimes (\det E)^{\otimes -d (d-1)^N} \right).
\end{equation}
The property \ref{family discriminant and smooth subschemes} of Definition-Proposition \ref{family discriminant} implies that for every scheme point~$x$ in $C$, the section $\disc_{N+1,d}(\tau)$ vanishes  at $x$ if and only if the hypersurface $H_{\kappa(x)} \subset \PP(E_{\kappa(x)})$ is singular. In particular, $ \disc_{N+1,d}(\tau)$ is not the zero section of~$M^{\otimes (N+1) (d-1)^N} \otimes (\det E)^{\otimes -d (d-1)^N}$.

Let $\Delta$ be the finite subset of critical values in $C$ of the morphism
$$f : H \lra C.$$
Thanks to the work of Eriksson (\cite[Th. 1.2]{Eriksson16}), one may express the order of vanishing of the section $\disc_{N+1,d}(\tau)$ at some point $x$ in $\Delta$ in terms of the localized $N$-th Chern class of the coherent sheaf $\Omega^1_{H/C}$, which is supported on the singular fibers of $H$.

In the special case where $H$ is smooth and the morphism $f$ has only non-degenerate critical points, his result easily implies the equality, for every $x$ in $\Delta$:
$$\mathrm{ord}_x \big(\disc_{N+1,d}(\tau)\big) = \vert \Sigma_x \vert,$$
where $\Sigma$ denotes the set of critical points of $f$ and:
$$\Sigma_x := \Sigma \cap f^{-1}(x).$$
This implies the equality in $\CH_0(C)$:
\begin{align*}
f_\ast [\Sigma] := \sum_{x \in \Delta} \vert \Sigma_x \vert\,  x &= \sum_{x \in \Delta} \mathrm{ord}_x (\disc_{N+1,d}(\tau))\, x, \\
&= \mathrm{div}\big(\disc_{N+1,d}(\tau)\big), \\
&= c_1\big(M^{\otimes (N+1) (d-1)^N} \otimes (\det E)^{\otimes -d (d-1)^N}\big), \\
&= (d-1)^N \big( (N+1) m - d e\big),
\end{align*}
and provides an alternative proof of \eqref{f(sigma) P(E)}.

\section{Linear pencils of hypersurfaces and Lefschetz pencils}

Another setting where   Proposition \ref{technical result} applies is provided by linear pencils of hypersurfaces.

Let $V$ be a connected smooth projective complex scheme of pure dimension $N\geq 1$, and let $H$ be a non-singular hypersurface in $V \times \PP^1$ such that the morphism 
$$
\mathrm{pr}_{1 | H} : H \lra V
$$
is dominant, or equivalently surjective, and let $\delta$ be its degree.
The line bundle $\cO_{V\times \PP^1}(H)$ is isomorphic to the line bundle $$
\mathrm{pr}_1^\ast M \otimes \mathrm{pr}_2^\ast \cO_{\PP^1}(\delta),
$$
where $M$ denotes some line bundle over $V$, which is unique up to isomorphism.

\begin{proposition}
\label{linear pencils of hypersurfaces}
With the above notation, 
let us assume that the morphism
$$
\mathrm{pr}_{2 | H} : H \lra \PP^1
$$
is surjective, and has a finite set $\Sigma$ of critical points, all of which are non-degenerate.

The following equality holds in $\CH_0(V \times \PP^1)$:
\begin{equation} \label{sigma linear pencil}
[\Sigma] = \delta \, \mathrm{pr}_1^\ast [(1 - c_1(M))^{-2} c(\Omega^1_V) ]^{(N)} \, \mathrm{pr}_2^\ast c_1(\cO_{\PP^1}(1)).
\end{equation}
In particular, the cardinality of $\Sigma$ satisfies:
$$
\vert \Sigma \vert = \delta \int_V (1 - c_1(M))^{-2} c(\Omega^1_V).
$$

Furthermore, if the line bundle $M$ on $V$ is ample, the following equalities of integers hold:
\begin{multline} \label{pas intro GK+XL lineaire}
\deg_{\PP^1}(\GK_{\PP^1,+}(\H^{N-1}(H_\eta / \PP^1_\eta) ))\\
= \frac{\delta}{12} \int_V (1 - c_1(M))^{-2} c_1(\Omega^1_V) c(\Omega^1_V) 
 +  \frac{(-1)^{N+1}\delta}{12} \chi_{\mathrm{top}}(V)
 +  v_N^+ \vert\Sigma\vert,
 \end{multline}
 and:
 \begin{multline} \label{pas intro GK-XL lineaire}
\deg_{\PP^1}(\GK_{\PP^1,-}(\H^{N-1}(H_\eta / \PP^1_\eta) ))\\
= \frac{\delta}{12} \int_V (1 - c_1(M))^{-2} c_1(\Omega^1_V) c(\Omega^1_V)  
 +  \frac{(-1)^{N+1}\delta}{12} \chi_{\mathrm{top}}(V)
 + v_N^-  \vert\Sigma\vert,
\end{multline}
where: $$\chi_{\mathrm{top}}(V) = (-1)^N \int_V c_N (\Omega^1_V)$$ denotes the topological Euler characteristic of $V$, and where $v_N^+$ and $v_N^-$ are the rational numbers defined in Proposition \ref{technical result}.
\end{proposition}

We could have formulated \eqref{pas intro GK+XL lineaire} and \eqref{pas intro GK-XL lineaire} as equalities in the Chow group $\CH^1(\PP^1)_\Q$ of the base $\PP^1$ of the pencil, as in Proposition \ref{technical result}. However the degree map establishes an isomorphism $\CH^1(\PP^1)_\Q \simeq \Q$, and the present formulation is actually equivalent to this \emph{a priori} more precise one.

\subsection{Proof of Proposition \ref{linear pencils of hypersurfaces}}

For simplicity, let us denote $h$ the class $c_1(\cO_{\PP^1}(1))$ in $\CH^1(\PP^1)$ and $m$ the class $c_1(M)$ in $\CH^1(V)$.

We can apply equality \eqref{pas intro SigmaLX} from Proposition \ref{technical result} to the smooth projective complex scheme $X := V \times \PP^1$; the smooth surjective morphism of complex schemes
$$
\pi := \mathrm{pr}_2 : X \lra \PP^1
$$
and the line bundle on $X$:
$$
L := \cO_{V \times \PP^1}(H) \simeq  \mathrm{pr}_1^\ast M \otimes \mathrm{pr}_2^\ast \cO_{\PP^1}(\delta).
$$
We obtain the equality in $\CH_0(V \times \PP^1)$:
\begin{align}
[\Sigma] &= \big[ (1 - c_1(L) )^{-1} c(\Omega^1_{V \times \PP^1 / \PP^1}) \big]^{(N+1)} \nonumber \\ 
&= \big[ (1 - \mathrm{pr}_1^\ast m - \delta \, \mathrm{pr}_2^\ast h)^{-1}  \mathrm{pr}_1^\ast c(\Omega^1_V) \big]^{(N+1)}\nonumber \\ 
&= \mathrm{pr}_1^\ast \big[ (1 - m)^{-1} c(\Omega^1_V) \big]^{(N+1)} 
+ \delta \, \mathrm{pr}_2^\ast h  \, \mathrm{pr}_1^\ast \big[(1 - m)^{-2}  c(\Omega^1_V) \big]^{(N)}\label{first eq linear pencils}\\
&= \delta \, \mathrm{pr}_2^\ast h  \, \mathrm{pr}_1^\ast\big[ (1 - m)^{-2} . c(\Omega^1_V) \big]^{(N)}, \label{second eq linear pencils}
\end{align}
where in \eqref{first eq linear pencils}, we have used that $\CH^2(\PP^1)$ vanishes, and in \eqref{second eq linear pencils}, we have used that $\CH^{N+1}(V)$ vanishes.

This shows \eqref{sigma linear pencil}.

Now, let us assume that the line bundle $M$ on $V$ is ample. This implies that the line bundle $L$ on $X$ is ample relatively to the morphism $\pi$, so that the hypothesis of \eqref{pas intro GK+XL} and \eqref{pas intro GK-XL} is satisfied.

Since $X$ is the product of $\PP^1$ by $V$, all its relative Hodge vector bundles over $\PP^1$ are trivial, and for every integer $n$, the Griffiths line bundle $\GK_{\PP^1}(\H^n(X / \PP^1))$ is trivial. 

Consequently, taking degrees in equality \eqref{pas intro GK+XL} yields the following equality of rational numbers:
\begin{multline}
\deg_{\PP^1}(\GK_{\PP^1, +}(\H^{N-1}(H_\eta / \PP^1_\eta))) = \frac{1}{12} \int_{V \times \PP^1} (1 - c_1(L) )^{-1} c_1(\Omega^1_{V \times \PP^1/\PP^1}) c(\Omega^1_{V \times \PP^1/\PP^1})  \\
 - \frac{1}{12} \int_{V \times \PP^1} (c_1(L) c_N(\Omega^1_{V \times \PP^1/\PP^1}) ) 
+ v_N^+ |\Sigma|. \label{first eq GK linear pencils}
\end{multline}

Reasoning as above, the first term in \eqref{first eq GK linear pencils} can be rewritten:
\begin{align}
&\frac{1}{12} \int_{V \times \PP^1}  (1 - c_1(L) )^{-1} c_1(\Omega^1_{V \times \PP^1/\PP^1}) c(\Omega^1_{V \times \PP^1/\PP^1})  \nonumber \\
&= \frac{1}{12} \int_{V \times \PP^1}  (1 - \mathrm{pr}_1^\ast m - \delta \mathrm{pr}_2^\ast h)^{-1} \mathrm{pr}_1^\ast \big(c_1(\Omega^1_V) c(\Omega^1_V) \big)  \nonumber \\
&= \frac{1}{12} \int_{V \times \PP^1} \mathrm{pr}_1^\ast (1 - m)^{-1} \mathrm{pr}_1^\ast \big(c_1(\Omega^1_V) c(\Omega^1_V) \big)  \nonumber \\
&+ \frac{\delta}{12} \int_{V \times \PP^1}  \mathrm{pr}_2^\ast h \, \mathrm{pr}_1^\ast (1 - m)^{-2} \mathrm{pr}_1^\ast (c_1(\Omega^1_V) c(\Omega^1_V) ), \nonumber\\
&= \frac{\delta}{12} \int_V  (1 - m)^{-2} c_1(\Omega^1_V) c(\Omega^1_V) , \label{6561}
\end{align}
where we have used the fact that the group $\CH^{N+1}(V)$ vanishes, as well as the classical equality, for $\alpha$ a $0$-cycle in $V$:
$$
\int_{V \times \PP^1} \mathrm{pr}_1^\ast \alpha  \, \mathrm{pr}_2^\ast h = \int_V \alpha.
$$
Similarly, the second term in \eqref{first eq GK linear pencils} can be rewritten:
\begin{align}
\frac{1}{12} \int_{V \times \PP^1} c_1(L) c_N(\Omega^1_{V \times \PP^1/\PP^1}) &= \frac{1}{12} \int_{V \times \PP^1} (\mathrm{pr}_1^\ast m + \delta \mathrm{pr}_2^\ast h) \mathrm{pr}_1^\ast c_N(\Omega^1_{V}) , \nonumber\\
&= \frac{1}{12} \int_{V \times \PP^1} \mathrm{pr}_1^\ast (m \,  c_N(\Omega^1_V))   
+ \frac{\delta}{12} \int_{V \times \PP^1} \mathrm{pr}_2^\ast h \,  \mathrm{pr}_1^\ast c_N(\Omega^1_V) ,\nonumber \\
&= \frac{\delta}{12} \int_V c_N(\Omega^1_V), \nonumber\\
&= \frac{(-1)^N \delta}{12} \chi_{\mathrm{top}}(V). \label{6562}
\end{align}

Replacing \eqref{6561} and \eqref{6562} in \eqref{first eq GK linear pencils} yields the equality:
$$
\deg_{\PP^1}(\GK_{\PP^1, +}(\H^{N-1}(H_\eta / \PP^1_\eta)))
= \frac{\delta}{12} \int_V (1 - m)^{-2} c_1(\Omega^1_V) c(\Omega^1_V)  
 + \frac{(-1)^{N+1} \delta}{12} \chi_{\mathrm{top}}(V)
+ v_N^+ |\Sigma|,
$$
which shows equality \eqref{pas intro GK+XL lineaire}. Equality \eqref{pas intro GK-XL lineaire} follows similarly from equality \eqref{pas intro GK-XL}.

\subsection{Application to Lefschetz pencils}

Proposition \ref{linear pencils of hypersurfaces} applies notably to Lefschetz pencils. 

Let $V$ be a connected smooth projective complex scheme of pure dimension $N\geq 1$, embedded into some projective space $\PP^r$, of dimension $r\geq \max(N,2).$

Let $ \Lambda$ be a projective subspace of dimension $r-2$ in $\PP^r$ that intersects $V$ transversally, and let $P \subset \PP^{r\vee}$ the projective line in the dual projective space $\PP^{r\vee}$ corresponding to $\Lambda$ by projective duality.

Let us denote by: 
$$\nu: \widetilde{\PP}^r_\Lambda \lra \PP^r$$
the blow-up of $\Lambda$ in $\PP^r$. If $\widetilde{V}$ denote the proper transform of $V$ by $\nu$, the restriction:
$$\nu_{\mid \widetilde{V}} : \widetilde{V} \lra V$$
may be identified with the blow-up in $V$ of $\Lambda \cap V$, which is smooth of dimension $r-2$.

Let $I$ be the incidence subscheme in $\PP^r \times \PP^{r \vee}$. It is an hypersurface in $\PP^r \times \PP^{r \vee}$, and both projections:
$$\mathrm{pr}_{1 \mid I} : I \lra \PP^r \quad \mbox{and} \quad \mathrm{pr}_{2 \mid I} : I \lra \PP^{r\vee}$$
are smooth morphisms. Moreover the line bundle $\cO(I)$ over $\PP^r \times {\PP}^{r\vee}$ is isomorphic to $\mathrm{pr}_1^\ast \cO_{\PP^r}(1) \otimes \mathrm{pr}_2^\ast \cO_{\PP^{r\vee}}(1).$

The projection of center $\Lambda$:
$$\PP^r - \Lambda \lra P$$
extends to a smooth morphism of complex schemes:
$$p: \widetilde{\PP}^r_\Lambda \lra P,$$
and the pair of morphisms $(\nu, p)$ establishes an isomorphism:
$$(\nu, p): \widetilde{\PP}^r_\Lambda \lrasim I \cap (\PP^r \times P)$$
between the scheme $\widetilde{\PP}^r_\Lambda$ and the smooth hypersurface $I \cap (\PP^r \times P)$ of $\PP^r \times P$.

Consequently, by restriction, the pair of morphisms $(\nu_{\mid \widetilde{V}}, p_{\mid \widetilde{V}})$ defines an isomorphism:
$$(\nu_{\mid \widetilde{V}}, p_{\mid \widetilde{V}}): \widetilde{V} \lrasim I \cap (V \times P)$$
between $\widetilde{V}$ and the smooth hypersurface $I \cap (V\times P)$ of $V \times P$. Moreover the line bundle $\cO\big(I \cap (V \times P)\big)$ over $V\times P$  is isomorphic to $\mathrm{pr}_1^\ast \cO_{\PP^r}(1)_{\mid V} \otimes \mathrm{pr}_2^\ast \cO_P(1).$ 

Recall that the pencil of hyperplanes in $\PP^r$ containing $\Lambda$ is said to be a Lefschetz pencil with respect to the subvariety $V$ of $\PP^r$ when the morphism:
$$p_{\mid \widetilde{V}} : \widetilde{V} \lra P$$
 has a finite set $\Sigma$ of critical points, all of which are non-degenerate, and when the restriction
$$
p_{| \Sigma} : \Sigma \lra P
$$
is an injective map.

We may apply Proposition \ref{linear pencils of hypersurfaces} to the hypersurface $I \cap (V \times P)$ in $V\times P$. With the notation of this proposition in this situation, the line bundle $M$ is (isomorphic to) the restriction $\cO_V(1)$ of $\cO_{\PP^r}(1)$ to $V$, and $\delta = 1$. Accordingly, we obtain the following result, which notably applies to Lefschetz pencils:
 
\begin{corollary}
\label{Lefschetz pencils}
With the above notation, let us assume that the morphism:
$$p_{\mid \widetilde{V}} : \widetilde{V} \lra P$$
 has a finite set $\Sigma$ of critical points, all of which are non-degenerate. 

The following equality holds in $\CH_0(V \times P)$:
$$
(\nu, p)_\ast [\Sigma] = \mathrm{pr}_1^\ast [(1 - c_1(\cO_V(1)))^{-2} c(\Omega^1_V) ]^{(N)} \,  \mathrm{pr}_2^\ast c_1(\cO_P(1)).
$$
In particular, the cardinality of $\Sigma$ satisfies:
\begin{equation}\label{pas intro KatzSGA}
\vert\Sigma\vert = \int_V (1 - c_1(\cO_V(1)))^{-2} c(\Omega^1_V).
\end{equation}

Furthermore, the following equalities of integers hold:
\begin{multline*}
\deg_P(\GK_{P,+}(\H^{N-1}(\widetilde{V}_\eta / P_\eta) )) \\
= \frac{1}{12} \int_V (1 - c_1(\cO_V(1)))^{-2} c_1(\Omega^1_V)c(\Omega^1_V) 
+  \frac{(-1)^{N+1}}{12} \chi_{\mathrm{top}}(V) 
+ v_N^+\vert\Sigma\vert,
\end{multline*}
and:
\begin{multline*}
\deg_P(\GK_{P,-}(\H^{N-1}(\widetilde{V}_\eta / P_\eta) )) \\
= \frac{1}{12} \int_V (1 - c_1(\cO_V(1)))^{-2} c_1(\Omega^1_V)c(\Omega^1_V)  
+  \frac{(-1)^{N+1}}{12} \chi_{\mathrm{top}}(V) 
+ v_N^- \vert\Sigma\vert.
\end{multline*}
\end{corollary}

As already mentioned in \ref{Lefschetz pencils Intro}, the expression \eqref{pas intro KatzSGA} for the number of critical points in a Lefschetz pencil is established by Katz in \cite[Expos\'{e} XVII, cor. 5.6]{SGA7II}.

\end{document}